\documentclass[11pt]{article}

\usepackage{amsmath, amsthm}
\usepackage{amsmath, amsfonts}
\usepackage{amsmath, amssymb}
\usepackage{amsmath}
\usepackage[all]{xy}

\newtheorem{theorem}{Theorem}[subsection]
\newtheorem{definition}[theorem]{Definition}
\newtheorem{definition-lemma}[theorem]{Definition/Lemma}
\newtheorem{definition-explanation}[theorem]{Definition/Explanation}
\newtheorem{explanation-definition}[theorem]{Explanation/Definition}
\newtheorem{definition-fact}[theorem]{Definition/Fact}
\newtheorem{definition-notation}[theorem]{Definition/Notation}
\newtheorem{lemma}[theorem]{Lemma}
\newtheorem{lemma-definition}[theorem]{Lemma/Definition}
\newtheorem{proposition}[theorem]{Proposition}
\newtheorem{corollary}[theorem]{Corollary}
\newtheorem{remark}[theorem]{\it Remark}

\newtheorem{example}[theorem]{Example}
\newtheorem{example-definition}[theorem]{Example/Definition}
\newtheorem{notation}[theorem]{Notation}

\newtheorem{definition-prototype}[theorem]{Definition-Prototype}

\newtheorem{statement}[theorem]{Statement}

\numberwithin{equation}{subsection}

\newtheorem{stheorem}{Theorem}[section]
\newtheorem{sdefinition}[stheorem]{Definition}
\newtheorem{sdefinition-lemma}[stheorem]{Definition/Lemma}
\newtheorem{sdefinition-explanation}[stheorem]{Definition/Explanation}
\newtheorem{sexplanation-definition}[stheorem]{Explanation/Definition}
\newtheorem{sdefinition-fact}[stheorem]{Definition/Fact}
\newtheorem{sdefinition-notation}[theorem]{Definition/Notation}

\newtheorem{slemma-definition}[stheorem]{Lemma/Definition}

\newtheorem{sexample-definition}[stheorem]{Example/Definition}

\newtheorem{sdefinition-prototype}[stheorem]{Definition-Prototype}

\newtheorem{ssdefinition-lemma}[sstheorem]{Definition/Lemma}
\newtheorem{ssdefinition-explanation}[sstheorem]{Definition/Explanation}
\newtheorem{ssexplanation-definition}[sstheorem]{Explanation/Definition}
\newtheorem{ssdefinition-fact}[stheorem]{Definition/Fact}
\newtheorem{ssdefinition-notation}[theorem]{Definition/Notation}

\newtheorem{sslemma-definition}[sstheorem]{Lemma/Definition}

\newtheorem{ssexample-definition}[sstheorem]{Example/Definition}

\newtheorem{ssdefinition-prototype}[sstheorem]{Definition-Prototype}

\newcommand{\Gr}{\mbox{\it Gr}\,}

\newcommand{\Hess}{\mbox{\it Hess}}

\newcommand{\Imaginary}{\mbox{\it Im}\,}
\newcommand{\Image}{\mbox{\it Im}\,}

\newcommand{\Ker}{\mbox{\it Ker}\,}

\newcommand{\Mor}{\mbox{\it Mor}\,}

\newcommand{\Real}{\mbox{\it Re}\,}

\newcommand{\distance}{\mbox{\it distance}\,}

\newcommand{\id}{\mbox{\it id}\,}

\newcommand{\pr}{\mbox{\it pr}}

\newcommand{\pt}{\mbox{\it pt}}

\newcommand{\vol}{\mbox{\it vol}\,}

\begin{document}

\enlargethispage{24cm}

\begin{titlepage}

$ $

\vspace{-1.5cm} 

\noindent\hspace{-1cm}
\parbox{6cm}{\small June 2011}\
   \hspace{7cm}\
   \parbox[t]{5cm}{yymm.nnnn [math.DG] \\
                   D(8.1): A-brane, sL $S^3$,\\ $\mbox{\hspace{3.3em}}$
                           L-unwrapping}

\vspace{2cm}

\centerline{\large\bf
 A natural family of immersed Lagrangian deformations of}
\vspace{1ex}
\centerline{\large\bf
 a branched covering of a special Lagrangian $3$-sphere
 in a Calabi-Yau $3$-fold}
\vspace{1ex}
\centerline{\large\bf
 and its deviation from Joyce's criteria:}
\vspace{1ex}
\centerline{\large\bf
 Potential image-support rigidity of A-branes that wrap around a sL $S^3$}


\vspace{3em}

\centerline{\large
  Chien-Hao Liu
  \hspace{1ex} and \hspace{1ex}
  Shing-Tung Yau
}

\vspace{6em}

\begin{quotation}
\centerline{\bf Abstract}

\vspace{0.3cm}

\baselineskip 12pt  
{\small
  Using
   a hyperK\"{a}hler rotation on complex structures
    of a Calabi-Yau $2$-fold  and
   rolling of an isotropic $2$-submanifold
    in a symplectic $6$-manifold,
  we construct, by gluing,
   a natural family of immersed Lagrangian deformations
   of a branched covering of a special Lagrangian $3$-sphere
   in a Calabi-Yau $3$-fold and
  study how they deviate from being deformable to a family of
   special Lagrangian deformations by examining in detail
   Joyce's criteria on this family.
  The result suggests a potential image-support rigidity of A-branes
   that wrap around a special Lagrangian $3$-sphere
   in a Calabi-Yau $3$-fold,
  which resembles a similar phenomenon for holomorphic curves
   that wrap around a rigid smooth rational curve
   in a Calabi-Yau $3$-fold in Gromov-Witten theory.
} 
\end{quotation}

\vspace{15em}

\baselineskip 12pt
{\footnotesize
} 

\bigskip

\noindent {\small MSC number 2010:
 53C38, 53C42, 57M12; 81T30, 81T75.
} 

\bigskip

\baselineskip 10pt
{\scriptsize
\noindent{\bf Acknowledgements.}
We thank
 Frederik Denef, Andrew Strominger, and Cumrun Vafa
  for lectures that influence our understanding of strings, branes,
  and gravity.
C.-H.L.\ thanks in addition
 Adrian Butscher
  for communicating his thesis with generous remarks;
 A.B., Dominic Joyce, Yng-Ing Lee, and Sema Salur
  for their work that influences his understanding
  of themes behind/beyond the current note;
 Clay Cordova
  for a discussion on the notion of pure D-branes;
 Liang Kong
  for communicating his thought on conformal field theory and D-branes;
 Erel Levine, Jacob Lurie, Andrew Strominger, and Horng-Tzer Yau
  for basic/topic courses, spring 2011; and
 Ling-Miao Chou for moral support.
 The project is supported by NSF grants DMS-9803347 and DMS-0074329.
} 

\end{titlepage}

\newpage
\begin{titlepage}

$ $

\vspace{12em} 

\centerline{\small\it
 Chien-Hao Liu dedicates this note to professors (time-ordered)}
\centerline{\small\it
 Hai-Chau Chang$^{\ast}$, Su-Win Yang, Ai-Nung Wang,}
\centerline{\small\it
 who introduced him the beauty of modern geometry and topology}
\centerline{\small\it
 in his undergraduate years.}

\vspace{36em}

\baselineskip 11pt

{\footnotesize
\noindent
$^{\ast}$(From C.-H.L.)
 Special tribute to {\it Prof.\ Chang},
  a mentor and a friend throughout my college years
 who has changed the course of my life.
} 

\end{titlepage}

\newpage
$ $

\vspace{-3em}

\centerline{\sc
 Immersed Lagrangian Deformation of Branched Covering of SL $S^3$}

\vspace{2em}


\begin{flushleft}
{\Large\bf 0. Introduction and outline.}
\end{flushleft}
In the geometric phase of the Wilson's theory-space
 for a boundary conformal field theory with weak D-brane tension,
D-branes are realized in part
 as morphisms from Azumaya noncommutative spaces
   with a fundamental module (with a connection)
   to string-theory target-space(-time).
For physical A-branes in the supersymmetric case,
 the connection is required to be flat, possibly with singularities,
  and the associated maps from the surrogates are required
  to be special Lagrangian morphisms.
See [L-Y2: Sec.~2.1] for more explanations and references.

In the special case when D-branes wraps a special Lagrangian $3$-sphere
 in a Calabi-Yau $3$-fold,
it follows from Robert McLean [McL] that the latter is rigid
 for a topological reason  and
it is natural to ask
 whether through such wrapping
  one can deform the image D-brane away
  from the given special Lagrangian $3$-sphere.
The answer would influence, for example,
 the details of the quiver gauge field theory associated to a collection
  of special Lagrangian 3-spheres that intersect transversely
  in a Calabi-Yau 3-fold  and
 the multiple cover formula for D3-branes (or Euclidean D2-branes)
  that wrap a special Lagrangian $3$-sphere.

Before we can address how to deal with this problem,
we devote first this note to analyzing in detail
 how some existing constructions and techniques could fail
 in the current situation.

This is a sequel to [L-Y1] (D(6)) and [L-Y2] (D(7)).
In the first part of this note (Sec.~1 -- Sec.~3),
we take a fundamental existence theorem of Dominic Joyce
  ([Jo3: III.~Theorem~5.3]) as the starting point
  (cf.\ Sec.~1),
 construct a natural family $f^t:N^t\rightarrow Y$
  of smooth immersed Lagrangian submanifolds
  -- in a similar spirit as is done in [Sa1] of Sema Salur
     for resolving a codimension-$2$ singularity of
     a singular special Lagrangian submanifold in a Calabi-Yau $3$-fold
     --
  that has the special Lagrangian branched covering
   $f:X\rightarrow Y$ to begin with
   as its limit in $C^{\infty}$-topology
   as well as in the sense of current when $t\rightarrow 0$
  (cf.\ Sec.~3.1),  and
 check how
   Joyce's criteria of deformability
    to special Lagrangian submanifolds  and
   standard techniques to justify them
  behave on the family $\{f^t\}_t$
  (cf.\ Sec.~3.2 -- Sec.~3.4).
Some necessary background and ingredients for the study are given
 in Sec.~2 and the beginning of Sec.~3.2; and
a summary on the deviation of $\{f_t\}_t$ from Joyce's criteria
 is given in Sec.~4.
The investigation reveals a potential image-support rigidity
 of A-branes that wrap around a special Lagrangian $3$-sphere
 in a Calabi-Yau $3$-fold.
This resembles a similar phenomenon for holomorphic curves
  that wrap around a rigid smooth rational curve
  in a Calabi-Yau $3$-fold in Gromov-Witten theory.

Finally, we should remark that,
  instead of taking Joyce's Existence Theorem as the starting point,
 one can also proceed to understand the problem
  from a direct approach through a route following
   Adrian Butscher ([Bu1], [Bu2]), Yng-Ing Lee [Lee], and
   Semar Salur [Sa1].

\bigskip

\noindent
{\bf Convention.}
 Standard notations, terminology, operations, facts\footnote{Cf.\
           [L-Y2: footnote~2] (D(7)): {\it Apology}.}
  in
  (1) Riemannian, spectral/hyperK\"{a}hler geometry;
  (2) analysis on Riemannian manifolds;
  (3) symplectic/calibrated geometry;
  (4) branched coverings
 can be found respectively in$\,$
  (1) [G-H-L], [S-Y]$\,/\,$[Jo2];$\,$
  (2) [Au];$\,$
  (3) [McD-S], [G-S]$\,/\,$[H-L], [Ha], [McL];$\,$
  (4) [Ro].
 \begin{itemize}
  \item[$\cdot$]
   `{\it $n$-(sub)manifold}'
     for {\it real} (sub)manifold of (real) dimension $n$
    vs.\ `{\it $n$-fold}'
     for {\it complex} manifold of (complex) dimension $n$.

  \item[$\cdot$]
   `{\it Branch locus}' $\Gamma$ of a map
    vs.\ `{\it graph}' $\Gamma(\alpha)$ of a $1$-form $\alpha$
    vs.\ the space $\Gamma(\,\cdot\,)$ of {\it sections} of a bundle
    vs.\ the `{\it Christoffel symbols}' $\Gamma^i_{jk}$.

  \item[$\cdot$]
   `{\it Connection}' {\boldmath $\omega$}
    and `{\it curvature}' {\boldmath $\Omega$}
    vs.\ `{\it K\"{a}hler/symplectic structure}' $\omega$
         and `{\it holomorphic $n$-form}' $\Omega$.
   The latter
    vs.\ the space $\Omega^k(N)$ of `{\it $k$-forms}' on a manifold $N$.

  \item[$\cdot$]
   `{\it Distribution}' in the sense of generalized functions
    vs.\ `{\it distribution}' in the sense of a subbundle
     of a tangent bundle or its restriction to a submanifold.

  \item[$\cdot$]
   The various constants $C,\,D,\,\cdots$ that appear in an estimate
    are unspecified constants that, in general, may of different values
    at difference places.

  \item[$\cdot$]
   The various built-in {\it projection maps} of
    bundles to their base are denoted by $\pi$.

  \item[$\cdot$]
   The Laplacian $\Delta$
   and its eigenvalues $\lambda_k$ for a Riemannian manifold $(M,g)$.
   When $\partial M\ne \emptyset$,
    this is referred to the related Dirichlet problem
    with vanishing boundary value.

  \item[$\cdot$]
   Notations related to the various constructions
    follow their counterpart in [Jo3: III and IV] as much as we can.
   Proofs that follow essentially the same argument
    as in their counterpart in ibidem will be
    either omitted
    or given only a sketch for spelling out the modified part.

  \item[$\cdot$]
   A partial review of D-branes and Azumaya noncommutative geometry
    is given in [L-Y1] (D(6)).
   The current work addresses [L-Y2: Sec.~2.3, Question 2.3.9] (D(7)).
 \end{itemize}

\bigskip

\bigskip

\begin{flushleft}
{\bf Outline.}
\end{flushleft}
{\small
\baselineskip 12pt  
\begin{itemize}
 \item[0.]
  Introduction.

 \item[1.]
  Joyce's Existence Theorem on immersed special Lagrangian submanifolds.

 \item[2.]
  Basic Riemannian, complex, and symplectic ingredients.
  \vspace{-.6ex}
  \begin{itemize}
   \item[2.1]
    A lower bound of the first eigenvalue of the Laplacian
    on a Riemannian manifold with tame curvature singularity.

   \item[2.2]
    A class of embedded special Lagrangian submanifolds in the flat
    Calabi-Yau 3-fold\\ ${\Bbb R}^2\times S^1\times{\Bbb R}^3$.

   \item[2.3]
    Lagrangian Neighborhood Theorems.

   \item[2.4]
    Admissible Lagrangian neighborhoods for $L^{\prime,a}$
    in $Y^{\prime}$ under $\psi^a$ and
    their geometry under partial scaling.
  \end{itemize}

 \item[3.]
  Immersed Lagrangian deformations of a simple normalized branched
   covering of a special Lagrangian 3-sphere in a Calabi-Yau 3-fold and
  their deviation from Joyce's criteria.
  \vspace{-.6ex}
  \begin{itemize}
   \item[3.1]
   Immersed Lagrangian deformations of a branched covering
    of a special Lagrangian 3-sphere in a Calabi-Yau 3-fold.

   \item[3.2]
    Estimating {\it Im}$\,\Omega^s|_{N^t}$.

   \item[3.3]
    Lagrangian neighborhoods and bounds on $R(g^t)$, $\delta(g^t)$.

   \item[3.4]
    Sobolev immersion inequalities on $N^t$.

  \end{itemize}

 %
 %
 %
 %

 \item[4.]
  Summary and remark:
  Input from the topology of $X$ and the branching of $f$.
  %
  %
  %
\end{itemize}
} 

\newpage

\section{Joyce's Existence Theorem on immersed special Lagrangian
         submanifolds.}

For the basic definitions, notations, and terminology
 to be used in this work,
we recall
  Joyce's Existence Theorem on deforming
  an immersed almost special Lagrangian submanifold to
  an immersed special Lagrangian submanifold,
  ([Jo3: III.\ Sec.~5.1 and Sec.~5.2])\footnote{While
                         it is theorems in the embedded case
                          that are stated in Joyce's work,
                         they generalize immediately to the immersed case.}
 with only
  mild adaptation
   from the general almost Calabi-Yau case to the Calabi-Yau case and
  mild change of notations for consistency with later part of the work.

\bigskip

Let $(N,g)$ be a Riemannian manifold,
 with injective radius $\delta(g)$ and volume form $dV_g$.
Denote the various Banach spaces of functions on $N$ as follows:
\begin{itemize}
 \item[$\cdot$]
  $C^k(N)$, $k\ge 0\,$:
   the {\it Banach space} of continuous bounded functions on $N$
   that have $k$ continuous bounded derivatives;
  $\;\|f\|_{C^k}:=\sum_{j=0}^k\sup_N|\nabla^jf|$.
  $\;C^{\infty}(N):= \bigcap_{k\ge 0}\,C^k(N)$.

 \item[$\cdot$]
  $C^{k,\alpha}$, $k\ge 0$, $\alpha\in (0,1)\,$:
   the {\it H\"{o}lder space} of elements $f \in C^k(N)$ for which
    $|\nabla^kf|_{\alpha}
     :=\sup_{x\ne y, dist(x,y)<\delta(g)}
          \frac{|\nabla^k f(x)-\nabla^k f(y)|}{d(x,y)^{\alpha}}$
    is finite;
  $\;\|f\|_{C^{k,\alpha}}:= \|f\|_{C^k}+|\nabla^kf|_{\alpha}$.

 \item[$\cdot$]
  $L^q(N)$, $q\ge 1\,$: the {\it Lebesque space}
   of locally integrable functions $f$ on $N$ for which the norm
   $\|f\|_{L^q}:=\left(\int_N|f|^q dV_g\right)^{1/q}$ is finite.

 \item[$\cdot$]
  $L^q_k(N)$, $q\ge 1$, $k\ge 0\,$:
   the {\it Sobolev space} of elements $f\in L^q(N)$ such that
    $f$ is $k$ times weakly differentiable and
    $|\nabla^jf|\in L^q(N)$ for $j\le k\,$;
  $\;\|f\|_{L^q_k}:=\left(\sum_{j=0}^k\int_N|\nabla^jf|^q\right)^{1/q}$.
\end{itemize}

\begin{stheorem}
{\bf [Sobolev embedding].}
{\rm ([Au: Theorem~2.30], [Jo3: III.~Theorem~5.1].)}
 Suppose
  $(N,g)$ is a compact Riemannian $n$-manifold,
  $k\ge l\ge 0$ are integers,
  $\alpha\in (0,1)$, and
  $q,\,r\ge 1$.
 If $\frac{1}{q}\le \frac{1}{r}+\frac{k-l}{n}$,
  then $L^q_k(N)$ is continuously embedded in $L^r_l(N)$ by inclusion.
 If $\frac{1}{q}\le \frac{k-l-\alpha}{n}$,
  then $L^q_k(N)$ is continuously embedded in $C^{l,\alpha}(N)$
   by inclusion.
\end{stheorem}

\begin{sdefinition}
{\bf [basic setup].} {\rm ([Jo3: III.~Definition~5.2].)}
{\rm
 Let $(M,J,\omega,\Omega)$ be a Calabi-Yau $m$-fold with metric $g_M$.
 Let $N$ be a compact, oriented, immersed, Lagrangian $m$-submanifold
  in $M$, with immersion $\iota:N\rightarrow M$,
  so that $\iota^{\ast}\omega\equiv 0$.
 Define $g:= \iota^{\ast}g_M$, so that $(N,g)$ is a Riemannian manifold.
 \begin{itemize}
  \item[(1)] {\it Phase function $e^{i\theta}$.} \hspace{1em}
   Let $dV:=dV_g$ be the volume form on $N$
    induced by the metric $g$ and orientation.
   Then $|\iota^{\ast}\Omega|\equiv 1$,
    calculating $|\,\cdot\,|$ using $g$ on $N$.
   Therefore, we may write
    $$
     \iota^{\ast}\Omega\;=\; e^{i\theta}\,dV
      \hspace{1em}\mbox{on $\,N$}\,,
    $$
    for some phase function $e^{i\theta}$ on $N$.
   Suppose that $\cos\theta\ge \frac{1}{2}$ on $N$.
   Then we can choose $\theta$ to be a smooth function
    $\theta:N\rightarrow (-\frac{\pi}{3},\frac{\pi}{3})$.
   Suppose that $[\iota^{\ast}(\Imaginary\Omega)]=0$ in $H^m(N;{\Bbb R})$.
   Then $\int_N\,\sin\theta\,dV=0$.

  \item[(2)] {\it Vector space $W$.} \hspace{1em}
   Let $W\subset C^{\infty}(N)$
    be a given a finite-dimensional vector space with $1\in W$.
   Define $\pi_W:L^2(N)\rightarrow W$ be the projection onto $W$
    using the $L^2$-inner product.

  \item[(3)]
   {\it Neighborhood ${\cal B}_r$ of the zero-section in $T^{\ast}N$.}
    \hspace{1em}
   For $r>0$, define ${\cal B}_r\subset T^{\ast}N$ to be the bundle of
    $1$-forms $\alpha$ on $N$ with $|\alpha|<r$.
   ${\cal B}_r$ is a noncompact $2m$-manifold with natural projection
    $\pi:{\cal B}_r\rightarrow N$, whose fiber at $x\in N$
    is the ball of radius $r$ about $0$ in $T^{\ast}_xN$.
   We identify $N$ also with the zero-section of ${\cal B}_r$
    and write $N\subset {\cal B}_r$.

  \item[(4)] {\it Geometry on ${\cal B}_r$.}
    \hspace{1em}
   At each $y\in{\cal B}_r$ with $\pi(y)=x\in N$,
    the Levi-Civita connection $\nabla$ of $g$ on $T^{\ast}N$
    defines a splitting $\,T_y{\cal B}_r\,=\, H_y\oplus V_y\,$
    into horizontal and vertical subspaces $H_y,\, V_y\,$,
    with $H_y\simeq T_xN$ and $V_y\simeq T^{\ast}_xN$.
   Let $\hat{\omega}=\omega_{can}$
    for the canonical symplectic structure on ${\cal B}_r\subset T^{\ast}N$,
    defined using $T{\cal B}_r=H\oplus V$ and $H\simeq V^{\ast}$.
   Define a natural Riemannian metric $\hat{g}$ on ${\cal B}_r$
    such that the subbundles $H,\, V$ are orthogonal, and
    $\,\hat{g}|_H\,=\, \pi^{\ast}(g)\,$,
    $\,\hat{g}|_V\,=\, \pi^{\ast}(g^{-1})\,$.

   \hspace{5.0mm}Let $\hat{\nabla}$ be the connection on
    $T_{\ast}{\cal B}_r\simeq H\oplus V$ given
     by the lift of the Levi-Civita connection $\nabla$ of $g$ on $N$
      in the horizontal directions $H$,  and
     by partial differentiation in the vertical directions $V$,
      which is well-defined as $T_{\ast}{\cal B}_r$ is naturally trivial
      along each fiber.\footnote{Recall
                 the projection map $\pi:{\cal B}_r\rightarrow N$.
                 The Levi-Civita $\nabla$ on $T^{\ast}N$
                 induces a bundle inclusion
                 $\pi^{\ast}(T_{\ast}N)\hookrightarrow T_{\ast}{\cal B}_r$
                 with image $H$.
                The pull-back partial connection on $\pi^{\ast}(T_{\ast}N)$
                 gives then a partial connection $\hat{\nabla}_H$ on $H$
                 along the horizontal distribution on ${\cal B}_r$
                 associated to $\nabla$.
                $V=\Ker( \pi_{\ast}:
                         T_{\ast}{\cal B}_r\rightarrow T_{\ast}N )$
                 with a built-in flat partial connection $\hat{\nabla}_V$
                 along fibers of $\pi$.
                As smooth vector bundles, $T_{\ast}{\cal B}_r=H\oplus V$
                 and the direct sum $\hat{\nabla}_H\oplus\hat{\nabla}_V$
                 of partial connections on the direct summands gives
                 the connection $\hat{\nabla}$ on $T_{\ast}{\cal B}_r$.}
    Then $\hat{\nabla}$ preserves $\hat{g}$, $\hat{\omega}$, and
     the splitting $T_{\ast}{\cal B}_r=H\oplus V$.
    It is not torsion-free in general, but has torsion $T(\hat{\nabla})$
     depending linearly on the Riemann curvature $R(g)$.

   \hspace{5.0mm}For convenience and with a slight abuse of terminology,
    we will call $\hat{\nabla}$
    the {\it pull-back connection} in the bundle $T_{\ast}{\cal B}_r$
     (or $T_{\ast}(T^{\ast}N)$)
    of the Levi-Civita connection $\nabla$ in $T_{\ast}N$
    via $\pi:{\cal B}_r\rightarrow N$.
    (Cf.\ Sec.~2.4.)

  \item[(5)] {\it ${\cal B}_r$ as immersed Lagrangian neighborhood
                  and $m$-form $\beta$ on ${\cal B}_r$} \hspace{1em}
   Since $\iota$ is an immersed Lagrangian submanifold,
   it follows from the Immersed Lagrangian Neighborhood Theorem that
    for some small $r>0$, there exists an immersion
    $\,\Phi:{\cal B}_r\rightarrow M\,$
    such that $\Phi^{\ast}\omega=\hat{\omega}$ and $\,\Phi|_N=\iota\,$.
   Define an $m$-form $\beta$ on ${\cal B}_r$
    by $\,\beta=\Phi^{\ast}(\Imaginary\Omega)\,$.

  \item[(6)] {\it Graph of 1-forms.} \hspace{1em}
   If $\alpha\in C^{\infty}(T^{\ast}N)$ with $|\alpha|<r$,
    write $\Gamma(\alpha)$ for the {\it graph} of $\alpha$
     in ${\cal B}_r$.
   Then $\Phi_{\ast}: \Gamma(\alpha)\rightarrow M$
    is a compact immersed submanifold in $M$ homotopic to
    $\iota\circ\pi|_{\Gamma(\alpha)}$.
 \end{itemize}
}\end{sdefinition}

\begin{stheorem}
{\bf [Joyce: from almost sL to sL].} {\rm ([Jo3: III.~Theorem~5.3].)}
 With the above notations with $m\ge 3$,
 let
  $\kappa>1$ and
  $A_1,\, A_2,\, A_4,\, A_5,\, A_6,\, A_7,\, A_8$ be real.\footnote{For
               an almost Calabi-Yau $m$-fold $(M, J,\omega,\Omega)$,
               $\,\psi^{2m}\omega^m/m!\,=\,
                (-1)^{m(m-1)/2}\,(i/2)^m\,\Omega\wedge\bar{\Omega}\,$
               for a unique smooth function $\psi:M\rightarrow (0,\infty)$.
              In our case,
               $\psi\equiv 1$, $A_3=1$, and
               Condition (ii) of [Jo3: III~Theorem~5.3],
                 which states that $\psi\ge A_3$ on $N$,
                holds automatically.
              As this work is based upon Joyce's work,
               we maintain the original notation and
                labelling of conditions in Joyce's Theorem.}
 Then
  there exist $\epsilon,\,K >0$ depending only on
   $\kappa,\, A_1,\, \cdots\,,\, A_8$ and $m$
  such that the following holds:

 Suppose $0<t\le \epsilon$ and Definition~0.2 holds with $r=A_1t$, and
  \begin{itemize}
  \item[]
  \begin{itemize}
   \item[{\rm (i)}]
    $\|sin\theta\|_{L^{2m/(m+2)}}\,\le\,A_2\,t^{\kappa+m/2}\,$,
    $\;\|\sin\theta\|_{C^0}\,\le\,A_2\,t^{\kappa-1}\,$,
    $\;\|d(\sin\theta)\|_{L^{2m}}\,\le\, A_2\,t^{\kappa-3/2}\,$,\\ and
    $\;\|\pi_W(\sin\theta)\|_{L^1}\,\le\,A_2\,t^{\kappa+m-1}\,$.

   \item[{\rm (iii)}]
    $\|\hat{\nabla}^k\beta\|_{C^0}\,\le\, A_4\,t^{-k}\,$
     for $k=0,\,1,\,2,\,$ and $3$.

   \item[{\rm (iv)}]
    The injective radius $\delta(g)$ satisfies
    $\,\|\delta(g)\|\,\ge\,A_5\,t\,$.

   \item[{\rm (v)}]
    The Riemann curvature $R(g)$ satisfies
     $\,\|R(g)\|_{C^0}\,\le\,A_6\,t^{-2}\,$.

   \item[{\rm (vi)}]
    If $\,v \in L^2_1(N)\,$ with $\,\pi_W(v)=0\,$,
    then
     $\,v\in L^{2m/(m-2)}(N)\,$ by Theorem~0.1, and\\
     $\,\|v\|_{L^{2m/(m-2)}}\,\le\,\,A_7\,\|dv\|_{L^2}\,$.

   \item[{\rm (vii)}]
    For all $w\in W$, we have
     $\,\|d^{\ast}dw\|_{L^{2m/(m+2)}}\,
                   \le\, \frac{1}{2}\,A_7^{-1}\,\|dw\|_{L^2}\,$.\\
    For all $w\in W$ with $\,\int_N\,w\,dV=0\,$, we have
     $\,\|w\|_{C^0}\,\le\,A_8\,t^{1-m/2}\,\|dw\|_{L^2}\,$.
  \end{itemize}
  \end{itemize}
 Here
  norms are computed using
   the metric $g$ on $N$ in {\rm (i), (v), (vi)} and {\rm (vii)}, and
   the metric $\hat{g}$ on ${\cal B}_{A_1t}$ in {\rm (iii)}.
 Then there exists $f\in C^{\infty}(N)$ with $\int_N\,f\,dV=0$
  such that
   $\,\|df\|_{C^0}\,\le\, Kt^{\kappa}\,<\,A_1t\,$  and
   $\Phi_{\ast}\,:\,\tilde{N}:= \Gamma(df)\,\rightarrow\, M$
    is an immersed special Lagrangian $m$-manifold
    in $(M,J,\omega,\Omega)$.
\end{stheorem}

\noindent
See [Jo3: IV, last paragraph of Sec.~5.1] for the reasons
 behind the design of the theorem.

\bigskip

\section{Basic Riemannian, complex, and symplectic ingredients.}

Preliminary ingredients
  that are needed to apply Joyce's Theorem in our situation
 are collected in this section.
They follow from standard techniques
 in Riemannian, complex, and symplectic geometry.

\bigskip

\subsection{A lower bound of the first eigenvalue of the Laplacian
  on a Riemannian manifold with tame curvature singularity.}

\begin{definition}
{\bf [Riemannian metric with tame curvature singularity].} {\rm
 Given a closed smooth manifold $X$,
  let $g^{\prime}$ be a Riemannian metric
   defined on a dense open submanifold $X^{\prime}$ of $X$
   with the property that
   \begin{itemize}
    \item[(1)]
     $g^{\prime}$ does not extend to a Riemannian metric on $X$,

    \item[(2)]
     there exist a Riemannian metric $g$ on $X$ and a constant $c>1$
      such that
      $$
       \frac{1}{c^2}\,g\;\le\; g^{\prime}\;\le\; c^2\,g
       \hspace{2em}\mbox{on $\;X^{\prime}$}\,.
      $$
   \end{itemize}
 Let $Z:= X-X^{\prime}$.
 We say that
  $g^{\prime}$ is a
   {\it Riemannian metric on $X$ with tame curvature singularity}
   supported on $Z$.
 In other words, while $g^{\prime}$ doesn't extend to $X$,
  it admits a quasi-conformal deformation with uniformly bounded dilatation
  that extends to a Riemannian metric on $X$.
 We call $(X,g^{\prime})$
  a closed {\it Riemannian manifold with tame curvature singularity}.
}\end{definition}

The following lemma is the counterpart of [Jo3: I.~Theorem~2.17]
 in our situation.

\begin{lemma}
{\bf [lower bound for first eigenvalue $\lambda_1$ of Laplacian].}
 Let
  $(X,g^{\prime})$
   be a (connected) closed Riemannian $m$-manifold
   with tame curvature singularity  and
  $X^{\prime}$ be a (connected) dense open submanifold of $X$
   on which $g^{\prime}$ is defined  and
   that there exists an exhausting sequence
    $\, X^{\prime}_1\,\subset\, X^{\prime}_2\,\subset\,
      \cdots\cdots\,\subset\, X^{\prime}$  of (connected)
    embedded compact submanifolds-with-smooth-boundary of dimension $m$
    with $\,\bigcup_{i=1}^{\infty}\,X^{\prime}_i\; =\; X^{\prime}\,$.
 Then
  there exists a constant $C>0$ such that
   whenever $u\in C^2_{cs}(X^{\prime})$
    with $\int_{X^{\prime}}u\,dV_{g^{\prime}}=0$,
   one has
   $$
    \|u\|_{L^2}\;
     \le\; C\,\|du\|_{L^2}\; \le\; C^2\,\|\Delta u\|_{L^2}\,.
   $$
\end{lemma}

\begin{proof}
 Let $g$ be a Riemannian metric on $X$ such that
  $\,\frac{1}{c^2}\,g\,\le\, g^{\prime}\,\le\, c^2\,g\,$ on $X^{\prime}$
  for some $c>1$.
 Then it follows from the Minimax Principle for the eigenvalues
  of Laplacians
  that
  $$
   0\;<\;
   \frac{1}{c^{2m+2}}\,\lambda_{k, (X,g)}\;
   \le\; \frac{1}{c^{2m+2}}\,\lambda_{k, (X^{\prime}_i, g)}\;
   \le\; \lambda_{k,(X^{\prime}_i, g^{\prime})}\,,
  $$
  for all $k\ge 1$ and $i\ge 1$.
 This gives in particular a positive uniform lower bound
  $\lambda_{1,(X,g)}/c^{2m+2}$
  for all the first eigenvalues $\lambda_{1,(X^{\prime}_i,g^{\prime})}$
   of the Laplacians on $(X^{\prime}_i,g^{\prime})$.
 Since
   $C^2_{cs}(X^{\prime})=\bigcup_{i=1}^{\infty}\,C^2_{cs}(X^{\prime}_i)$,
  as done in
   [Jo3: I.~Proof of Theorem~2.17, last two paragraphs],
  taking the set of normalized eigenfunctions of the Laplacian
     $\Delta_{(X^{\prime}_i,g^{\prime})}$
    as an orthonormal basis in the Hilbert space $V_i$
    from completing
    $\left\{ u\in C^2_{cs}(X^{\prime\,\circ}_i)\,:\,
                     \int_{X^{\prime}_i}\,u\,=\,0 \right\}$
    in the $L^2$-norm
   gives then the inequalities in Lemma
    with $\,C= c^{m+1}\!\!\left/\lambda_{1,(X,g)}^{1/2}\right.\,$.
 This completes the proof.

\end{proof}

\begin{example}
{\bf [branched covering of $S^3$].} {\rm
 Let
  $(S^3, g_0)$ be a Riemannian $3$-sphere and
  $f:X\rightarrow S^3$ be a (connected smooth) branched covering of $S^3$
   of finite degree.
 Assume that
    the branch locus $\underline{\Gamma}\subset S^3$
      is a smooth link in $S^3$  and
    so is the branch locus $\Gamma:=f^{-1}(\underline{\Gamma})$ in $X$,
  and that $\Gamma\simeq\underline{\Gamma}$ under $f$.

 \bigskip

 \noindent
 $(a)$ {\sl Claim.} {\it
  $(X,f^{\ast}g_0)$ is a Riemannian manifold
  with tame curvature singularity supported on $\Gamma$.}

 \bigskip

 \noindent
 {\it Proof.}
  Since any two Riemannian metrics on a closed smooth manifold
   are quasi-conformal with uniformly bounded dilation,
  without loss of generality and through a partition of unity
   we may assume that the metric $g_0$ on $S^3$ has
   the additional property that there exists a tubular neighborhood
   $N_{\epsilon}(\underline{\Gamma})$ of $\underline\Gamma\subset S^3$
   such that $g_0|_{N_{\epsilon}(\underline{\Gamma})}$ is isometric
   to a finite disjoint union
    $\amalg_{i\in H_0(\underline{\Gamma};{\Bbb Z})}\underline{N}_i$
    of the solid torus
    $\underline{N}_i\simeq D^2_{\epsilon}\times S^1$
    with the flat product metric.
  Here, $D^2_{\epsilon}$ is a closed disk at the origin
   of radius $\epsilon$ in the standard flat ${\Bbb R}^2$.
  Each connected component $N_i$, $i\in H_0(\Gamma;{\Bbb Z})$, of
    $f^{-1}(N_{\epsilon}(\underline{\Gamma}))$
   is diffeomorphic to a solid torus $D^2\times S^1$ as well.
  With an appropriate choice of local coordinates on $N_i-\Gamma$,
   we may assume that
   $$
    \begin{array}{cccccc}
     f & : & N_i\,-\,\Gamma & \longrightarrow & \underline{N}_i\\
     && (r,\,\phi,\, \theta) & \longmapsto
      & (\underline{r},\, \underline{\phi},\, \underline{\theta})\;
        =\; (r,\, m_i\phi,\, \theta) &,
    \end{array}
   $$
  where
   $m_i\ge 2$ is the branching index of $f$
    along $\Gamma\cap N_i$;
   $(r, \phi)$, $(\underline{r},\underline{\phi})$
    are the polar coordinates of (the flat) $D^2$, $D^2_{\epsilon}$
    respectively; and
   $\theta\in {\Bbb R}/2\pi\simeq S^1$.
  In terms of this,
   $g_0|_{\underline{N}_i}$ is given by
    $\,ds^2_0\,=\, d\underline{r}^2\,
       +\,\underline{r}^2d\underline{\phi}^2\,
       +\,\frac{l_i^2}{4\pi^2}d\underline{\theta}^2\,$,
     where $l_i$ is the length of
      $\underline{\Gamma}\cap \underline{N}_i$ in $(S^3,g_0)$,
      and
   $(f^{\ast}g_0)|_{N_i-\Gamma}$ is given by
    $$
     d{s^{\prime}}^2\;
      =\;  dr^2\,+\,m_i^2r^2d\phi^2\,
                 +\,\frac{l_i^2}{4\pi^2}d\theta^2\;
      =\; (dr)^2\,+\, (m_ird\phi)^2\,
                  +\,\left( \frac{l_i}{2\pi}d\theta\right)^2\;
      =:\; \omega_r^2\,+\, \omega_{\phi}^2\,+\,\omega_{\theta}^2\,.
    $$
  Let $h_i:(0,\epsilon)\rightarrow {\Bbb R}^+$
   be an increasing smooth function
    with $h|_{(0,\,\epsilon/3)}$ the constant $1/m_i$ and
         $h|_{(2\epsilon/3,\,\epsilon)}$ the constant $1$.
  This defines a smooth function, still denoted by $h_i$, on $N_i-\Gamma$
   by identifying $(0,\epsilon)$ as the $r$-coordinate.
  Let
   $g^{\prime}:= f^{\ast}g_0$ on $X^{\prime}:=X-\Gamma$ and
   $g$ be the metric on $X^{\prime}$ defined by
   $$
    ds^2\;=\; \left\{
     \begin{array}{ccl}
      g^{\prime}
       && \mbox{on $\; X^{\prime}-N_{\epsilon}(\Gamma)$}\,,\\[1.2ex]
      \omega_r^2\,+\, \left( h_i\cdot\omega_{\phi}\right)^2\,
                  +\,\omega_{\theta}^2
       && \mbox{on $\; N_i-\Gamma\,$, $\,i\in H_0(\Gamma;{\Bbb Z})$}\,.
     \end{array}
              \right.
   $$
  Let $c:=\max_i\{m_i\}$.
  Then, by construction,
   $\,\frac{1}{c^2}g\,\le\, g\, \le g^{\prime}\,\le \, c^2g\,$
    on $X^{\prime}$.
  Furthermore,
   $g$ on $X^{\prime}$ extends to a Riemannian metric on $X$
   without singularity.
  This proves the claim.

 \noindent\parbox{16cm}{\raggedleft $\square\hspace{3em}$}

 \bigskip

 \noindent $(b)$
 Define
  $\,X^{\prime}_j\,:=\,
   \{p \in X^{\prime}\,:\,
       \distance_{g^{\prime}}(x,\Gamma)\ge 1/(j+j_0)\}\,$,
  $j\in {\Bbb N}$.
 For $j_0\in {\Bbb N}$ large enough,
  $\,X^{\prime}_1\subset X^{\prime}_2\subset\,\cdots\cdots\,$
  are connected compact $3$-manifolds-with-smooth-boundary
  that exhaust $X^{\prime}\,$:
  $\,\bigcup_j\,X^{\prime}_j\,=\,X^{\prime}$.
 It follows thus from Claim in Part $(a)$ and Lemma~2.1.2 that:
  %
  \begin{itemize}
   \item[$\cdot$]
    {\it There exists a constant $C>0$ such that
        whenever $u\in C^2_{cs}(X^{\prime})$
      with $\int_{X^{\prime}}u\,dV_{g^{\prime}}=0$,\\
     one has
     $\,\|u\|_{L^2}\,
        \le\, C\,\|du\|_{L^2}\, \le\, C^2\,\|\Delta u\|_{L^2}\,$.}
  \end{itemize}
}\end{example}

\bigskip

\subsection{A class of embedded special Lagrangian submanifolds in the
    flat Calabi-Yau 3-fold ${\Bbb R}^2\times S^1\times {\Bbb R}^3$.}

Let $Y^{\prime}:={\Bbb R}^2\times S^1\times {\Bbb R}^3$
 be the flat Calabi-Yau 3-fold with
  \begin{itemize}
   \item[$\cdot$]
    the real coordinates $(u_1, u_2, u_3, v_1, v_2, v_3)$,
          where $u_1, u_2, v_1, v_2, v_3\in {\Bbb R}$ and
                                $u_3\in {\Bbb R}/l)$,

   \item[$\cdot$]
    the complex structure $J^{\prime}$
     specified by the complex coordinates $(z_1, z_2, z_3)$ with\\
      $z_1=u_1+\sqrt{-1}v_1$, $\, z_2=u_2+\sqrt{-1}v_2$, and
      $\, z_3=u_3+\sqrt{-1}v_3$,

   \item[$\cdot$]
    the K\"{a}hler form
     $\omega^{\prime}
      =\frac{\sqrt{-1}}{2}
       (dz_1\wedge d\bar{z}_1 + dz_2\wedge d\bar{z}_2
                              + dz_3\wedge d\bar{z}_3)$,
     which specifies the K\"{a}hler metric
      $g^{\prime}=|dz_1|^2+|dz_2|^2+|dz_3|^2$, and

   \item[$\cdot$]
    the holomorphic $3$-form
     $\Omega^{\prime}=dz_1\wedge dz_2 \wedge dz_3$,
     which gives the calibration $\Real\Omega^{\prime}$.
  \end{itemize}
We'll denote $(Y^{\prime}, J^{\prime}, \omega^{\prime}, \Omega^{\prime})$
 also collectively by $Y^{\prime}$.

\bigskip

\begin{flushleft}
{\bf A class of embedded special Lagrangian submanifolds of $Y^{\prime}$
     via rolling isotropic submanifolds.}
\end{flushleft}
As Calabi-Yau manifolds,
 $Y^{\prime}$ is isomorphic to the product
  $Y^{\prime\prime}\times Y^{\prime\prime\prime}
   := {\Bbb R}^4\times (S^1\times {\Bbb R})$
 of a Calabi-Yau 2-fold and a Calabi-Yau 1-fold,
  where
   ${\Bbb R}^4$ is the $(u_1,u_2,v_1,v_2)$-coordinate subspace and
   $S^1\times {\Bbb R}^1$ is the $(u_3, v_3)$-coordinate subspace,
   with the induced Calabi-Yau manifold structures
    $(J^{\prime\prime}, \omega^{\prime\prime}, \Omega^{\prime\prime})$
     and
    $(J^{\prime\prime\prime}, \omega^{\prime\prime\prime},
      \Omega^{\prime\prime\prime})$
     respectively.
Note that $Y^{\prime\prime}$ is hyperK\"{a}hler.
Thus,
 special Lagrangian submanifolds in $Y^{\prime\prime}$
  can be obtained by
   introducing a new complex structure $\hat{J}^{\prime\prime}$
    on $Y^{\prime\prime}$ --
     with the associated complex coordinates given by
     $\,(\hat{z}_1,\,\hat{z}_2)\,
      =\,(u_1+\sqrt{-1}u_2,\,v_1-\sqrt{-1}v_2)\,$ --
    via a hyperK\"{a}hler rotation  and
   taking smooth holomorphic curves $C$ in
    $\hat{Y}^{\prime\prime}
     :=(Y^{\prime\prime}, \hat{J}^{\prime\prime})$.
Such $C$'s are isotropic submanifolds in
 $Y^{\prime\prime}\times Y^{\prime\prime\prime}$
  (with the original complex structure
        $(J^{\prime\prime}, J^{\prime\prime\prime})$).
$C\times S^1\times\{a\}$, for $a\in {\Bbb R}$, give then
 a class of embedded special Lagrangian submanifolds
 in $Y^{\prime\prime}\times Y^{\prime\prime\prime}$
 and, hence, in $Y^{\prime}$.

\begin{remark} {\it $[$basic expression under hyperK\"{a}hler rotation$]$.}
{\rm
 For later use,
  note that in terms of the hyperK\"{a}hler rotated complex coordinates
   $(\hat{z_1}, \hat{z_2})$ on ${\Bbb R}^4$,
  $$
  \begin{array}{rcl}
   \frac{\sqrt{-1}}{2}(dz_1\wedge d\bar{z}_1+dz_2\wedge d\bar{z}_2)
    & =
    & \frac{1}{2}(d\hat{z}_1\wedge d\hat{z}_2
                  + d\bar{\hat{z}}_1\wedge d\bar{\hat{z}}_2)\;\;
      =\;\; \Real(d\hat{z}_1\wedge d\hat{z}_2)\,,               \\[2ex]
   dz_1\wedge dz_2   & =
    & \frac{\sqrt{-1}}{2}( d\hat{z}_1\wedge d\bar{\hat{z}}_1
                           + d\hat{z}_2\wedge d\bar{\hat{z}}_2 )\,
      -\,\frac{1}{2}( d\hat{z}_1\wedge d\hat{z}_2
                    - d\bar{\hat{z}}_1\wedge d\bar{\hat{z}}_2 ) \\[1.2ex]
    & = & -\frac{1}{2}
            \Imaginary( d\hat{z}_1\wedge d\bar{\hat{z}}_1
                        + d\hat{z}_2\wedge d\bar{\hat{z}}_2 )\,
          -\,\sqrt{-1}\Imaginary(d\hat{z}_1\wedge d\hat{z}_2)\,.
  \end{array}
  $$
}\end{remark}

\bigskip

\begin{flushleft}
{\bf Deformations of a branched covering of the special Lagrangian
      ${\Bbb R}^2\times S^1\times\{\mbox{\boldmath $0$}\}$
     to embedded special Lagrangian submanifolds of $Y^{\prime}$.}
\end{flushleft}
Consider the embedded special Lagrangian submanifold
  $L^{\prime}:={\Bbb R}^2\times S^1\times\{\mbox{\boldmath $0$}\}$,
   where {\boldmath $0$}$\,\in {\Bbb R}^3$ is the origin,
  of $Y^{\prime}$.
In terms of the ${\Bbb C}$-${\Bbb R}$-valued coordinates
 $(\hat{z}_1,u_3,\hat{z}_2,v_3)$ on $Y^{\prime}$
 (and hence on $L^{\prime}$ as well),
let
 $$
  \begin{array}{ccccccc}
   \psi &: &  L\,\,\simeq {\Bbb R}^2\times S^1
        & \longrightarrow & Y^{\prime}           \\[1.2ex]
   && (x_1,x_2,x_3) & \longmapsto
    & ( (x_1+\sqrt{-1}x_2)^m\,,\, x_3,\, 0,\, 0)
  \end{array}
 $$
 be a branched covering of $L^{\prime}$ of degree $m\ge 2$
 with branch locus
  $\Gamma=\{(0,0)\}\times S^1 \subset L$ and
  $\Gamma^{\prime}
    = \{0\}\times S^1\times \{(0,0)\} \subset L^{\prime}$
  respectively.
This is an immersion in codimension $1$
 (cf.\ [L-Y6: Definition~2.3.8] (D(7))).
It can be deformed to smooth special Lagrangian embeddings
 by\footnote{$\psi^a$ can be defined alternatively by
              $\,(x_1,x_2,x_3)\; \longrightarrow\;
                 (\,  a^{-1/m}\,(x_1+\sqrt{-1}x_2)^m\,,\, x_3,\,
                     (x_1+\sqrt{-1}x_2),\, 0\,)$.
              This has the same image special Lagrangian submanifold
               $L^{\prime,a}$ but now with
               $D\psi^a(0,0,x_3)$ independent of $a$.
              While even the $C^0$-convergence
               of the alternative $\psi^a$ to $\psi$ fails
                for $\psi^a$ thus alternatively defined,
              $\psi^a\rightarrow \psi$, as $a\rightarrow 0$,
               remains to hold in the sense of currents.}
 $$
  \begin{array}{ccccccc}
   \psi^a &: &  L\,\,\simeq {\Bbb R}^2\times S^1
        & \longrightarrow & Y^{\prime}           \\[1.2ex]
   && (x_1,x_2,x_3) & \longmapsto
    & \left(\, (x_1+\sqrt{-1}x_2)^m\,,\, x_3,\,
               a^{1/m}\,(x_1+\sqrt{-1}x_2),\, 0\, \right)\,,
  \end{array}
 $$
 for $a>0$.
$\,\psi^a$ embeds $L$ into $Y^{\prime}$
 as the smooth embedded special Lagrangian submanifold
  $$
   L^{\prime, a}\; :=\; \{ a\hat{z}_1-\hat{z}_2^m=0,\, v_3=0 \}\;
                    \subset\; Y^{\prime}\,.
 $$
Then as $a\rightarrow 0$, $\psi^a$ is $C^{\infty}$-convergent
 to $\psi=:\psi^0$ on any compact subset $L$.
One can also interpret this convergence
 in the sense of currents on $Y^{\prime}$.
(In notation,
 $L^{\prime,a}\rightarrow mL^{\prime}$ as $a\rightarrow 0$.)

Let
 $\,L^{\prime,\diamond}:=L^{\prime}\cap \{\hat{z}_1\ne 0\}\,$,
 $\,L^{\prime,a,\diamond}:=L^{\prime,a}\cap \{\hat{z}_1\ne 0\}\,$
  for $a>0$,
 $\,L^{\diamond}:=\{(x_1,x_2)\ne (0,0)\}\,
   (\simeq ({\Bbb R}^2-\{(0,0)\})\times S^1)\,\subset L\,$, and
 $\,Y^{\prime,\diamond}
   := \{\hat{z}_1\ne 0\}\,
   (\simeq ({\Bbb R}^2-\{(0,0)\})\times S^1\times {\Bbb R}^3)\,
   \subset Y^{\prime}\,$
   with the induced Calabi-Yau structure still denoted
   by $(J^{\prime}, \omega^{\prime}, \Omega^{\prime})$.
Then
 $L^{\prime,a,\diamond}$ is an open dense subset of $L^{\prime,a}$
  that can be expressed
  via the graph of an exact $1$-form on $L^{\diamond}$
 as follows.
First, observe that
 since
  $\omega^{\prime}=du_1\wedge dv_1+du_2\wedge dv_2+ du_3\wedge dv_3$,
 the map
 $$
  \begin{array}{cccccc}
   \Phi^{\prime} & : & T^{\ast}L^{\prime,\diamond}
    & \longrightarrow & Y^{\prime,\diamond}  \\[1.2ex]
   && (u_1, u_2, u_3, p_{u_1}, p_{u_2}, p_{u_3})
    & \longmapsto
    & (u_1+\sqrt{-1}u_2,\, u_3,\,
       p_{u_1}-\sqrt{-1}p_{u_2},\, p_{u_3})  &,
  \end{array}
 $$
 where
  $(u_1, u_2, u_3, p_{u_1}, p_{u_2}, p_{u_3})$ corresponds to
  $p_{u_1}du_1+p_{u_2}du_2+p_{u_3}du_3
              \in T^{\ast}_{(u_1,u_2,u_3)}L^{\prime,\diamond}$,
 is a symplectomorphism with
  $\Phi^{\prime, \ast}\omega^{\prime}=\omega^{\prime}_{can}$.
Precomposing $\Phi^{\prime}$
 with the symplectic covering map
  $f_{!}: T^{\ast}L^{\diamond}\rightarrow T^{\ast}L^{\prime,\diamond}$
  that is canonically induced by the covering map
  $f:L^{\diamond}\rightarrow L^{\prime,\diamond}$,
one obtains a covering map
 $$
  \begin{array}{l}
   \hspace{8em}
   \Phi\;:=\; \Phi^{\prime} \circ f_{!}\; :\;
    T^{\ast}L^{\diamond} \hspace{1em}
    \longrightarrow  \hspace{1em} Y^{\prime,\diamond}   \\[1.2ex]
   (x_1\,,\, x_2\,,\, x_3\,,\,
    p_{x_1}\,,\, p_{x_2}\,,\, p_{x_3})                  \\[1.2ex]
   \hspace{4em}
    \longmapsto\hspace{1em}
     \left(
     (x_1+\sqrt{-1}x_2)^m\,,\, x_3\,,\,
     \frac{1}{m}\,(p_{x_1}-\sqrt{-1}p_{x_2})\,(x_1+\sqrt{-1}x_2)^{1-m}\,,\,
     p_{x_3}
    \right)
  \end{array}
 $$
 of degree $m$ such that $\Phi^{\ast}\omega^{\prime}=\omega_{can}$.
By construction,
 $L^{\prime,a,\diamond}$ lifts under $\Phi$ to a smooth section
 $\tilde{L}^{\prime,a,\diamond}$ of $T^{\ast}L^{\diamond}$
 over $L^{\diamond}$ that is Lagrangian with respect to
 $\omega_{can}
  =dx_1\wedge dp_{x_1} + dx_2\wedge dp_{x_2} + dx_3\wedge dp_{x_3}$.
Explicitly,
 $$
  \begin{array}{l}
   \Phi^{-1}(L^{\prime,a,\diamond})\;=\;
    \left\{
     (p_{x_1}-\sqrt{-1}p_{x_2})\,
      -\,m\,a^{1/m}\,(x_1+\sqrt{-1}x_2)^m\,
                     e^{2\pi\sqrt{-1}\,k/m}\;=\;0\,,
    \right.                                           \\[1.2ex]
   \hspace{16em}
    p_{x_3}\,=\,0 \hspace{6em}\left.\rule{0ex}{1.2em}
                  :\: k\,=\,0,\,\cdots\,,\,m-1 \right\}
  \end{array}
 $$
In the following, we choose $\tilde{L}^{\prime,a,\diamond}$
 that corresponds to the component with $k=0$.

\begin{lemma}
{\bf [$L^{\prime,a,\diamond}$ via exact 1-form on $L$].}
 $\tilde{L}^{\prime,a,\diamond}$ extends to a smooth section
  of $T^{\ast}L$ that corresponds to the exact $1$-form $dh^a$ with
  $h^a := \frac{m}{m+1}\,a^{1/m}\,
          \Real\left( (x_1+\sqrt{-1}x_2)^{m+1} \right)\,
   \in\, C^{\infty}(L)$.
\end{lemma}

\begin{proof}
 Since $\tilde{L}^{\prime,a,\diamond}$ is a smooth section of
  $T^{\ast}L^{\diamond}$ that is Lagrangian with respect to $\omega_{can}$,
 $\tilde{L}^{\prime,a,\diamond}$ is the graph of a closed $1$-form
  $\alpha$ on $L^{\diamond}$.
 To see that $\alpha$ is exact, one only needs to show that
  $\int_{\gamma}\alpha=0$ for any (smooth) closed loop $\gamma$
  in $L^{\diamond}$.
 Since
  $H_1(L^{\diamond}; {\Bbb Z})
    ={\Bbb Z}[\gamma_{(x_1,x_2)}] \oplus {\Bbb Z}[\gamma_{x_3}]$,
   where $\gamma_{(x_1,x_2)}$ is the (oriented) loop
    $\{ x_1^2+x_2^2=\epsilon^2, x_3=0\}$ for a small $\epsilon>0$
    and $\gamma_{x_3}$ is the (oriented) loop
     $\{x_1=\epsilon, x_2=0\}$ in $L^{\diamond}$,
   and
  $\int_{\gamma_{x_3}}\alpha=0$
   due to the fact that
    $\tilde{L}^{\prime,a,\diamond}$ lies in $\{p_{x_3}=0\}$,
 we only need to check that
   $\int_{\gamma_{(x_1,x_2)}}\alpha=0$.
 The latter follows immediately from the fact that,
  from the explicit equations for $\tilde{L}^{\prime,a,\diamond}\,$,
   $\,\tilde{L}^{\prime,a,\diamond}$ extends in $T^{\ast}L$
   to an embedded smooth Lagrangian submanifold
    that is realizable as a smooth section of $T^{\ast}L$ and
  that $[\gamma_{(x_1,x_2)}]=0$ in $H_1(L;{\Bbb Z})$.

 %

 Explicitly,
 let
  $$
   \lambda_{can}\;
    =\; p_{x_1}dx_1+p_{x_2}dx_2+p_{x_3}dx_3\;
    =\;\Real((p_{x_1}-\sqrt{-1}p_{x_2})(dx_1+\sqrt{-1}dx_2))\,
       +\, p_{x_3}dx_3
  $$
   be the canonical $1$-form on $T^{\ast}L$  and
  $\pi:T^{\ast}L\rightarrow L$ be the projection map.
 Then
  \begin{eqnarray*}
   \lefteqn{\alpha\;
    =\;\pi_{\ast}(\lambda_{can}|_{\tilde{L}^{\prime,a,\diamond}})\;
    =\; \Real\left(
         m\,a^{1/m}\,(x_1+\sqrt{-1}x_2)^m\,(dx_1+\sqrt{-1}dx_2)
             \right) }\\[1.2ex]
    && \hspace{7em}
       =\; \Real\left(
         \frac{m}{m+1}\,a^{1/m}\,d\left((x_1+\sqrt{-1}x_2)^{m+1}\right)
             \right)\,.
  \end{eqnarray*}
 This proves the lemma.

\end{proof}

%

\bigskip

\begin{notation}
{\bf [partial scaling on $Y^{\prime}$].} {\rm
 The partial scaling
   $t\cdot (\hat{z}_1, u_3, \hat{z_2}, v_3)
    := (t\hat{z}_1, u_3, t\hat{z_2}, v_3)$
   on $Y^{\prime}$ for $t>0$
  sends $L^{\prime, a}$ to $t\cdot L^{\prime, a}:= L^{\prime, at^{m-1}}$.
 For convenience, we denote $\psi^{at^{m-1}}$ also by $t\cdot \psi^a$.
 In particular, given $a>0$, then
  $t\cdot \psi^a\rightarrow \psi$ in the $C^{\infty}$-topology
  when $t\rightarrow 0$.
}\end{notation}

\bigskip

\subsection{Lagrangian Neighborhood Theorems.}

Recall first the following two Lagrangian Neighborhood Theorems:

\begin{theorem}
{\bf [immersed Lagrangian submanifold].}
 Let $(M,\omega)$ be a symplectic manifold and
  $f:N\rightarrow M$ be a compact immersed Lagrangian submanifold.
 Then there exist
  a neighborhood $U\subset T^{\ast}N$ of the zero-section  and
  an immersion $\Phi:U_N\rightarrow M$
  such that
   $\Phi|_N=f$ and
   $\Phi^{\ast}\omega=\omega_{can}$,
    where $\omega_{can}$ is the canonical symplectic form on $T^{\ast}N$.
\end{theorem}

\begin{theorem}
{\bf [Lagrangian foliation].}
{\rm ([Jo3: I, Theorem 4.2] and [We: Theorem 7.1].)}
 Let
  $(M,\omega)$ be $2m$-dimensional symplectic manifold and
  $N\subset M$ an embedded $m$-dimensional submanifold.
 Suppose $\,\{L_x: x\in N\}\,$ is a smooth family of embedded, noncompact
   Lagrangian submanifolds in $M$ parameterized by $x\in N$
   such that for each $x\in N$
    we have $x\in L_x$ and $T_xL_x\cap T_xN=\{0\}$.
 Then there exist
  an open neighborhood $U$ of the zero-section $N$ in $T^{\ast}N$
   such that the fibers of the natural projection $\pi:U\rightarrow N$
    are connected, and
  a unique embedding $\Phi:U\rightarrow M$ with
   $\,\Phi(\pi^{-1}(x))\subset L_x\,$ for each $x\in N$,
   $\,\Phi|_N = \id_N: N\rightarrow N\,$ and
   $\,\Phi^{\ast}\omega\,=\, \omega_{can}+\pi^{\ast}(\omega|_N)\,$,
    where $\omega_{can}$
     is the canonical symplectic structure on $T^{\ast}N$.
\end{theorem}

\noindent
The first theorem follows from essentially the same proof
 as that for the embedded Lagrangian case ([McD-S] and [We]) and
the second is stated in [Jo3: I, Theorem 4.2] as a variation of
 [We: Theorem 7.1].

\bigskip

\begin{flushleft}
{\bf Lagrangian Neighborhood Theorem for an embedded Lagrangian
     submanifold with a transverse Lagrangian distribution.}
\end{flushleft}
The following variation of Lagrangian Neighborhood Theorems
 is also needed in this work:

\begin{theorem}
{\bf [Lagrangian submanifold with transverse Lagrangian distribution].}
 Let
  $(M,\omega)$ be a $2m$-dimensional symplectic manifold,
  $N\subset M$ be a compact embedded Lagrangian submanifold, and
  $\Gamma_N\subset (T_{\ast}M)|_N$
   be a distribution of tangent $m$-planes in $M$ along $N$
   such that $\Gamma_x:=\Gamma_N|_x$
              is a Lagrangian subspace in $(T_xM, \omega_x)$
    and that $\,\Gamma_x\cap T_xD=\{0\}\,$ for all $x\in D$.
 Then there exist
   an open neighborhood $U\subset T^{\ast}N$
    of the zero-section and
   an embedding $\Phi:U\rightarrow M$
  such that
   $\,\Phi|_N=\id_N:N\rightarrow N\,$,
   $\,\Phi^{\ast}\omega=\omega_{can}\,$,
    where $\omega_{can}$
     is the canonical symplectic form on $T^{\ast}N$,  and
   $\,\Phi_{\ast}(T_0(T^{\ast}_xN))= \Gamma_x\,$ for all $x\in N$.
\end{theorem}

\begin{proof}
 $(a)$ {\it Reformulation of the problem.} \hspace{1em}
 Since the problem concerns only a neighborhood of $N$ in $M$,
 by Theorem~2.3.1
  one may assume that
   $(M,\omega)=(T^{\ast}N,\omega_{can})$ with $N$ the zero-section.
 The fact that $\Gamma_N$ is transverse to $N$ implies that
  $\Gamma_N$ defines
   a bundle map $T^{\ast}N\rightarrow T_{\ast}N$ and, hence,
   a bilinear map
  $(\,\cdot\,,\cdot\,)_{\Gamma_N}:
                      T^{\ast}N\oplus T^{\ast}N\rightarrow {\Bbb R}$.
 The Lagrangian property of $\Gamma_N$ implies that
  $(\,\cdot\,,\,\cdot\,)_{\Gamma_N}$
  is a symmetric bilinear functional on $T^{\ast}N$.

 \medskip

 \noindent
 $(b)$ {\it Realization of $(\,\cdot\,,\,\cdot\,)_{\Gamma_N}$
             as a restriction of Hessian.} \hspace{1em}
 Recall that
 a smooth function $f:T^{\ast}N\rightarrow {\Bbb R}$
   such that $(df)|_N\equiv 0$
  defines a symmetric bilinear functional $\Hess_f$
  on $(T_{\ast}(T^{\ast}N))|_N$ by setting
  $\,\Hess_f(v,w)\,=\,v(\tilde{w}f)\,$,
  where
   $v, w\in T_x(T^{\ast}N)$ for some $x\in N$ and
   $\tilde{w}$ is an extension of $w$
    to a smooth vector field in a neighborhood of $x$ in $T^{\ast}N$.
 $\Hess_f$ is independent of the extension $\tilde{w}$ of $w$
  and, hence, well-defined.
 It is called
  the {\it Hessian of $f$ at the critical manifold $N$ of $f$}.
 Since $T^{\ast}N$ is canonically embedded in $(T_{\ast}(T^{\ast}N))|_N$
  as vector bundles over $N$,
 $\Hess_f$ defines further a symmetric bilinear functional
  $\Hess_f^{\,\perp}$ on $T^{\ast}N$ by its restriction to $T^{\ast}N$.

 \bigskip

 \noindent
 {\sl Claim.} {\it
  There exists a smooth function $f: T^{\ast}N\rightarrow {\Bbb R}$
    that is supported in a compact neighborhood $U^{\prime}$
    of the zero-section
    such that
      both $f|_N$ and $(df)|_N$ vanish  and
    $\Hess_f^{\,\perp}=(\,\cdot\,,\,\cdot\,)_{\Gamma_N}$.
  $U^{\prime}$ can be chosen to be arbitrarily small.
 } 

 \bigskip

 \noindent
 {\it Proof of Claim.}
  Let $\pi: T^{\ast}N\rightarrow N$ be the built-in map.
  Then a local chart $V$
    (with coordinates ${\mathbf q}:=(q_1,\,\cdots\,,\,q_m)$) on $N$
   induces a local chart $\pi^{-1}(V)$ on $T^{\ast}N$
    (with induced canonical coordinates
     $({\mathbf q},{\mathbf p})
       :=(q_1,\,\cdots\,,\,q_m, p_1,\,\cdots\,,\,p_m)$.
  Let $\Gamma_V:=\Gamma_N|_V$.
  Then the symmetric bilinear functional
   $(\,\cdot\,,\,\cdot\,)_{\Gamma_V}$ on $T^{\ast}V$
   can be expressed as a ${\mathbf q}$-dependent quadratic form
    $g_V:=\frac{1}{2}\,\sum_{1\le i,j\le m}\,a_{ij}({\mathbf q})p_ip_j$,
    where $a_{ij}({\mathbf q})=a_{ji}({\mathbf q})$ for all $i,\,j$,
   in ${\mathbf p}$.
  As a function on $\pi^{-1}(V)$,
   $g_V$ satisfies the property
    that both $g_V|_V$ and $(dg_V)|_V$ vanish  and
    that $\Hess_{g_V}^{\,\perp}=(\,\cdot\,,\,\cdot\,)_{\Gamma_V}$.
  Observe now the following property,
   which can be checked straightforwardly:
  \begin{itemize}
   \item[$\cdot$]
    For convenience,
      call a function $g:\pi^{-1}(V)\rightarrow {\Bbb R}$ {\it admissible}
     if $g$ satisfies the condition that
       both $g|_V$ and $(dg)|_V$ vanish  and that
       $\Hess_g^{\,\perp}=(\,\cdot\,,\,\cdot\,)_{\Gamma_V}$.
    Then,
    {\it if $g_1,\,g_2$ are admissible functions on $T^{\ast}V$ and
            $h_1,\,h_2$ are positive functions on $V$,
         then $h_1g_1+h_2g_2$ is an admissible function on $T^{\ast}V$}.
  \end{itemize}
  Let
   $\{\mu_i\}_i$ be a partition of unity on $N$
    subordinate to a locally finite covering $\{V_i\}_i$ of $N$  and
   $g_{V_i}$ be an admissible function on $T^{\ast}V_i$,
    whose existence is demonstrated by the explicit example.
  Then it follows from the observation above
   that $g:=\sum_i\mu_ig_{V_i}$ is an admissible function on $T^{\ast}N$.
  Finally,
   introduce a cutoff function $\chi: T^{\ast}N\rightarrow [0,1]$
     that is supported on an arbitrarily small neighborhood
     of the zero-section.
  Then $f:=\chi g$ satisfies all the required properties in the Claim.

 \noindent\parbox{16cm}{\raggedleft $\square\hspace{3em}$}

 \medskip

 \noindent
 $(c)$ {\it $\Phi$ from time-$1$ map of Hamiltonian flow.} \hspace{1em}
 Let
  $f: T^{\ast}N\rightarrow {\Bbb R}$ be a smooth function on $T^{\ast}N$
   as constructed in Part (b),
  $X_f$ be the Hamiltonian vector field on $T^{\ast}N$ associated to $f$
   (i.e.\ $i(X_f)\omega_{can}=df$)
    and
  $\Phi:=\Phi_1:T^{\ast}N\rightarrow T^{\ast}N$ be the time-$1$ map
   of the Hamiltonian flow $\Phi_t$, $t\in {\Bbb R}$, on $T^{\ast}N$
   associated to $f$
   (i.e.\ $\frac{d}{dt}\Phi_t=X_f\circ\Phi_t$).
 By construction,
  $\Phi$ is a symplectomorphism on $T^{\ast}N$ that leaves $N$ fixed.
 Recall the proof of Claim in Part (b) and
        the notations and terminology therein.
 Then, since $f$ is admissible,
  $$
   f|_{\pi^{-1}(V)}\;
    =\; \frac{1}{2}\,\sum_{1\le i,j\le m}\,a_{ij}({\mathbf q})p_ip_j\,
        +\, o(|{\mathbf p}|^2)
  $$
  and
  $$
   X_f|_{\pi^{-1}(V)}\;=\;
    \sum_{i=1}^m \left(\sum_{j=1}^m\,a_{ij}({\mathbf q})\,p_j\right)\,
                 \frac{\partial}{\partial q_i}\,
        +\, o(|{\mathbf p}|)\,.
  $$
 Let $x\in V$, $V_x$ be a neighborhood of $x$ in $\pi^{-1}(V)$.
 Assume that $V_x$ is small enough and
  let $\Phi_{1,V_x}^{\prime}: V_x\rightarrow \pi^{-1}(V)$
   be the time-$1$ map of the flow generated by
  $\sum_{i=1}^m (\sum_{j=1}^m\,a_{ij}({\mathbf q})\,p_j)\,
                 \frac{\partial}{\partial q_i}\,$.
 Then, from the above approximation of $X_f|_{\pi^{-1}(V)}$,
  $$
   \Phi_{\ast}(T_0(T^{\ast}_xN))\;
   =\; \Phi_{1,V_x, \ast}^{\prime}(T_0(T^{\ast}_x\pi^{-1}(V)))\;
   =\; \Gamma_x\,.
  $$

 Take now $U=\Phi^{-1}(U^{\prime})$.
 Then $\Phi:U\rightarrow M$ gives a neighborhood of $N$ in $M$
  as required.
 This completes the proof of the theorem.

\end{proof}


%
%
%
%
%
%
%
%
%
%

\bigskip

\subsection{Admissible Lagrangian neighborhoods for $L^{\prime,a}$
    in $Y^{\prime}$ under $\psi^a$ and
     their geometry under partial scaling.}
Recall
 the flat Calabi-Yau $3$-fold
  $Y^{\prime}
   = ({\Bbb R}^2\times S^1\times {\Bbb R}^3,
      J^{\prime}, \omega^{\prime}, \Omega^{\prime})$
   with the real (resp.\ complex, hyperK\"{a}hler-rotated) coordinates
    $\,(u_1, u_2, u_3, v_1, v_2, v_3)\,$
   (resp.\
    $\,(z_1, z_2, z_3)
       = (u_1+\sqrt{-1}v_1, u_2+\sqrt{-1}v_2, u_3+\sqrt{-1}v_3)\,$,
    $\,(\hat{z}_1, u_3, \hat{z}_2, v_3)
       = (u_1+\sqrt{-1}u_2, u_3, v_1-\sqrt{-1}v_2, v_3)\,$)  and
 the embedded smooth special Lagrangian submanifolds
   $\,L^{\prime}\,=\, {\Bbb R}^2\times S^1\times\{\mbox{\boldmath $0$}\}\,$
    and
   $\,L^{\prime,a}\, =\, \{ a\hat{z}_1-\hat{z}_2^m=0,\, v_3=0 \}\,$,
     for $a>0$,
  in $Y^{\prime}$
 in Sec.~2.2.
Under the real coordinates, $(Y^{\prime},\omega^{\prime})$
 is identified with $(T^{\ast}L^{\prime}, \omega_{can})$ under the map
 $(u_1, u_2, u_3, v_1, v_2, v_3)
                    = (u_1, u_2, u_3, p_{u_1}, p_{u_2}, p_{u_3})$.
Recall also
 the embedded special Lagrangian map
  $\psi^a:L\rightarrow Y^{\prime}$,
   $\,(x_1, x_2, x_3)\mapsto
        ((x_1+\sqrt{-1}x_2)^m,\, x_3,\,
                 a^{1/m}\,(x_1+\sqrt{-1}x_2),\, 0\,)\,$,
   with image $L^{\prime,a}$,
 $h^a\in C^{\infty}(L)$ in Lemma~2.2.2,
   and
 the partial scaling on $Y^{\prime}$,
 $(\hat{z}_1, u_3, \hat{z}_2, v_3)
                   \mapsto (t\hat{z}_1, u_3, t\hat{z}_2, v_3)$,
 for $t>0$ in Notation~2.2.3.

\begin{proposition}
{\bf [admissible Lagrangian neighborhood: existence].}
 Given $R^{\prime}_0>0$,
 there exist
  a neighborhood $U_L^a$ of the zero-section of $T^{\ast}L$  and
  a symplectic embedding $\Phi_L^a: U_L^a\rightarrow Y^{\prime}$
   whose restriction to the zero-section is $\psi^a$
  such that
   $$
    \Phi_L^a|_{\pi^{-1}(\{|(x_1+\sqrt{-1}x_2)^m|\ge R^{\prime}_0\})}
     (\,\cdot\,)\;
    =\; \psi^0_{\,!}(dh^a+\,\cdot\,)
   $$
   and
  $\Phi_L^a$ is equivariant with respect to the translations along
   the $S^1$-direction under $u_3=x_3$.
 Here, $\pi:U_L^a\rightarrow L$ is the restriction of
  the projection map $\pi: T^{\ast}L\rightarrow L$.
\end{proposition}

\begin{proof}
 The decomposition-by-Lagrangian-subbundles
  $T_{\ast}Y^{\prime}|_{L^{\prime,a}}
   =T_{\ast}L^{\prime, a}\oplus J\cdot T_{\ast}L^{\prime,a}$
  and the fact that $T_{\ast}L^{\prime,a}$ is a trivial bundle
 imply that
  a Lagrangian distribution along and transverse to $L^{\prime,a}$
   in $Y^{\prime}$ is the same as a section of the trivial bundle
   $\Mor^{sym}(J\cdot T_{\ast}L^{\prime,a}, T_{\ast}L^{\prime,a})$
   of linear maps from $J\cdot T_pL^{\prime,a}$ to
   $T_pL^{\prime,a}$, $p\in L^{\prime,a}$,
   that are represented by symmetric matrices
   under the canonical isomorphism
   $T_{\ast}L^{\prime,a}
      \stackrel{\sim}{\rightarrow} J\cdot T_{\ast}L^{\prime,a}$
   by $J$ and a fixed trivialization of
   $J\cdot T_{\ast}L^{\prime, a}\simeq L^{\prime,a}\times{\Bbb R}^3$.
 Under the identification $T^{\ast}L^{\prime}\simeq Y^{\prime}$,
  the fibers of $(T^{\ast}L^{\prime})|_{\{|\hat{z}_1|\ge R^{\prime}_0\}}$
  give a transverse Lagrangian foliation along $L^{\prime,a,\diamond}$
  and, hence, a smooth map
  $L^{\prime,a}\cap \{|\hat{z}_1|\ge R^{\prime}_0\}
                                   \rightarrow {\Bbb R}^3$.
 It can always be extended to a smooth map
  $L^{\prime,a}\rightarrow {\Bbb R}^3$ and, hence,
  a transverse Lagrangian distribution along $L^{\prime,a}$.
 The linear structure on $Y^{\prime}$ turns further
  a transverse Lagrangian distribution along $L^{\prime,a}$
  into a transverse Lagrangian foliation in a neighborhood of
  $L^{\prime,a}$ in $Y^{\prime}$.

 In our situation, all these constructions can be made invariant
  under the $S^1$-translations on $Y^{\prime}$ in the $u_3$-coordinate.
 The proposition now follows from the proof of Theorem~2.3.2.

\end{proof}

\begin{definition}
{\bf [admissible Lagrangian neighborhood for $L^{\prime, a}$
      under $\psi^a$].} {\rm
 A Lagrangian neighborhood
   $\Phi_L^a: U_L^a\rightarrow Y^{\prime}$ for $L^{\prime,a}$
   as in Proposition~2.4.1 for some $R^{\prime}_0>0$
  is called an {\it admissible Lagrangian neighborhood}
   for $L^{\prime, a}$ under $\psi^a$.
}\end{definition}

\begin{notation}
{\bf [partial scaling].} {\rm
 Recall the map $t\cdot\psi^a:= \psi^{at^{m-1}}$ from Notation~2.2.3.
 We now extend this to denote
  $$
   \begin{array}{ccccc}
    t\cdot\Phi_L^a\; :=\;  \Phi_L^{a,t}
     & : & U_L^{a,t}  & \longrightarrow  & Y^{\prime}\\[.8ex]
    && (x_1,x_2,x_3, p_{x_1}, p_{x_2}, p_{x_3})
     & \longmapsto
     & t\cdot \Phi_L^a
       \left(t^{-1/m}\cdot(x_1,x_2,x_3, p_{x_1}, p_{x_2}, p_{x_3}) \right)
   \end{array}
  $$
  for a given admissible Lagrangian neighborhood
  $\Phi_L^a: U_L^a\rightarrow Y^{\prime}$ for $L^{\prime,a}$
  under $\psi^a$.
 Here,
  $t^{-1/m}\cdot\,$ is the partial scaling on $T^{\ast}L$,
   defined by
   $$
    (x_1,x_2,x_3, p_{x_1}, p_{x_2}, p_{x_3})\;
    \longmapsto\;
    (t^{-1/m}x_1, t^{-1/m}x_2,x_3,
                  t^{-1/m}p_{x_1}, t^{-1/m}p_{x_2}, p_{x_3})\,,
   $$
   and
  $U_L^{a,t}=t^{1/m}\cdot U_L^a \subset
   (T^{\ast}L,
    t^{2(m-1)/m}(dx_1\wedge dp_{x_1}+dx_2\wedge dp_{x_2})
    +dx_3\wedge dp_{x_3})$.
 By construction, $t\cdot\Phi_L^a$ is a symplectic embedding
  that gives an admissible Lagrangian neighborhood
  for $t\cdot L^{\prime,a}$ under $t\cdot \psi^a$.
 It has the image $t\cdot \Image\Psi_L^a$.
}\end{notation}

\begin{proposition}
{\bf [radius of fibers of $U_L^{a,t}/L$].}
 Let
  $g^{a,t}:=(t\cdot\psi^a)^{\ast}g^{\prime}$
   be the pull-back metric on $L$, $0<t<1$, and
  $|\cdot|_{g^{a,t}}$ be the norm on fibers of $T^{\ast}L$
   via $(g^{a,t})^{-1}$ with respect to the dual basis.
 Note that for $A^{\prime}_1>0$ small enough,
  $U_L^a$ contains the neighborhood
  $\{\alpha\in T^{\ast}L\,:\, |\alpha|_{g^a}< A^{\prime}_1 \}$
  of the zero-section in $T^{\ast}L$, where $g^a:= g^{a,1}$.
 Then, $U_L^{a,t}$ contains the neighborhood
  $\{\alpha\in T^{\ast}L\,:\, |\alpha|_{g^{a,t}}< A^{\prime}_1 t\}$
  of the zero-section in $T^{\ast}L$.
\end{proposition}

\begin{proof}
 Note that $U_L^{a,t}=t^{1/m}\cdot U_L^a$
   as submanifolds in $T^{\ast}L$ and
  that $g^{a,t}\ge t^2g^a\cdot(t^{-1/m}\cdot\,)_{\ast}$
   since $t\cdot \psi^a=t\cdot(\psi^a(t^{-1/m}\cdot\;\;))$
    and the scaling is only partial.
 The proposition follows.

\end{proof}

Continuing the situation in Proposition~2.4.4.
Recall from Definition~1.2
 the pull-back connection $\hat{\nabla}^{a,t}$ on $U_L^{a,t}$
 that is constructed solely by $g^{a,t}$.

\begin{proposition}
{\bf [compatibility of pull-back connection under partial scaling].}
 The partial scaling $t^{1/m}\cdot$ on $T^{\ast}L$
  takes $(U_L^a,\hat{\nabla}^a)$ to $(U_L^{a,t}, \hat{\nabla}^{a,t})$,
  where $\hat{\nabla}^a:=\hat{\nabla}^{a,1}$.
\end{proposition}

\begin{proof}
 Consider
  the associated horizontal distribution $\hat{\nabla}^{\prime,a,t}$
  of $\hat{\nabla}^{a,t}$ on the fibered (codimension-$0$) submanifold
  $U_{L^{\prime,at^{m-1}}}:=\Phi_L^{a,t}(U_L^{a,t})\subset Y^{\prime}$
  and
 set $\hat{\nabla}^{\prime,a}:= \hat{\nabla}^{\prime,a,1}$.
 Then, the proposition is equivalent to the statement that {\it
  the partial scaling $t\cdot$ on $L^{\prime}$ takes
  $(U_{L^{\prime,a}}, \hat{\nabla}^{\prime,a})$
  to $(U_{L^{\prime,at^{m-1}}}, \hat{\nabla}^{\prime,a,t})$}.
 Which follows by construction.

\end{proof}

Continuing the discussion.
Recall
 the standard holomorphic $3$-form $\Omega^{\prime}$ on $Y^{\prime}$
 and let $\beta^{a,t}:=(\Phi_L^{a,t})^{\ast}(\Imaginary\Omega^{\prime})$.

\begin{proposition}
{\bf [bound for $\|(\hat{\nabla}^{a,t})^k\beta^{a,t}\|_{C^0}$,
      $k=0,\,1,\,2,\,3$].}
 Assume that $t\in(0,1]$.
 Then there exists a constant $A^{\prime}_4>0$ such that
  $\|(\hat{\nabla}^{a,t})^k\beta^{a,t}\|_{C^0}\le A^{\prime}_4t^{-k}$
  for $k=0,\,1,\,2,\,3$.
\end{proposition}

\begin{proof}
 Let
  $g^{\prime,a,t}$ be the metric on $L^{\prime,at^{m-1}}$
   induced by $g^{\prime}$ on $Y^{\prime}$ and
  $\beta^{\prime,a,t}
    := (\Imaginary\Omega^{\prime})|_{U_{L^{\prime,at^{m-1}}}}$.
 As $g^{a,t}$ on $L$ and $\hat{\nabla}^{a,t}, \beta^{a,t}$ on $U_L^{a,t}$
  come from the pull-back of
   $g^{\prime,a,t}$, $\hat{\nabla}^{\prime,a,t}$, and $\beta^{\prime,a,t}$
   via $\Phi_L^{a,t}$,
 we will directly prove the corresponding inequalities
  on the $Y^{\prime}$-side in three steps.

 \bigskip

 \noindent
 $(a)$
 {\it An explicit expression for $\hat{\nabla}^{\prime,a,\diamond}$
      in $T^{\ast}(T^{\ast}L^{\prime,a,\diamond})$.}
 \hspace{1em}
 Let
  $\hat{z}_1=\hat{r}_1e^{\sqrt{-1}\hat{\theta}_1}$
   and $\hat{z}_2=\hat{r}_2e^{\sqrt{-1}\hat{\theta}_2}$.
 Then
  $$
   L^{\prime,a,\diamond}\;
    =\; \{(\hat{r}_1,\hat{\theta}_1,u_3,\hat{r}_2,\hat{\theta}_2,0)\;:\;
           \hat{r}_2=a^{1/m}\hat{r}_1^{1/m},\,
           \hat{\theta}_2=\hat{\theta}_1/m,\, \hat{r}_1>0\}
  $$
   and
  $(\hat{r}_1,\hat{\theta}_1,u_3)
   \in {\Bbb R}^+\times ({\Bbb R}/(2\pi m)) \times ({\Bbb R}/l)$
   serves as a global coordinate chart on $L^{\prime,a,\diamond}$.
 Let
  $(\hat{r}_1,\hat{\theta}_1,u_3,
     s_{\hat{r}_1}, s_{\hat{\theta}_1}, s_{u_3})$
   be the induced coordinates on $T_{\ast}L^{\prime,a,\diamond}$
   through the trivialization of $T_{\ast}L^{\prime,a,\diamond}$
   by the coordinate frame
   $(\partial_{\hat{r}_1}, \partial_{\hat{\theta}_1},\partial_{u_3})$
   on $L^{\prime,a,\diamond}$
   and
  $(\hat{r}_1,\hat{\theta}_1,u_3,
     p_{\hat{r}_1}, p_{\hat{\theta}_1}, p_{u_3})$
   be the induced coordinates on $T^{\ast}L^{\prime,a,\diamond}$
   through the trivialization of $T^{\ast}L^{\prime,a,\diamond}$
   by the dual coframe $(d\hat{r}_1, d\hat{\theta}_1,du_3)$
   on $L^{\prime,a,\diamond}$.
  In terms of these coordinates,
   $$
    \begin{array}{rcl}
     g^{\prime,a,\diamond}
      & :=\: & g^{\prime}|_{L^{\prime,a,\diamond}} \\[.8ex]
     &   =
      & \left( 1\,+\,m^{-2}a^{2/m}\hat{r}_1^{2(1-m)/m}
        \right) d\hat{r}_1^2\,
         +\, \hat{r}_1^2\,
              \left( 1\,+\,m^{-2}a^{2/m}\hat{r}_1^{2(1-m)/m}
              \right) d\hat{\theta}_1^2\,
         +\, du_3^2                                \\[.8ex]
     & \:=: & A(\hat{r}_1)\,d\hat{r}_1^2\,
              +\, B(\hat{r}_1)\,d\hat{\theta}_1^2\, +\, du_3^2\,.
    \end{array}
   $$
  This is a product of a $2$-dimensional conformally flat metric
   with a circle.
  The Levi-Civita connection $\nabla^{\prime,a,\diamond}$
   from $g^{\prime,a,\diamond}$ defines
   a horizontal distribution $^{\ast}H^{\prime,a,\diamond}$
    in $T^{\ast}L^{\prime,a,\diamond}$,
    given by the kernel of the following ${\Bbb R}^3$-valued $1$-form
    on $T^{\ast}L^{\prime,a,\diamond}$:
    $$
     \left(
      \begin{array}{ll}
       dp_{\hat{r}_1}\\[1.2ex]  dp_{\hat{\theta}_1}\\[1.2ex] dp_{u_3}
      \end{array}\right)\;
       -\; \left(
            \begin{array}{lll}
             \mbox{\boldmath $\omega$}_{\hat{r}_1}^{\hat{r}_1}
              & \mbox{\boldmath $\omega$}_{\hat{r}_1}^{\hat{\theta}_1}
              & 0 \\[1.2ex]
             \mbox{\boldmath $\omega$}_{\hat{\theta}_1}^{\hat{r}_1}
              & \mbox{\boldmath $\omega$}_{\hat{\theta}_1}^{\hat{\theta}_1}
              & 0 \\[1.2ex]
             \;0 & \;0 & 0
            \end{array}
           \right)
     \left(
      \begin{array}{ll}
       p_{\hat{r}_1}\\[1.2ex] p_{\hat{\theta}_1}\\[1.2ex] p_{u_3}
      \end{array}\right)\,.
    $$
    and
   a horizontal distribution $_{\ast}H^{\prime,a,\diamond}$
    in $T_{\ast}L^{\prime,a,\diamond}$,
    given by the kernel of the following ${\Bbb R}^3$-valued $1$-form
    on $T_{\ast}L^{\prime,a,\diamond}$:
    $$
     \left(
      \begin{array}{ll}
       ds_{\hat{r}_1}\\[1.2ex]  ds_{\hat{\theta}_1}\\[1.2ex] ds_{u_3}
      \end{array}\right)\;
       +\; \left(
            \begin{array}{lll}
             \mbox{\boldmath $\omega$}_{\hat{r}_1}^{\hat{r}_1}
              & \mbox{\boldmath $\omega$}_{\hat{\theta}_1}^{\hat{r}_1}
              & 0 \\[1.2ex]
             \mbox{\boldmath $\omega$}_{\hat{r}_1}^{\hat{\theta_1}}
              & \mbox{\boldmath $\omega$}_{\hat{\theta}_1}^{\hat{\theta}_1}
              & 0 \\[1.2ex]
             \;0 & \;0 & 0
            \end{array}
           \right)
     \left(
      \begin{array}{ll}
       s_{\hat{r}_1}\\[1.2ex] s_{\hat{\theta}_1}\\[1.2ex] s_{u_3}
      \end{array}\right)\,.
    $$
  Here, the $1$-forms
   $\mbox{\boldmath $\omega$}_{\mbox{\tiny $\bullet$}}
                             ^{\mbox{\tiny $\bullet$}}$
   on $L^{\prime,a,\diamond}$ are given by
   $$
    \nabla^{\prime,a,\diamond}\partial_{\hat{r}_1}\;
    =\; \mbox{\boldmath $\omega$}_{\hat{r}_1}^{\hat{r}_1}\,
        \partial_{\hat{r}_1}\,
       +\,\mbox{\boldmath $\omega$}_{\hat{r}_1}^{\hat{\theta}_1}\,
          \partial_{\hat{\theta}_1}
     \hspace{2em}\mbox{and}\hspace{2em}
    \nabla^{\prime,a,\diamond}\partial_{\hat{\theta}_1}\;
    =\; \mbox{\boldmath $\omega$}_{\hat{\theta}_1}^{\hat{r}_1}\,
        \partial_{\hat{r}_1}\,
       +\,\mbox{\boldmath $\omega$}_{\hat{\theta}_1}^{\hat{\theta}_1}\,
          \partial_{\hat{\theta}_1}\,;
   $$
  explicitly,
   $$
    \begin{array}{lclcr}
     \mbox{\boldmath $\omega$}_{\hat{r}_1}^{\hat{r}_1}  & =
      & \Gamma_{\hat{r}_1\hat{r}_1}^{\hat{r}_1} d\hat{r}_1\,
        +\,\Gamma_{\hat{\theta}_1\hat{r}_1}^{\hat{r}_1} d\hat{\theta}_1
      & = & \frac{1}{2}A(\hat{r}_1)^{-1}
             \frac{d}{d{\hat{r}_1}}A(\hat{r}_1)\,d\hat{r}_1\,, \\[1.2ex]
     \mbox{\boldmath $\omega$}_{\hat{r}_1}^{\hat{\theta}_1}  & =
      & \Gamma_{\hat{r}_1\hat{r}_1}^{\hat{\theta}_1}d\hat{r}_1\,
        +\,\Gamma_{\hat{\theta}_1\hat{r}_1}^{\hat{\theta}_1}\,
           d\hat{\theta}_1
      & = & \frac{1}{2}B(\hat{r}_1)^{-1}
             \frac{d}{d\hat{r}_1}B(\hat{r}_1)\,d\hat{\theta}_1\,,
                                                             \\[1.2ex]
     \mbox{\boldmath $\omega$}_{\hat{\theta}_1}^{\hat{r}_1}
      & = & \Gamma_{\hat{r}_1\hat{\theta}_1}^{\hat{r}_1} d\hat{r}_1\,
            +\,\Gamma_{\hat{\theta}_1\hat{\theta}_1}^{\hat{r}_1}
               d\hat{\theta}_1
      & = & -\frac{1}{2}A(\hat{r}_1)^{-1}
              \frac{d}{d\hat{r}_1}B(\hat{r}_1)\, d\hat{\theta}_1\,,
                                                          \\[1.2ex]
     \mbox{\boldmath $\omega$}_{\hat{\theta}_1}^{\hat{\theta}_1}
      & = & \Gamma_{\hat{r}_1\hat{\theta}_1}^{\hat{\theta}_1} d\hat{r}_1\,
            +\,\Gamma_{\hat{\theta}_1\hat{\theta}_1}^{\hat{\theta}_1}
               d\hat{\theta}_1
      & = & \frac{1}{2}B(\hat{r}_1)^{-1}
             \frac{d}{d\hat{r}_1}B(\hat{r}_1)\,d\hat{r}_1\,.\\[1.2ex]
    \end{array}
   $$
 The coordinate frame
  $(\partial_{\hat{r}_1}, \partial_{\hat{\theta}_1}, \partial_{u_3},
    \partial_{p_{\hat{r}_1}}, \partial_{p_{\hat{\theta}_1}},
    \partial_{p_{u_3}})$
  on $T^{\ast}L^{\prime,a,\diamond}$ specifies a trivialization
  of the bundle $T_{\ast}(T^{\ast}L^{\prime,a,\diamond})$ over
  $T^{\ast}L^{\prime,a,\diamond}$
  and hence coordinates
  $$
   (\hat{r}_1,\,           \hat{\theta}_1,\,           u_3,\,
    p_{\hat{r}_1},\,       p_{\hat{\theta}_1},\,       p_{u_3},\,
    \xi_{\hat{r}_1},\,     \xi_{\hat{\theta}_1},\,     \xi_{u_3},\,
    \xi_{p_{\hat{r}_1}},\, \xi_{p_{\hat{\theta}_1}},\, \xi_{p_{u_3}})
   $$
  thereupon.
 In terms of this, the horizontal distribution in
  $T_{\ast}(T^{\ast}L^{\prime,a,\diamond})$ that defines the connection
   $\hat{\nabla}^{\prime,a,\diamond}$
  is given by the kernel of the following ${\Bbb R}^6$-valued $1$-form
  on $T_{\ast}(T^{\ast}L^{\prime,a,\diamond})$:
  $$
   \left(\begin{array}{c}
    d\xi_{\hat{r}_1}\\[1.2ex]     d\xi_{\hat{\theta}_1}\\[1.2ex]
    d\xi_{u_3}\\[1.2ex]           d\xi_{p_{\hat{r}_1}}\\[1.2ex]
    d\xi_{p_{\hat{\theta}_1}}\\[1.2ex]    d\xi_{p_{u_3}}
         \end{array}\right)\;
    +\; \left(
         \begin{array}{llllll}
          \mbox{\boldmath $\omega$}_{\hat{r}_1}^{\hat{r}_1}
           & \mbox{\boldmath $\omega$}_{\hat{\theta}_1}^{\hat{r}_1}
           & 0 & 0 & 0 & 0 \\[1.2ex]
          \mbox{\boldmath $\omega$}_{\hat{r}_1}^{\hat{\theta_1}}
           & \mbox{\boldmath $\omega$}_{\hat{\theta}_1}^{\hat{\theta}_1}
           & 0 & 0 & 0 & 0 \\[1.2ex]
          \;0 & \;0 & 0   & 0 & 0 & 0 \\[1.2ex]
          \;0 & \;0 & 0   & 0 & 0 & 0 \\[1.2ex]
          \;0 & \;0 & 0   & 0 & 0 & 0 \\[1.2ex]
          \;0 & \;0 & 0   & 0 & 0 & 0
         \end{array}
        \right)
   \left(\begin{array}{c}
    \xi_{\hat{r}_1}\\[1.2ex]     \xi_{\hat{\theta}_1}\\[1.2ex]
    \xi_{u_3}\\[1.2ex]           \xi_{p_{\hat{r}_1}}\\[1.2ex]
    \xi_{p_{\hat{\theta}_1}}\\[1.2ex]     \xi_{p_{u_3}}
         \end{array}\right)\hspace{1em}
   =:\; d\vec{\xi}\,+\,\mbox{\boldmath $\hat{\omega}$}\vec{\xi}\,.
  $$
 It follows that in terms of the coordinates
  $(\hat{r}_1,\, \hat{\theta}_1,\, u_3,\,
    p_{\hat{r}_1},\,p_{\hat{\theta}_1},\,p_{u_3})$
  on $T^{\ast}L^{\prime,a,\diamond}$,
  $$
   \hat{\nabla}^{\prime,a,\diamond}s\;
    =\; ds\,-\, ^t\mbox{\boldmath $\hat{\omega}$}s
  $$
  for $s$ a section of the bundle
   $T^{\ast}(T^{\ast}L^{\prime,a,\diamond})$ over
    $T^{\ast}L^{\prime,a,\diamond}$
   that is trivialized by the coframe
   $(d\hat{r}_1,\, d\hat{\theta}_1,\, du_3,\,
     dp_{\hat{r}_1},\, dp_{\hat{\theta}_1},\, dp_{u_3})$
   on $T^{\ast}L^{\prime,a,\diamond}$.
 Here, $^t\mbox{\boldmath $\hat{\omega}$}$ is the transpose
  of the matrix-valued $1$-form $\mbox{\boldmath $\hat{\omega}$}$.
 One can convert the above expression
  for $\hat{\nabla}^{\prime,a,\diamond}$ to an expression
  in terms of the complex-real canonical coordinates
  $(\hat{z}_1, p_{\hat{z}_1})
   :=(u_1+\sqrt{-1}u_2,u_3, p_{u_1}-\sqrt{-1}p_{u_2},p_{u_3})$
  as

  \begin{eqnarray*}
   \hat{\nabla}^{\prime,a,\diamond}
    _{\partial_{\hat{z}_1}}\,d\hat{z}_1   & =
    & -\,\frac{
          m^{-3}(1-m)\,a^{2/m}\,
           \hat{z}_1^{(1-2m)/m}\,
           \bar{\hat{z}}_1^{(1-m)/m} }
       {1\, +\, m^{-2}\,a^{2/m}\,
            \hat{z}_1^{(1-m)/m}\,
            \bar{\hat{z}}_1^{(1-m)/m} }\,
      d\hat{z}_1\,,                                      \\[2.4ex]
   \hat{\nabla}^{\prime,a,\diamond}
    _{\partial_{\bar{\hat{z}}_1}}\,d\bar{\hat{z}}_1   & =
    & -\,\frac{
          m^{-3}(1-m)\,a^{2/m}\,
           \hat{z}_1^{(1-m)/m}\,
           \bar{\hat{z}}_1^{(1-2m)/m} }
       {1\, +\, m^{-2}\,a^{2/m}\,
            \hat{z}_1^{(1-m)/m}\,
            \bar{\hat{z}}_1^{(1-m)/m} }\,
      d\bar{\hat{z}}_1\,,                                \\[2.4ex]
   \hat{\nabla}^{\prime,a,\diamond}
    _{\partial_{\hat{z}_1}}\,d\hat{z}_1   & =
    & \hat{\nabla}^{\prime,a,\diamond}
       _{\partial_{\hat{z}_1}}\,d\hat{z}_1
      \hspace{1em}=\hspace{1em}
      \mbox{all other $\hat{\nabla}^{\prime,a,\diamond}
         _{\bullet}\,\bullet$}
      \hspace{1em}=\hspace{1em} 0\,.
  \end{eqnarray*}

 \bigskip

 \noindent
 $(b)$
 {\it A uniform bound for
      $\|(\hat{\nabla}^{\prime,a})^k\beta^{\prime,a}\|_{C^0}$,
      $k=0,\,1,\,2,\,3$, for $U_{L^{\prime,a}}$.}
 \hspace{1em}
 Recall first some basic variations to express $\beta^{\prime}$
  on $Y^{\prime}$:
  $$
   \begin{array}{crl}
    \beta^{\prime}  & :=
     & \Imaginary(dz_1\wedge dz_2\wedge dz_3)\;\;
        =\;\; \Real(dz_1\wedge dz_2)\wedge dv_3
              + \Imaginary(dz_1\wedge dz_2)\wedge du_3   \\[1.2ex]
     & = & \frac{\sqrt{-1}}{2}(
            d\hat{z}_1\wedge d\bar{\hat{z}}_1
             + d\hat{z}_2\wedge d\bar{\hat{z}}_2)\wedge dv_3\,
           -\, \Imaginary(d\hat{z}_1\wedge d\hat{z}_2)\wedge du_3\,.
   \end{array}
  $$
 To re-express $\beta^{\prime}$
  in terms of $T^{\ast}L^{\prime,a,\diamond}$
  in the region $\{|\hat{z}_1|\ge R^{\prime}_0\}$,
 consider the smooth map
  $Y^{\prime,\diamond}\rightarrow Y^{\prime,\diamond}$,
  $(\hat{z}_1,u_3,\hat{z}_2,v_3)
   \mapsto (\hat{z}_1, u_3, \hat{z}_2+a^{1/m}\hat{z}_1^{1/m}, v_3)$.
 Then, $\beta^{\prime}$ is pulled back to
  $$
   \begin{array}{rcl}
    \beta^{\prime}|
     _{T^{\ast}L^{\prime,a,\diamond}_{\{|\hat{z}_1|\ge R^{\prime}_0\}}}
     & =
     & \frac{\sqrt{-1}}{2}\left(
         (1+m^{-2}a^{2/m}\hat{z}_1^{(1-m)/m}\bar{\hat{z}}_1^{(1-m)/m})\,
          d\hat{z}_1\wedge d\bar{\hat{z}}_1
          + dp_{\hat{z}_1}\wedge d\bar{p}_{\hat{z}_1}
                           \right)\wedge dp_{u_3}             \\[1.2ex]
    && +\, \frac{\sqrt{-1}}{2}\,
            m^{-1}a^{1/m}\,
            \left(
             \hat{z}_1^{(1-m)/m}\, d\hat{z}_1\wedge d\bar{p}_{\hat{z}_1}
              - \bar{\hat{z}}_1^{(1-m)/m}\,
                  d\bar{\hat{z}}_1\wedge dp_{\hat{z}_1}
            \right) \wedge dp_{u_3}                           \\[1.2ex]
    && -\, \Imaginary(d\hat{z}_1\wedge dp_{\hat{z}_1})\wedge du_3
   \end{array}
  $$
  with the coordinates on $T^{\ast}L^{\prime,a,\diamond}$ given by
  $(u_1+\sqrt{-1}u_2,u_3, p_{u_1}-\sqrt{-1}p_{u_2}, p_{u_3})
   =:(\hat{z}_1,u_3,p_{\hat{z}_1},v_3)$.
 After a further change of coordinates
  $(\hat{z}_1,u_3,p_{\hat{z}_1},v_3)
   =(\hat{r}_1e^{\sqrt{-1}\hat{\theta}_1},u_3,
     e^{-\sqrt{-1}\hat{\theta}_1}
      (p_{\hat{r}_1}-\mbox{\footnotesize $\sqrt{-1}\,$}
                           \frac{p_{\hat{\theta}_1}}{\hat{r}_1}),
     p_{u_3})$,
 it can be expressed also as
  $$
   \begin{array}{l}
    \beta^{\prime}|
     _{T^{\ast}L^{\prime,a,\diamond}_{\{\hat{r}_1\ge R^{\prime}_0\}}}
                                                             \\[1.2ex]
    =\;
     \left(
      \left[
        (1+m^{-2}a^{2/m}\hat{r}_1^{2(1-m)/m})\,\hat{r}_1
        - \hat{r}_1^{-3}p_{\hat{\theta}_1}^2
      \right] d\hat{r}_1\wedge d\hat{\theta}_1
     \right.                                                 \\[1.2ex]
    \hspace{4em}
     \left.
      -\, \hat{r}_1^{-2}p_{\hat{\theta}_1}\,
            d\hat{r}_1\wedge dp_{\hat{r}_1}\,
      +\, p_{\hat{r}_1}\, d\hat{\theta}_1\wedge dp_{\hat{r}_1}\,
      +\, \hat{r}_1^{-2}p_{\hat{\theta}_1}\,
           d\hat{\theta}_1\wedge dp_{\hat{\theta}_1}\,
      -\, \hat{r}_1^{-1}\, dp_{\hat{r}_1}\wedge dp_{\hat{\theta}_1}
     \right)\wedge dp_{u_3}                                  \\[1.2ex]
    \hspace{1em}
      -\, m^{-1}a^{1/m}\, \hat{r}_1^{(1-m)/m}\,
            \cos\left(\frac{(1+m)\hat{\theta}_1}{m}\right)
            \left(
      p_{\hat{r}_1}\, d\hat{r}_1\wedge d\hat{\theta}_1
      +\, \hat{r}_1^{-1}\, d\hat{r}_1\wedge dp_{\hat{\theta}_1}\,
      +\, \hat{r}_1\, d\hat{\theta}_1\wedge dp_{\hat{r}_1}
            \right) \wedge dp_{u_3}                          \\[1.2ex]
    \hspace{1em}
     +\, m^{-1}a^{1/m}\, \hat{r}_1^{(1-m)/m}\,
            \sin\left(\frac{(1+m)\hat{\theta}_1}{m}\right)
            \left(
      2\hat{r}_1^{-1}p_{\hat{\theta}_1}\,
         d\hat{r}_1\wedge d\hat{\theta}_1\,
      -\, d\hat{r}_1\wedge dp_{\hat{r}_1}\,
      +\, d\hat{\theta}_1\wedge dp_{\hat{\theta}_1}
            \right) \wedge dp_{u_3}                          \\[1.2ex]
    \hspace{1em}
     +\, \left(
          p_{\hat{r}_1}\, d\hat{r}_1\wedge d\hat{\theta}_1\,
          +\, \hat{r}_1^{-1}\, d\hat{r}_1\wedge dp_{\hat{\theta}_1}\,
          -\, \hat{r}_1\, d\hat{\theta}_1\wedge dp_{\hat{r}_1}
         \right) \wedge du_3\,.
   \end{array}
  $$
 To show that
  $\|(\hat{\nabla}^{\prime,a})^k\beta^{\prime,a}\|_{C^0}$,
          $k=0,\,1,\,2,\,3$,
  is uniformly bounded on $U_{L^{\prime,a}}$,
 one only needs to show that
  $\|(\hat{\nabla}^{\prime,a})^k\beta^{\prime,a}\|_{C^0}$,
            $k=0,\,1,\,2,\,3$,
  is uniformly bounded on
  $U_{L^{\prime,a}}\cap \{|\hat{z}_1|\ge R^{\prime}_0\}$.

 First, recall Definition~1.2
  that the horizontal distribution $\hat{\nabla}^{\prime,a}$
   on $U_{L^{\prime,a}}$ is used to define a metric $\hat{h}^{\prime,a}$
   on $U_{L^{\prime,a}}$.
 From the explicit expression in Part (a) in coordinates
   $(\hat{z}_1,u_3,p_{\hat{z}_1},p_{u_3})$,
  $\hat{\nabla}^{\prime,a}=O(|\hat{z}_1|^{(2-3m)/m})$ and,
   hence, $\rightarrow 0$ uniformly on $U_{L^{\prime,a}}$
  as $|\hat{z}_1|\rightarrow \infty$.
 Thus, the $\|\cdot\|_{C^0}$-norm, with respect to $\hat{h}^{\prime,a}$,
  of tensor product of of elements in
  $\{d\hat{z}_1,\,d\bar{\hat{z}}_1,\, du_3,\,
     dp_{\hat{z}_1},\, dp_{\bar{\hat{z}}_1},\,dp_{u_3}\}$
  are all uniformly bounded on $U_{L^{\prime,a}}$.
 It remains to analyze the large-$|\hat{z}_1|$ behavior of
  the coefficients of $(\hat{\nabla}^{\prime,a})^k\beta^{\prime,a}$
   in terms of the basis from these tensor products.
 Re-write
  $$
   \begin{array}{rcl}
    \beta^{\prime}|
     _{T^{\ast}L^{\prime,a,\diamond}_{\{|\hat{z}_1|\ge R^{\prime}_0\}}}
     & =
     & \frac{\sqrt{-1}}{2}\left(
         (1+O(|\hat{z}_1|^{2(1-m)/m})\,d\hat{z}_1\wedge d\bar{\hat{z}}_1
          + dp_{\hat{z}_1}\wedge d\bar{p}_{\hat{z}_1}
                           \right)\wedge dp_{u_3}             \\[1.2ex]
    && +\, \left(
            O(|\hat{z}_1|^{(1-m)/m})\,
              d\hat{z}_1\wedge d\bar{p}_{\hat{z}_1}
              - O(|\bar{\hat{z}}_1|^{(1-m)/m})\,
                  d\bar{\hat{z}}_1\wedge dp_{\hat{z}_1}
           \right) \wedge dp_{u_3}                                \\[1.2ex]
    && -\, \Imaginary(d\hat{z}_1\wedge dp_{\hat{z}_1})\wedge du_3 \\[1.2ex]
    && \mbox{as}\;\;|\hat{z}_1|\rightarrow \infty\,.
   \end{array}
  $$
 From
  the fact that
   $\hat{\nabla}^{\prime,a}=O(|\hat{z}_1|^{(2-3m)/m})$
    as $|\hat{z}|_1\rightarrow \infty$ and
  the identities
   $$
    \begin{array}{rcl}
     (\hat{\nabla}^{\prime,a})^2_{e_{i_1},e_{e_2}} \beta^{\prime,a}
      & =
      & \hat{\nabla}^{\prime,a}_{e_{i_1}}
        \hat{\nabla}^{\prime,a}_{e_{i_2}} \beta^{\prime,a}\,
        -\, \hat{\nabla}^{\prime,a}
             _{\hat{\nabla}^{\prime,a}_{e_{i_1}}e_{e_2}}
            \beta^{\prime,a}\,,                              \\[1.2ex]
     (\hat{\nabla}^{\prime,a})^3_{e_{i_1},e_{i_2},e_{i_3}}
       \beta^{\prime,a}
      & =
      & \hat{\nabla}^{\prime,a}_{e_{i_1}}
        \hat{\nabla}^{\prime,a}_{e_{i_2}}
        \hat{\nabla}^{\prime,a}_{e_{i_3}} \beta^{\prime,a}\,
        -\, \hat{\nabla}^{\prime,a}
             _{\hat{\nabla}^{\prime,a}_{e_{i_1}}e_{i_2}}
            \hat{\nabla}^{\prime,a}_{e_{i_3}}  \beta^{\prime,a}\,
        -\, \hat{\nabla}^{\prime,a}_{e_{i_2}}
            \hat{\nabla}^{\prime,a}
             _{\hat{\nabla}^{\prime,a}_{e_{i_1}}e_{i_3}}
            \beta^{\prime,a}                                 \\[1.2ex]
     && -\, \hat{\nabla}^{\prime,a}_{e_{i_1}}
            \hat{\nabla}^{\prime,a}
             _{\hat{\nabla}^{\prime,a}_{e_{i_2}}e_{i_3}}
            \beta^{\prime,a}\,
        +\, \hat{\nabla}^{\prime,a}
             _{\hat{\nabla}^{\prime,a}_{
                 \hat{\nabla}^{\prime,a}_{e_{i_1}}e_{i_2}
                                        }e_{i_3}}
            \beta^{\prime,a}\,
        +\, \hat{\nabla}^{\prime,a}
             _{\hat{\nabla}^{\prime,a}_{e_{i_2}}
                 \hat{\nabla}^{\prime,a}_{e_{i_1}} e_{i_3}}
            \beta^{\prime,a}\,,
    \end{array}
   $$
  one observes that
   each time a covariant derivative is applied to a term,
   the large-$|\hat{z}_1|$ behavior of the coefficients of
    the resulting terms
    either remains the same or
    shifts from $O(|\hat{z}_1|^{\bullet})$ to
     $O(|\hat{z}_1|^{\bullet\,-1})$.
 It follows that all coefficients of
   $\beta^{\prime}|
     _{T^{\ast}L^{\prime,a,\diamond}_{\{|\hat{z}_1|\ge R^{\prime}_0\}}}$
  are uniformly bounded on $U_{L^{\prime,a}}$.
 Thus,
  $\|(\hat{\nabla}^{\prime,a})^k\beta^{\prime,a}\|_{C^0}$,
    $k=0,\,1,\,2,\,3$, are uniformly bounded on $U_{L^{\prime,a}}$.

 \bigskip

 \noindent
 $(c)$
 {\it Bounds for
      $\|(\hat{\nabla}^{\prime,a,t})^k\beta^{\prime,a,t}\|_{C^0}$,
      $k=0,\,1,\,2,\,3$, for $U_{L^{\prime,at^{m-1}}}$.}
 \hspace{1em}
 The scaling argument of Joyce (cf.\ [Jo3: III, Sec.~6.3]) applies here
  only better since the scaling involved in our situation is only partial.
 To proceed, first note that
  for a diffeomorphism $\tau: M_1\rightarrow M_2$ on manifolds and
   $k$-tensor $\alpha$ and vector fields $X_1,\,\cdots,\,X_k$ on $M_1$,
  $(\tau_!\alpha)(\tau_{\ast}X_1,\,\cdots,\, \tau_{\ast}X_k)
    = (\tau^{\ast}(\tau^{-1})^{\ast}\alpha)(X_1,\,\cdots,\, X_k)
    = \tau(X_1,\,\cdots,\,X_k)$.
 Let
  $(e_1, e_2, e_3, e_4, e_5, e_6)$ be the coordinate frame
   $(\partial_{u_1},\partial_{u_2},\partial_{u_3},
     \partial_{v_1},\partial_{v_2},\partial_{v_3})$
   on $Y^{\prime}$  and
  $(e^1, e^2, e^3, e^4, e^5, e^6)$ its dual coframe.
 Since
  $t^{-1}_!\hat{\nabla}^{\prime,a,t}=\hat{\nabla}^{\prime,a}$
   as horizontal distributions and
  $t^{-1}_!\beta^{\prime,a,t}=t^2\beta^{\prime,a}$ by construction,
 after passing to $U_{L^{\prime,at^{m-1}}}$ and for $t\in (0,1]$,
  $$
   \begin{array}{l}
    ((\hat{\nabla}^{\prime,a,t})^k_{e_{i_1},\,\cdots,\,e_{i_k}}
      \beta^{\prime,a,t})(e_{k+1},\, e_{k+2},\, e_{k+3})
      \hspace{2em}\mbox{at}\hspace{1em}
      \pt\;\in\; U_{L^{\prime,a,t}}                      \\[1.2ex]
    \hspace{1em} =\;
    ((t^{-1}_!\hat{\nabla}^{\prime,a,t})^k
        _{t^{-1}_{\ast}e_{i_1},\,\cdots,\,t^{-1}_{\ast}e_{i_k}}
     (t^{-1}_!\beta^{\prime,a,t}))
        (t^{-1}_{\ast}e_{k+1},\,
         t^{-1}_{\ast}e_{k+2},\, t^{-1}_{\ast}e_{k+3})
      \hspace{2em}\mbox{at}\hspace{1em}
      t^{-1}\cdot \pt\;\in\; U_{L^{\prime,a}}            \\[1.2ex]
    \hspace{1em} \le\;
    t^{-k}\,
    ((\hat{\nabla}^{\prime,a})^k_{e_{i_1},\,\cdots,\,e_{i_k}}
      \beta^{\prime,a})(e_{k+1},\, e_{k+2},\, e_{k+3})
   \end{array}
  $$
  as sections in the contraction
   $((\otimes_kT^{\ast}Y^{\prime})\otimes\Omega^3(Y^{\prime}))
     \otimes(\otimes_{k+3}T_{\ast}Y^{\prime})
    \rightarrow C^{\infty}(Y^{\prime})$ through evaluation,
   over submanifolds
   $U_{L^{\prime,a,t}},\, U_{L^{\prime,a}}\subset Y^{\prime}$.
 Here, we used the fact that
  exactly one of $e_{k+1}$, $e_{k+2}$, $e_{k+3}$
  must be in $\{e_3,e_6\}$ for the above contraction to be non-zero.
 Since
   there exists a constant $C^{\prime}_1>0$ such that
    $\|e^i\|_{C^0_{U_{L^{\prime,a,t}}}}
      \le C^{\prime}_1\,t\,\|e^i\|_{C^0_{U_{L^{\prime,a}}}}$
     for $i=1,2,4,5$ and
    $=\|e^i\|_{C^0_{U_{L^{\prime,a}}}}$ for $i=3,6$,
  one has
  $$
   \|(\hat{\nabla}^{\prime,a,t})^k\beta^{\prime,a,t}\|
                                  _{C^0_{U_{L^{\prime,a,t}}}}\;
   \le\; C^{\prime}_2\,t^{-k}\,
         \|(\hat{\nabla}^{\prime,a})^k\beta^{\prime,a}\|
                                  _{C^0_{U_{L^{\prime,a}}}}\;
   \le\; A^{\prime}_4\,t^{-k}\,,
  $$
  for some constants $C^{\prime}_2>0$ and $A^{\prime}_4>0$,
  from the uniform bound for
      $\|(\hat{\nabla}^{\prime,a})^k\beta^{\prime,a}\|_{C^0}$,
      $k=0,\,1,\,2,\,3$, for $U_{L^{\prime,a}}$
  in Part (b).
 This proves the proposition.

\end{proof}

\bigskip

\section{Immersed Lagrangian deformations of a simple normalized
 branched covering of a special Lagrangian 3-sphere
 in a Calabi-Yau 3-fold and their deviation from Joyce's criteria.}

We construct in Sec.~3.1
 a natural family of immersed Lagrangian deformations of
  a simple normalized branched covering of a special Lagrangian 3-sphere
  in a Calabi-Yau 3-fold and
 compute in Sec.~3.2 - Sec.~3.4 their deviation from Joyce's criteria.

\bigskip

\subsection{Immersed Lagrangian deformations of a branched covering
            of a special Lagrangian 3-sphere in a Calabi-Yau 3-fold.
            }

Let $Z_0\simeq S^3\subset Y$ be a special Lagrangian $3$-sphere
 in a Calabi-Yau $3$-fold $Y=(Y, J, \omega, \Omega)$.
It follows from [McL] that $Z_0$ is rigid.
The canonical inclusion $J\cdot T_{\ast}Z_0\subset (T_{\ast}Y)|_{Z_0}$
 gives a distribution on $Y$ along $Z_0$ that is perpendicular to
 the canonical inclusion $T_{\ast}Z_0\subset (T_{\ast}Y)|_{Z_0}$.
Let $\Phi_{Z_0}:U_{Z_0}\rightarrow Y$ be a symplectomorphism
 from a neighborhood $U_{Z_0}$ of the zero-section of $T^{\ast}Z_0$
 to a neighborhood of $Z_0$ in $Y$ such that
 its restriction to the zero-section is $\id_{Z_0}:Z_0\rightarrow Z_0$
  and
 the embedding $\Phi_{Z_0,\ast}: T_{\ast}U_{Z_0}\rightarrow T_{\ast}Y$
  sends $T_0(\pi_{Z_0}^{-1}(z))$ to $J\cdot T_zZ_0$ for $z\in Z_0$.
Here $\pi_{Z_0}: U_{Z_0}\rightarrow Z_0$ is the restriction of
 the bundle map $T^{\ast}Z_0\rightarrow Z_0$ to $U_{Z_0}$
 and one endows $U_{Z_0}$ with a Calabi-Yau structure via $\Phi_{Z_0}$.

\bigskip

\begin{flushleft}
{\bf Simple normalized branched coverings of a sL 3-sphere
     in a Calabi-Yau 3-fold.}
\end{flushleft}
Let
 \begin{itemize}
  \item[$\cdot$]
   $X$ be a closed oriented $3$-manifold;

  \item[$\cdot$]
   $f:X\rightarrow Z_0$ be a smooth, orientation-preserving, finite,
    branched covering of $Z_0$,
    with branch locus $\Gamma\subset Z_0$ (downstairs)
    and $\tilde{\Gamma}\subset X$ (upstairs);

  \item[$\cdot$]
   $\Gamma=\amalg_{i=1}^n\Gamma_i$ and
   $\tilde{\Gamma}=\amalg_{j=1}^{\tilde{n}_0}\tilde{\Gamma}_j$
    be the decomposition of $\Gamma$ and $\tilde{\Gamma}$
    into connected components.
 \end{itemize}

\begin{definition}
{\bf [simple normalized branched covering].} {\rm
 $f:X\rightarrow Z_0$
  is called a {\it simple normalized} branched covering of $Z_0$ in $Y$
 if it satisfies in addition the following conditions:
 \begin{itemize}
  \item[$\cdot$]
   [{\it simple}$\,$]\hspace{1em}
   Each $\Gamma_i$, $\tilde{\Gamma}_j$
    is a smooth $1$-submanifold of $X$ isomorphic to a circle $S^1$ and
   $f$ maps each $\tilde{\Gamma}_j$ diffeomorphically
    to some $\Gamma_i$;

  \item[$\cdot$]
   [{\it normalized}$\,$]\hspace{1em}
   When $f:\tilde{\Gamma}_j\rightarrow \Gamma_i$,
   there exist
    a tubular neighborhood
     $\nu_X(\tilde{\Gamma}_j)$ of $\tilde{\Gamma}_j$ in $X$,
    a tubular neighborhood of $\nu_{Z_0}^j(\Gamma_i)$
     of $\Gamma_i$ in $Z_0$,
    coordinates $(x_1,x_2,x_3)\in {\Bbb R}^2\times({\Bbb R}/l_i)$
     on $\nu_{X}(\tilde{\Gamma}_j)$, and
    coordinates $(u_1,u_2,u_3)\in {\Bbb R}^2\times({\Bbb R}/l_i)$
     on $\nu_{Z_0}^j(\Gamma_i)$,
     where $l_i$ is the length of $\Gamma_i$ in $Y$,
    such that the following holds:
    \begin{itemize}
     \item[$\cdot$]
      The restriction of
       $(\frac{\partial}{\partial u_1}, \frac{\partial}{\partial u_2}
         \frac{\partial}{\partial u_3})$ to $\Gamma_i$
       is an orthonormal frame along $\Gamma_i$ with
      $\frac{\partial}{\partial u_3}$ tangent to $\Gamma_i$.

     \item[$\cdot$]
      The restriction
       $f:\nu_X(\tilde{\Gamma}_j)\rightarrow \nu_{Z_0}^j(\Gamma_i)$
       is given by
       $$
        (u_1,u_2,u_3)\;
         =\; f(x_1,x_2,x_3)\;
         =\; \left(
              \Real\left((x_1+\sqrt{-1}x_2)^{m_j}\right),\,
              \Imaginary\left((x_1+\sqrt{-1}x_2)^{m_j}\right)\,;\, x_3
             \right)
       $$
       for some $m_j\in {\Bbb Z}_{\ge 2}$.
    \end{itemize}
   $m_j$ is called the {\it degree/multiplicity/order of $f$
   around $\tilde{\Gamma}_j$}.
 \end{itemize}
}\end{definition}

\begin{remark} {$[$weaker condition$]$.} {\rm
 In the above definition, the requirement that
  \begin{itemize}
   \item[]
   `The restriction of
     $(\frac{\partial}{\partial u_1}, \frac{\partial}{\partial u_2}
          \frac{\partial}{\partial u_3})$ to $\Gamma_i$
        is an orthonormal frame along $\Gamma_i$ with
       $\frac{\partial}{\partial u_3}$\\
       $\mbox{$\,\,$}$tangent to $\Gamma_i$.'
  \end{itemize}
  is only for the simplicity of the presentation in this note.
 It can be replaced by the weaker requirement that
  \begin{itemize}
   \item[]
   `The restriction of $\frac{\partial}{\partial u_3}$ to $\Gamma_i$
    is tangent to $\Gamma_i$.'
  \end{itemize}
 Furthermore,
  since the construction below is local in nature and
   the branched covering map has finite degree,
  the condition that
  \begin{itemize}
   \item[]
   `$f$ maps each $\tilde{\Gamma}_j$ diffeomorphically to some $\Gamma_i$'
  \end{itemize}
  can be weakened to the condition
  \begin{itemize}
   \item[]
   `$f$ maps each $\tilde{\Gamma}_j$ to some $\Gamma_i$
    as a smooth finite covering map'.
  \end{itemize}
}\end{remark}

\bigskip

\begin{flushleft}
{\bf Auxiliary flat Lagrangian neighborhoods
     $Y^{\prime,j}$, $j=1,\,\cdots,\,\tilde{n}_0$,
     for the branch locus $\Gamma\subset Y$ of $f$.}
\end{flushleft}
The inclusion $\nu_{Z_0}^j(\Gamma_i)\subset Z_0$ induces an inclusion
 $T^{\ast}\nu_{Z_0}^j(\Gamma_i)\subset T^{\ast}Z_0$.
Let
 $U_{\nu_{Z_0}^j(\Gamma_i)} := U_{Z_0}\cap T^{\ast}\nu_{Z_0}^j(\Gamma_i)$
  and
 $\Phi_{\nu_{Z_0}(\Gamma_i)}:= \Phi_{Z_0}|_{U_{\nu_{Z_0}^j(\Gamma_i)}}$.
Shrinking all these tubular neighborhoods if necessary,
 one has then a symplectomorphism between tubular neighborhoods
 $$
  \Phi_{\nu_{Z_0}^j(\Gamma_i)}\;:\;
   U_{\nu_{Z_0}^j(\Gamma_i)}\;
      \longrightarrow\; \nu_Y^j(\Gamma_i)\,\subset\, Y\,,
 $$
 whose restriction to the zero-section is the identity map
  $\nu_{Z_0}^j(\Gamma_i)\rightarrow \nu_{Z_0}^j(\Gamma_i)$.

Recall the coordinates $(u_1,u_2,u_3)$ on $\nu_{Z_0}^j(\Gamma_i)$
 and the construction in Sec.~2.2.
Identify the canonical coordinates
 $(u_1, u_2, u_3, p_{u_1}, p_{u_2}, p_{u_3})$ on
 $U_{\nu_{Z_0}^j(\Gamma_i)}\subset T^{\ast}\nu_{Z_0}^j(\Gamma_i)$ here
 with the coordinates $(u_1, u_2, u_3, v_1, v_2, v_3)$ on $Y^{\prime}$
 there.
Then the Calabi-Yau structure
 $(J^{\prime}, \omega^{\prime},\Omega^{\prime})$ on $Y^{\prime}$
 in Sec.~2.2 induces a Calabi-Yau structure,
  denoted also by $(J^{\prime},\omega^{\prime},\Omega^{\prime})$,
 on $U_{\nu_{Z_0}^j(\Gamma_i)}$ that is flat.
Denote
 $Y^{\prime,j} := U_{\nu_{Z_0}^j(\Gamma_i)}$  and
 $\Phi_{\nu_{Z_0}^j(\Gamma_i)}: U_{\nu_{Z_0}^j(\Gamma_i)}
                                \rightarrow \nu_Y^j(\Gamma_i)$
  by
  $$
   \Upsilon^j\;:\;
    (Y^{\prime,j},\omega^{\prime})\;
       \longrightarrow\; (\nu_Y^j(\Gamma_i), \omega)\,\subset\, Y\,.
  $$
Then
 the K\"{a}hler property of a Calabi-Yau structure and
  the special Lagrangian property with respect to a calibration
  imply that
 $$
  \Upsilon^{j\ast}(J,\Omega)|_{\Gamma_i}\;
   =\; (J^{\prime}, \Omega^{\prime})|_{\Gamma_i}\,.
 $$
In this sense,
 $(Y^{\prime,j},J^{\prime}, \omega^{\prime}, \Omega^{\prime})$,
  denoted collectively also by $Y^{\prime,j}$,
 is an {\it infinitesimal flat approximation} of
 $(\nu_Y^j(\Gamma_i), J,\omega,\Omega)$
 ($=: \nu_Y^j(\Gamma_i)$ collectively) in $Y$
 via $\Upsilon^j$  and
one can identify $Y^{\prime,j}$ with a tubular neighborhood of
 the zero-section of the orthogonal complement
 $(\frac{\partial}{\partial u_3}|_{\Gamma_i})^{\perp}$
 of the nowhere-zero section
 $\frac{\partial}{\partial u_3}|_{\Gamma_i}$ in $(T_{\ast}Y)|_{\Gamma_i}$,
 with the induced flat Calabi-Yau structure.

\bigskip

\begin{flushleft}
{\bf Immersed Lagrangian deformations $f^t$ of $f$ from gluing.}
\end{flushleft}
The restriction
 $f:\nu_X(\tilde{\Gamma}_j)\rightarrow \nu_{Z_0}^j(\Gamma_i)\subset Y$
 defines a special Lagrangian map
 $$
  \psi^j\; :\; \nu_X(\tilde{\Gamma}_j)\;
          \longrightarrow\; \nu_{Z_0}^j(\Gamma_i)\subset Y^{\prime,j}
 $$
 such that $f= \Upsilon^j\circ\psi^j$ on $\nu_X(\tilde{\Gamma}_j)$.
We'll glue the immersed Lagrangian deformation $\psi^{j,a_jt^{m_j-1}}$
 of $\psi^j$ as constructed in Sec.~2.2 to $f$
 to give an immersed Lagrangian deformation $f^t$ of $f$.

The following class of cutoff functions with their first three derivatives
 bounded in the best possible manner
 is the basis of our gluing construction and some later estimates:

\begin{lemma}
{\bf [cutoff function].}
 Given $\delta>0$ and $R_0>0$,
 let
  $t\in (0,\delta)$ and
  $0<b_1^t<b_2^t <R_0$ be constants that depend smoothly on $t$
  such that $b_1^t,\, b_2^t,\, b_1^t/b_2^t\rightarrow 0$
   when $t\rightarrow 0$.
 Then, there exist
   smooth functions $\chi^t:(0,R_0)\rightarrow [0,1]$
    that depend smoothly on $t$ as well and
   a constant $C_0>0$, independent of $t$,
  such that the following hold:
  \begin{itemize}
   \item[$\cdot$]
    $\chi^t:(0,b_1^t]\rightarrow \{1\}$.

   \item[$\cdot$]
    $\chi^t:[b_2^t,R_0)\rightarrow \{0\}$.

   \item[$\cdot$]
    For $t$ small enough,
     $\left|\frac{d^k}{dr^k}\chi^t(r)\right|
       \le C_0\cdot (b_2^t)^{-k}$, for $k=1,\,2,\,3$.
  \end{itemize}
\end{lemma}

\begin{proof}
 For $t$ small enough,
  one may assume that $0<b_1^t<\frac{1}{8}\cdot b_2^t$.
 Consider the following piecewise linear continuous function
  $$
   \hat{\chi}^t(r)\;=\; \left\{
    \begin{array}{llllll}
     0
      &&& \mbox{for}  && 0 < r < b_1^t\,,    \\[2ex]
     \frac{a_1}{r_1-b_1^t}\,(r-b_1^t)
      &&& \mbox{for}  && b_1^t\le r < r_1\,, \\[2ex]
     \frac{-a_1}{r_2-r_1}\,(r-r_1)\,+\, a_1
      &&& \mbox{for}  && r_1\le r < r_2\,,   \\[2ex]
     \frac{a_2}{r_3-r_2}\,(r-r_2)
      &&& \mbox{for}  && r_2\le r < r_3\,,   \\[2ex]
     \frac{-a_2}{r_4-r_3}\,(r-r_3)\,+\, a_2
      &&& \mbox{for}  && r_3\le r < r_4\,,   \\[2ex]
     \frac{a_3}{r_5-r_4}\,(r-r_4)
      &&& \mbox{for}  && r_4\le r < r_5\,,   \\[2ex]
     \frac{-a_3}{b_2^t-r_5}\,(r-r_5)\,+\, a_3
      &&& \mbox{for}  && r_5\le r < b_2^t\,, \\[2ex]
     0
      &&& \mbox{for}  && b_2^t\le r < R_0\,.
    \end{array}
             \right.
  $$
  which depends on the parameters
   $a_1, a_3<0;\, a_2>0;\, b_1^t< r_1 <  r_2 < r_3 < r_4 < r_5 < b_2^t$
    with $r_1$ (resp.\ $r_2,\, r_3,\, r_4,\, r_5$)
    in a small neighborhood of $\frac{1}{8}\,b_2^t$
    (resp.\ $\frac{1}{4}\,b_2^t,\, \frac{1}{2}\,b_2^t,\,
             \frac{3}{4}\,b_2^t,\, \frac{7}{8}\,b_2^t$).
 By an appropriate adjustment of these parameters and
  a smoothing $\hat{\chi}^{t,\sim}$ of $\hat{\chi}^t$
   in the $C^{\infty}$-topology,
 one can choose $\chi^t$ as required by taking
  $$
   \chi^t(r)\;
    =\; 1\,
        +\, \int_0^r \int_0^{r^{\prime}} \int_0^{r^{\prime\prime}}\,
             \hat{\chi}^{t,\sim}(r^{\prime\prime\prime})\,
             dr^{\prime\prime\prime} dr^{\prime\prime} dr^{\prime}\,.
  $$
\end{proof}

We'll denote $d\chi^t/dr$ and $d^2\chi^t/dr^2$
 also by $\dot{\chi}^t$ and $\ddot{\chi}^t$ respectively.

\begin{definition}
{\bf [immersed Lagrangian deformations $f^t$ of $f$ from gluing].} {\rm
 Let
  $$
   \begin{array}{rclccrcl}
    X^{\diamond}   & := & X-\tilde{\Gamma}\,, &&
     & Z_0^{\diamond} & := & Z_0-\Gamma\,, \\[1.2ex]
    \nu_X(\tilde{\Gamma}_j)^{\diamond}
     & := & \nu_X(\tilde{\Gamma}_j)-\tilde{\Gamma}_j\;
            =\; \nu_X(\tilde{\Gamma}_j)\cap X^{\diamond}\,, &&
     & \nu_{Z_0}^j(\Gamma_i)^{\diamond}
     & := & \nu_{Z_0}^j(\Gamma_i)^{\diamond}\;
            =\; \nu_{Z_0}^j(\Gamma_i)\cap Z_0^{\diamond}\,.
   \end{array}
  $$
 The restriction $f^{\diamond}:X^{\diamond}\rightarrow Z_0^{\diamond}$
  of $f$ to $X^{\diamond}$ is a covering map and, hence,
  induces a covering map
  $$
   f^{\diamond}_{!}\; :\; T^{\ast}X^{\diamond}\;
     \longrightarrow\; T^{\ast}Z_0^{\diamond}
  $$
  that is a local symplectomorphism.
 Recall Sec.~2.2 and the notations therein.
 For a pair $(j,i)$ with $f(\tilde{\Gamma}_j)=\Gamma_i$,
 let $R_0>0$ and $a_j>0$ be small enough so that
  \begin{itemize}
   \item[$\cdot$]
    the solid torus $\{|\hat{z}_1|\le R_0\}$ around $\Gamma_i$ in $Z_0$
     is contained in $\nu_{Z_0}^j(\Gamma_i)$  and,
    for notational convenience,
    we shrink $\nu_{Z_0}^j(\Gamma_i)$
     so that they are the same from now on,

   \item[$\cdot$]
    the solid torus $\{ |x_1+\sqrt{-1}x_2|\le R_0^{1/m_j}\}$ around
    $\tilde{\Gamma}_j$ in $X$ is contained in $\nu_X(\tilde{\Gamma}_j)$
    and we now shrink $\nu_X(\tilde{\Gamma}_j)$
     so that they are the same from now on,

   \item[$\cdot$]
    the smooth embedded special Lagrangian submanifold with boundary
     $$
      L^{\prime,a_j}_{R_0}\;
       :=\; \{\, |\hat{z}_1|\le R_0\,,\;
                 a_j\hat{z}_1-\hat{z}_2^{m_j}=0\,,\; v_3=0\,\}\;
       \subset\;  Y^{\prime}
     $$
     is contained in $Y^{\prime,j}$.
  \end{itemize}
  Recall the graph $\Gamma(\alpha^j)$ of the associated exact $1$-form
   $$
    \alpha^j\;
      =\; dh^j\;
     :=\; \Real\left(
          \frac{m_j}{m_j+1}\,a_j^{1/m_j}\,
                    d\left((x_1+\sqrt{-1}x_2)^{m_j+1}\right)
           \right)
   $$
   on $\nu_X(\tilde{\Gamma}_j)$
   whose restriction over $\nu_X(\tilde{\Gamma}_j)^{\diamond}$ is mapped to
    $L^{\prime, a}_{R_0}\cap \nu_{Z_0}^j(\Gamma_i)^{\diamond}$
    under $f^{\diamond}_!$.
   (Cf.\ Lemma~2.2.2.)
 While $f^{\diamond}_!$ does not extend to fibers of $T^{\ast}X$
  over $\tilde{\Gamma}$,
 it follows from the explicit study in Sec.~2.2 that
  the restriction of $f^{\diamond}_!$
    on $\Gamma(\alpha^j)|_{\nu_X(\tilde{\Gamma}_j)^{\diamond}}$
   extends to the whole $\Gamma(\alpha^j)$ and
  defines a smooth Lagrangian embedding
   $$
    \psi^{j,a_j}\; :\;  \nu_X(\tilde{\Gamma}_j)\;
                                \hookrightarrow\;  Y^{\prime,j}
   $$
   with the image $L^{\prime,a_j}_{R_0}$ and satisfying
   $\pi\circ \psi^{j,a_j} = f|_{\nu_X(\tilde{\Gamma}_j)}$,
   where $\pi:T^{\ast}Z_0\rightarrow Z_0$ is the bundle map.

 Let
   $t\in (0,\delta)$, $0< b_1^t< b_2^t< R_0$, and
   $\chi^t:(0,R_0)\rightarrow [0,1]$ with $r=|(x_1+\sqrt{-1}x_2)^{m_j}|$
  be as in Lemma~3.1.3.
 Recall the partial scaling
  $t\cdot L^{\prime,a_j} := L^{\prime, a_jt^{m_j-1}}$ of $L^{\prime,a_j}$
   in $Y^{\prime}$;
  cf.\ Notation~2.2.3.
 The associated $1$-form on $\nu_X(\tilde{\Gamma}_j)$
  is thus $t^{(m_j-1)/m_j}\alpha^j = d(t^{(m_j-1)/m_j}h^j)$\,.
 Define
  $$
   h^{j,t}\;
    :=\; \chi^t\cdot (t^{(m_j-1)/m_j}\cdot h^j)
    \hspace{2em}\mbox{and}\hspace{2em}
   \alpha^{j,t}\;=\; dh^{j,t}\,.
  $$
 Then,
  since the two graphs
   $\Gamma(\alpha^{j,t})\,,\,
    \Gamma(t^{(m_j-1)/m_j}\cdot\alpha^j)\,
    \subset\,T^{\ast}\nu_X(\tilde{\Gamma}_j)$ of $1$-forms
   are identical over
   $\{|x_1+\sqrt{-1}x_2|^{m_j}\le b_1^t\}\subset \nu_X(\tilde{\Gamma}_j)$,
  the restriction of $f^{\diamond}_!$
    on $\Gamma(\alpha^{j,t})|_{\nu_X(\tilde{\Gamma}_j)^{\diamond}}$
   extends to the whole $\Gamma(\alpha^{j,t})$ as well.
 Since
   $\Gamma(t^{(m_j-1)/m_j}\cdot\alpha^j)$ defines also
    a smooth Lagrangian embedding of $\nu_X(\tilde{\Gamma}_j)$ under
    the extension of $f^{\diamond}_!$
    and
   $f$ is a smooth Lagrangian immersion
    of $\nu_X(\tilde{\Gamma}_j)^{\diamond}$,
  $f^{\diamond}_!$ defines now a smooth Lagrangian immersion:
   (following the notation of Sec.~2.2)
   $$
    \psi^{j,a_jt^{m_j-1}}\; :\;  \nu_X(\tilde{\Gamma}_j)\;
                            \hookrightarrow\;  Y^{\prime,j}
   $$
   with the image $L^{\prime,a_jt^{(m_j-1)/m_j}}_{R_0}$ and
   satisfying also
    $\pi\circ \psi^{j,a_jt^{m_j-1}} = f|_{\nu_X(\tilde{\Gamma}_j)}\,$.
 After the post-composition with $\Upsilon^j: Y^{\prime,j}\rightarrow Y$,
  one has then a smooth immersion
   $$
    \Upsilon^j\circ \psi^{j,a_jt^{m_j-1}}\; :\;
     \nu_X(\tilde{\Gamma}_j)\; \longrightarrow\; Y
   $$
  of Lagrangian submanifold with boundary.
 Since
  $$
   \chi^t|_{[b_2^t,R_0)}\equiv 0
    \hspace{1em}\mbox{and}\hspace{1em}
   (\Upsilon^j\circ \psi^{j,a_jt^{m_j-1}})|
               _{\{ b_2^t\le |(x_1+\sqrt{-1}x_2)^{m_j}| < R_0\}}
   \equiv f|_{\{ b_2^t\le |(x_1+\sqrt{-1}x_2)^{m_j}| < R_0\}}\,,
  $$
 it follows that
  $$
   \coprod_{j=1}^{\tilde{n}_0}\,(\Upsilon^j\circ \psi^{j,a_jt^{m_j-1}})\;
    :\; \coprod_{j=1}^{\tilde{n}_0}\, \nu_X(\tilde{\Gamma}_j)\;
    \longrightarrow\; Y
  $$
  can be extended by $f$ on $X-\coprod_j\,\nu_X(\tilde{\Gamma}_j)$
  to a smooth Lagrangian immersion
  $$
   f^t\; :\; N^t\,=\,X\; \longrightarrow\; Y\,.
  $$
 In other words,
  recall
    the Lagrangian neighborhood $\Phi_{Z_0}:U_{Z_0}\rightarrow Y$
     of $Z_0\subset Y$ and
    the projection map $\pi_{Z_0}:U_{Z_0}\rightarrow Z_0$
   at the beginning of this subsection.
 Then,
  the smooth function
   $$
    \coprod_{j=1}^{\tilde{n}_0}\,h^{j,t}\; :\;
     \coprod_{j=1}^{\tilde{n}_0}\nu_X(\tilde{\Gamma}_j)\;
      \longrightarrow\; {\Bbb R}
   $$
   extends to a smooth function
   $$
    h^t\; :\; X\; \longrightarrow\; {\Bbb R}
   $$
   by $0$;
  $f^{\diamond}_!$ extends to $f_!$ on the graph $\Gamma(dh^t)$ of $dh^t$
   in $T^{\ast}X$; and
  $$
   \,\,f^t\;=\; \Phi_{Z_0}\circ f_!\circ dh^t\,,
  $$
   where $dh^t$ is regarded
    as a section $X\rightarrow T^{\ast}X$ of $T^{\ast}X$ over $X$.
}\end{definition}

By construction,
 $f^t\rightarrow f=:f^0$, as $t\rightarrow 0$,
  both in the sense of currents
  and in the sense of $C^{\infty}$-topology on any compact subset
   of $X^{\diamond}$, and
 $\pi_{Z_0}\circ f^t=f$ for all $t\in (0,\delta)$.

\begin{notation}
{\bf [pull-back metric].} {\rm
 Let $g_Y$ be the Calabi-Yau metric on $Y$ determined by $(J,\omega)$.
 Denote by $g^t$ the pull-back metric $(f^t)^{\ast}g_Y$ on $N^t=X$.
}\end{notation}

\begin{remark}
{$[$expression in $Y^{\prime,j}$$]$.} {\rm
 By construction, the only difference of $f^t$ and $f=: f^0$
  lies in $\nu_{X}(\tilde{\Gamma})$, whose image in $Y$ lies in
  $\cup_j(\Upsilon^j(Y^{\prime,j}))$.
 In terms of the complex-real coordinates
  $(u_1+\sqrt{-1}u_2,u_3,v_1-\sqrt{-1}v_2,v_3)
   = (\hat{z}_1,u_2,\hat{z}_2,v_3)
   = (r_1e^{\sqrt{-1}\theta_1}, u_3, r_2e^{\sqrt{-1}\theta_2},v_3)$
  on $Y^{\prime,j}$ and, hence, on $\Upsilon^j(Y^{\prime,j})$,
 the immersed image $f^t(N^t)\cap \Upsilon^j(Y^{\prime,j})$
  is given by
  \begin{eqnarray*}
   \lefteqn{f^t(N^t)\cap \Upsilon^j(Y^{\prime,j})  }      \\[1,2ex]
    &&\hspace{-2ex} =\;
     \left\{
      (r_1e^{\sqrt{-1}\theta_1},\, u_3,\,
       r_2e^{\sqrt{-1}\theta_2},\, 0)\,
     \left|
      \begin{array}{l}
       \cdot\; r_2\;
        =\; a^{1/m_j}\,t^{(m_j-1)/m_j}\,
            \left[\frac{m_j}{m_j+1}\,\dot{\chi}^t(r_1)\,
                  +\,\chi^t(r_1)\,r_1^{1/m_j}\right]   \\[1.2ex]
       \cdot\; m_j\,\theta_2\; =\; \theta_1\;\;\;\;
                              (\,\mbox{\rm mod}\; 2\pi\,)
      \end{array}\right.\hspace{-1ex}\right\}.
  \end{eqnarray*}
}\end{remark}

\begin{notation}
{\bf [$f^t$ in three parts].} {\rm
 With the notation in this subsection,
 let
  $$
   \begin{array}{lcl}
    P_j^t & :=
    & \{(x_1,x_2,x_3)\in \nu_X(\tilde{\Gamma}_j)\; :\;
      0\,\le\,|(x_1+\sqrt{-1}x_2)^{m_j}|\,\le\, b_1^t \}\,, \\[1.2ex]
    Q_j^t & :=
    & \{(x_1,x_2,x_3)\in \nu_X(\tilde{\Gamma}_j)\; :\;
      b_1^t\,\le\,|(x_1+\sqrt{-1}x_2)^{m_j}|\,\le\, b_2^t \}\,,
     \\[1.2ex]
    K^t   & := & X-\coprod_{j=1}^{\tilde{n}_0}(P_j\cup Q_j)\,.
   \end{array}
  $$
 Then, from the gluing construction of $f^t$,
  $$
   f^t\;
    =\;  \left( \cup_{j=1}^{\tilde{n_0}}
                (f^t|_{P_j^t}\cup f^t|_{Q_j^t} ) \right)
         \bigcup f^t|_{K^t}\,.
  $$
}\end{notation}

\bigskip

\subsection{Estimating {\it Im}$\,\Omega|_{N^t}$.}

We now estimate the Sobolev norms of $\Imaginary\Omega|_{N^t}$
 in Criterion (i) of Theorem 1.1.
The discussion is based upon Taylor's formula  and
 a finite-dimensional nature of the problem.

\bigskip

\begin{flushleft}
{\bf Taylor's formula.}
\end{flushleft}
Let
 $f$ be an ${\Bbb R}$-valued function
  on an open subset $S\subset {\Bbb R}^m$ and
 ${\mathbf y}:= (y_1,\,\cdots\,,y_m)$ be the coordinates on ${\Bbb R}^m$.
For ${\mathbf a}\in S$ and ${\mathbf t}\in {\Bbb R}^m$,
 if all $l$-th order partial derivatives of $f$ exist at ${\mathbf a}$,
 then write
  $$
    f^{(l)}({\mathbf a};{\mathbf t})\;
    :=\; \sum_{j_l=1}^m\,\cdots\,\sum_{j_1=1}^m\,
          \frac{\partial^l\,f}
               {\partial y_{j_l}\,\cdots\,\partial y_{j_1}}\,
             ({\mathbf a})\,
          t_{j_1}\,\cdots\,t_{j_l}\,.
  $$

\begin{theorem}
{\bf [Taylor's formula with remainder].}
{\rm (E.g.\ [Ap: Theorem~12.14].)}
 Assume that $f$ and all its partial derivatives of order $\le l$
  are differentiable at each point of an open set $S$ in ${\Bbb R}^m$.
 If ${\mathbf a}$ and ${\mathbf b}$ are two points of $S$ such that
  the line segment $\overline{{\mathbf a},\!{\mathbf b}}$ in ${\Bbb R}^m$
   that connects ${\mathbf a}$ and ${\mathbf b}$
  is contained in $S$,
 then there is a point
  ${\mathbf c}\in \overline{{\mathbf a},\!{\mathbf b}}$
  such that
  $$
   f({\mathbf b})\;
    =\; \sum_{i=0}^l\,
          \frac{1}{i!}\,
            f^{(i)}({\mathbf a}; {\mathbf b}-{\mathbf a})\:
        +\: \frac{1}{(l+1)!}\,
              f^{(l+1)}({\mathbf c}; {\mathbf b}-{\mathbf a})\,.
  $$
\end{theorem}

\bigskip

\begin{flushleft}
{\bf Oriented-Lagrangian Grassmannian bundles,
     prolongation of Lagrangian immersions, and calibrations.}
\end{flushleft}
Let
 $(Y,J,\omega,\Omega)$ be a Calabi-Yau $m$-fold,
  denoted collectively by $Y$, and
 $\Gr^{L^+}(T_{\ast}Y)$ be the oriented-Lagrangian Grassmannian bundle
  over $Y$, whose fiber over $y\in Y$ is given by
  the Grassmannian manifold $\Gr^{L^+}(T_yY)$
  of oriented Lagrangian subspaces of $T_yY$.
By construction,
 a Lagrangian immersion $f:X\rightarrow Y$
  from an oriented $m$-manifold $X$ to $Y$ has a unique lifting
  $$
   \xymatrix{
      &&  \Gr^{L^+}(T_{\ast}Y)\ar[d]^-{\pi_Y} \\
    X\ar[rru]^-{Gr^{L^+}\!\!f}\ar[rr]_-f && Y  &\hspace{-4em} ,
   }
  $$
  defined by $(\Gr^{L^+}\!\!f)(x)=[f_{\ast}(T_xX)]\in \Gr^{L^+}(T_{f(x)}Y)$
  for $x\in X$, where $f_{\ast}(T_xX)$
  is equipped with an orientation from that of $T_xX$
  via the isomorphism $f_{\ast}: T_xX\rightarrow f_{\ast}(T_xX)$.

\begin{definition}
{\bf [prolongation of Lagrangian immersion].} {\rm
 $\Gr^{L^+}\!\!f:X\rightarrow \Gr^{L^+}(T_{\ast}Y)$ is called
 the {\it prolongation} of the Lagrangian map $f:X\rightarrow Y$
 to $\Gr^{L^+}(T_{\ast}Y)$.
}\end{definition}

The holomorphic $m$-form $\Omega$ on $Y$ defines a map
 $e^{\sqrt{-1}\alpha}:\Gr^{L^+}(T_{\ast}Y)
                      \rightarrow U(1)\subset {\Bbb C}^{\ast}$
 by $[L]\mapsto \Omega|_{L}/\vol_L$,
 where $\vol_L$ is the volume-form on $L$ induced by the metric on $Y$.
The imaginary part $\varepsilon:= \sin\alpha$ of $e^{\sqrt{-1}\alpha}$
 defines a smooth function on $\Gr^{L^+}(T_{\ast}Y)$.
This defines in turn
 a smooth section $d\varepsilon$ of $T^{\ast}(\Gr^{L^+}(T_{\ast}Y))$
  and
 a smooth function
  $|d\varepsilon|^2:=|(\Gr^{L^+}\!\!f)^{\ast}d\varepsilon|^2$ on $X$,
    using the pullback metric tensor on $X$ under $f$.

\bigskip

\begin{flushleft}
{\bf Local charts on $\Gr^{L^+}(T_{\ast}Y)$ and prolongations as 2-jets.}
\end{flushleft}
Note that any Lagrangian tangent subspace of $Y$ is tangent
 to some embedded Lagrangian submanifold of $Y$.
Local charts on $\Gr^{L^+}(T_{\ast}Y)$ can thus be provided by
 Lagrangian neighborhoods on $Y$ as follows.\footnote{Such
                          local charts on {\it Gr}$\,^{L^+}(T_{\ast}Y)$
                           are  more convenient for our purpose.
                          One can also consider local charts
                           on {\it Gr}$\,^{L^+}(T_{\ast}Y)$
                           induced by Darboux charts on $Y$.}
Let
 $Z\subset Y$ be an oriented embedded Lagrangian submanifold $Y$,
 $\tilde{Z}\subset\Gr^{L^+}(T_{\ast}Y)$ be its prolongation
  to $\Gr^{L^+}(T_{\ast}Y)$,  and
 $\Phi_Z:U_Z\rightarrow Y$ be a Lagrangian neighborhood of $Z\subset Y$,
  where $U_Z$ is a neighborhood of the zero-section of $T^{\ast}Y$.
Local coordinates $(u_1,\,\cdots\,,u_m)$ of a chart $U$ on $Z$,
 with the orientation specified by $du_1\wedge\,\cdots\,\wedge du_m$
 induce local coordinates
 $(u_1,\,\cdots\,,u_m, p_{u_1},\,\cdots\,, p_{u_m})
  =:\, (\mbox{\boldmath $u$},\mbox{\boldmath $p_u$})\,
       \in {\Bbb R}^m_{(1)}\times{\Bbb R}^m_{(2)}$
 on the associated chart on $U_Z$ with
 $$
  (u_1,\,\cdots\,,u_m, p_{u_1},\,\cdots\,, p_{u_m})\;
  \longleftrightarrow\;
  p_{u_1}du_1+\,\cdots\,+p_{u_m}du_m \in T^{\ast}Z\,.
 $$
In terms of this and by [McD-S: Lemma~2.28],
 a coordinate chart $\tilde{U}$ for a neighborhood of
 $\tilde{Z}\subset\Gr^{L^+}(T_{\ast}Y)$ is given by
 $$
  \{((\mbox{\boldmath $u$},\mbox{\boldmath $p_u$}),A)\,|\,
    (\mbox{\boldmath $u$},\mbox{\boldmath $p_u$})
     \in\,\mbox{chart on $U_Z$ associated to $U$}\,;\,
    A:\,\mbox{symmetric $m\times m$-matrix}\}\,.
 $$
Here, $((\mbox{\boldmath $u$},\mbox{\boldmath $p_u$}),A)$
 specifies the oriented Lagrangian tangent subspace at
 $(\mbox{\boldmath $u$},\mbox{\boldmath $p_u$})$ given by
 $$
  \Lambda_{((\mbox{\scriptsize\boldmath $u$},
             \mbox{\scriptsize\boldmath $p_u$}),A)}\;
  :=\; \{(\mbox{\boldmath $\xi$}\,, A\mbox{\boldmath $\xi$})\,:\,
          \mbox{\boldmath $\xi$}\in {\Bbb R}^m_{(1)}\}\;\subset\;
       {\Bbb R}^m_{(1)}\times{\Bbb R}^m_{(2)}\,,
 $$
 where
  we have identified the tangent space of a point on
   ${\Bbb R}^m_{(1)}\times{\Bbb R}^m_{(2)}$ canonically with
   ${\Bbb R}^m_{(1)}\times{\Bbb R}^m_{(2)}$ itself,
   using the linear structure,  and
  the orientation of
   $\Lambda_{((\mbox{\scriptsize\boldmath $u$},
               \mbox{\scriptsize\boldmath $p_u$}),A)}$
   is specified by the orientation on ${\Bbb R}^m_{(1)}$ via
    the restriction of the projection map
    ${\Bbb R}^m_{(1)}\times{\Bbb R}^m_{(2)}\rightarrow {\Bbb R}^m_{(1)}$
    to $\Lambda_{((\mbox{\scriptsize\boldmath $u$},
                   \mbox{\scriptsize\boldmath $p_u$}),A)}$.

Now, an oriented embedded Lagrangian submanifold $Z^{\prime}$ in $Y$
  that is $C^1$-close to $Z$
 can be expressed as the graph of a closed $1$-form on $Z$
  under $\Phi_Z$.
On a small enough chart $U$ on $Z$, $Z^{\prime}$ is thus given by
  $$
   \Gamma_{dh}\;=\;
    \left\{\left.
     \left(\mbox{\boldmath $u$},
           \frac{\partial h}{\partial u_1}(\mbox{\boldmath $u$}),\,
           \cdots\,,
           \frac{\partial h}{\partial u_m}(\mbox{\boldmath $u$})
      \right) \right|\, \mbox{\boldmath $u$} \in U
    \right\}
  $$
  for some $h\in C^{\infty}(U)$.
The prolongation $\tilde{Z}^{\prime}$ of $Z^{\prime}$
 to $\Gr^{L^+}(T_{\ast}Y)$ is thus locally given by
  $$
   \tilde{\Gamma}_{dh}\; =\;
    \left\{\left.
     \left(\mbox{\boldmath $u$},
           \frac{\partial h}{\partial u_1}(\mbox{\boldmath $u$}),\,
           \cdots\,,
           \frac{\partial h}{\partial u_m}(\mbox{\boldmath $u$}),
     \left(\frac{\partial^2 h}{\partial u_i\partial u_j}
                               (\mbox{\boldmath $u$})\right)_{j,i}
      \right) \right|\, \mbox{\boldmath $u$} \in U
    \right\}\,.
  $$
It follows that the prolongation
 $\Gr^{L^+}\!\!f:X\rightarrow \Gr^{L^+}(T_{\ast}Y)$
 of an oriented Lagrangian immersion $f:X\rightarrow Y$
 can be expressed locally as the prolongation of a $2$-jet.

%
%
%
%
%
%

\bigskip

\begin{flushleft}
{\bf Estimating $\Imaginary\Omega|_{N^t}$.}
\end{flushleft}

\begin{definition-notation}
{\bf [setup: reference map of prolongation].} {\rm
 Recall Notation~3.1.7:
 The submanifolds with boundary
   $P_j^t$, $Q_j^t$, $K^t\subset N^t=X$ and
   the decomposition
    $$
     f^t\;
      =\;  \left( \cup_{j=1}^{\tilde{n_0}}
                  (f^t|_{P_j^t}\cup f^t|_{Q_j^t} ) \right)
           \bigcup f^t|_{K^t}
    $$
   of Lagrangian immersions $f^t:X\rightarrow Y$.
 With respect to this decomposition,
 define
   $$
    \Theta^t\;  :=\;
     \left( \coprod_{j=1}^{\tilde{n_0}}
      \left(\Theta^t_{P_j^t}\coprod \Theta^t_{Q_j^t} \right) \right)
     \coprod \Theta^t_{K^t}\;:\;
     \left( \coprod_{j=1}^{\tilde{n_0}}
      \left(P_j^t\coprod Q_j^t \right) \right) \coprod K^t\;
      \longrightarrow\;
      \Gr^{L^+}(T_{\ast}Y)
   $$
  by
  $$
   \begin{array}{ccccclc}
    \Theta^t_{P_j^t} & :
     & P_j^t & \longrightarrow
     & \Gr^{L^+}((T_{\ast}Y)|_{\Gamma_i})
     & \subset\; \Gr^{L^+}(T_{\ast}Y) \\[1.2ex]
    && (x_1,x_2,x_3)  & \longmapsto
     & \pr^{Gr}_2([f^t_{\ast}(T_{(x_1,x_2,x_3)}P_j^t)])
     & \in\;\, \Gr^{L^+}(T_{(0,0,x_3)}Y)  & , \\[1.8ex]
   \end{array}
  $$
  $$
   \Theta^t_{Q_j^t}\; =\; \Gr^{L^+}\!\!(f^0|_{Q_j^t})\,,
    \hspace{2em}\mbox{and}\hspace{2em}
   \Theta^t_{K^t}\; =\; \Gr^{L^+}\!\!(f^0|_{K^t})\,. \hspace{8.3em}
  $$
 Here,
  \begin{itemize}
   \item[$\cdot$]
    $f(\tilde{\Gamma}_j)=\Gamma_i$;

   \item[$\cdot$]
    the symplectic coordinates
     $(u_1,u_2,u_3,p_{u_1},p_{u_2},p_{u_3})=(u_1,u_2,u_3,v_1,v_2,v_3)$
     on $\nu^j_{Z_0}(\Gamma_i)$ induces a trivialization
     $\Gr^{L^+}(T_{\ast}\nu^j_{Z_0}(\Gamma_i)) \simeq
      T_{\ast}\nu^j_{Z_0}(\Gamma_i) \times_{\Gamma_i}
       \Gr^{L^+}(T_{\ast}Y|_{\Gamma_i})$
     via the symplectic linear structure from the coordinates  and
    $\pr^{Gr}_2$ is the projection map to the second factor;

   \item[$\cdot$]
    recall that $f^0:= f$.
  \end{itemize}
 We'll call $\Theta^t$
  a {\it (piecewise-smooth) reference map} for the prolongation
  $\Gr^{L^+}\!\!f^t:X\rightarrow \Gr^{L^+}(T_{\ast}Y)$ of $f^t$.
}\end{definition-notation}

\begin{proposition}
{\bf [basic estimate].}
 In the situation and notations in Definition~3.2.3,
  making $\delta>0$ smaller if necessary  and
  assuming that $b_1^t=t^{c_1}$, $b_2^t=t^{c_2}$ for some $0<c_2<c_1$,
 then there exists a constant $C>0$
  such that for all $t\in (0,\delta)$, one has
  $$
   \begin{array}{rcl}
    |\varepsilon^t|   & \le  & \left\{
     \begin{array}{lccl}
      C\,t^{c_1}\,+\, C\,t^{(1-\frac{1}{m_j})+\frac{c_1}{m_j}}
       && \mbox{on}\hspace{2ex} P_j^t\,,  \\[1.2ex]
      C\,t^{ 1-\frac{1}{m_j}-2c_2}\,
       +\, C\,t^{(1-\frac{1}{m_j})(1-c_1)}
       && \mbox{on}\hspace{2ex} Q_j^t\,,  \\[1.2ex]
      0 && \mbox{on}\hspace{2ex} K^t\,;
     \end{array}              \right.        \\[-1ex]
     & &                                     \\
    |d\varepsilon^t|  & \le  & \left\{
     \begin{array}{lccl}
      C\,+\, C\,t^{(m_j-1)(c_1-1)/m_j}\,.
       && \mbox{on}\hspace{2ex} P_j^t\,,  \\[1.2ex]
      C\,t^{1-\frac{1}{m_j}-3c_2}\,
       +\, C\,t^{(1-\frac{1}{m_j})(1-c_1)-c_2}\,
       +\, C\,t^{(1-\frac{1}{m_j})-c_1(2-\frac{1}{m_j})}
       && \mbox{on}\hspace{2ex} Q_j^t\,,  \\[1.2ex]
      0 && \mbox{on}\hspace{2ex} K^t
     \end{array}         \right.
   \end{array}
  $$
  for all $j=1,\,\cdots\,,\tilde{n}_0$.
 Here $|\cdot|$ is computed using the metric $g^t$ on $N^t$.
\end{proposition}

\begin{proof}
 (See Item (a.1) and Item (b.1) in the proof of Proposition~3.2.5,
      where all the necessary expressions are collected.)
 Note that
  $(\Theta^t)^{\ast}\varepsilon^t \equiv 0
   \equiv (\Theta^t)^{\ast}(d\varepsilon)$
  on
  $(\coprod_{j=1}^{\tilde{n_0}}
    (\Theta^t_{P_j^t}\coprod \Theta^t_{Q_j^t})) \coprod \Theta^t_{K^t}$.
 All the estimates can be made
  with $P_j^t\cup Q_j^t\subset (X,g^t)$
   approximated by the flat geometry on $L^{\prime,j}$.
 Recall Remark~3.1.6.
 The estimate for $\varepsilon^t$
  follows thus from pointwise Taylor's formula over $X$
   for $\varepsilon$ on $\Gr^{L^+}(T_{\ast}Y)$
   with the corresponding point in $\Image\Theta^t$
    as the reference point
 and the following estimates:
 %
 %
 %
 %
 %
 For $P_j^t$,
 consider $t\cdot L^{\prime,j}$ with $r_1\le b_1^t$.
  %
 Then,
  $$
   0\; \le\; r_1\; \le b_1^t\; =\; t^{c_1}\,,
  $$
  $$
   0\; \le\; r_2\;
   =\;   a_j^{1/m_j}\,t^{(m_j-1)/m_j}\,r_1^{1/m_j}\;
   \le\; O(t^{(m_j-1)/m_j})\,(b_1^t)^{1/m_j}\;
   =\;   O(\,t^{(1-\frac{1}{m_j})+\frac{c_1}{m_j}}\,)\,.
  $$
 For $Q_j^t$,
 consider $t\cdot L^{\prime,j}$ with $b_1^t\le r_1 \le b_2^t$:
  %
 Then,
  $$
   r_2\;
   =\;   a_j^{1/m_j}\,t^{(m_j-1)/m_j}\,r_1^{1/m_j}\;
   \ge\; O(t^{(m_j-1)/m_j})\,(b_1^t)^{1/m_j}\;
   \ge\; O(\,t^{(1-\frac{1}{m_j})+\frac{c_1}{m_j}}\,)\,;
 $$
 $$
   r_2\;
   =\;   a_j^{1/m_j}\,t^{(m_j-1)/m_j}\,r_1^{1/m_j}\;
   \le\; O(t^{(m_j-1)/m_j})\,(b_2^t)^{1/m_j}\;
   =\;   O(\,t^{(1-\frac{1}{m_j})+\frac{c_2}{m_j}}\,)\,;
 $$

 \begin{eqnarray*}
  \lefteqn{\left|\,
   a_j^{1/m_j}\,t^{(m_j-1)/m_j}\,
     \left[\frac{m_j}{m_j+1}\,\dot{\chi}^t(r_1)\,
          +\,\chi^t(r_1)\,r_1^{1/m_j}\right] \right|}     \\[1.2ex]
   && \le\;
    a_j^{1/m_j}\,t^{(m_j-1)/m_j}\,
    \left(\frac{m_j}{m_j+1}\cdot \frac{C_0}{b_2^t}\,
          +\, r_1^{1/m_j} \right)                         \\[1.2ex]
   && \le\;
    O(t^{(m_j-1)/m_j})\,(b_2^t)^{-1}\,
    +\, O(t^{(m_j-1)/m_j})\,(b_2^t)^{1/m_j}               \\[1.2ex]
   && =\;
    O(\,t^{ 1-\frac{1}{m_j}-c_2}\,)\,
    +\, O(\,t^{1-\frac{1}{m_j}+\frac{c_2}{m_j} }\,)       \\[1.2ex]
   && =\;
    O(\,t^{ 1-\frac{1}{m_j}-c_2}\,)\,;
 \end{eqnarray*}

 \begin{eqnarray*}
  \lefteqn{\left|
   \frac{d}{dr_1}
     \left(
      a_j^{1/m_j}\,t^{(m_j-1)/m_j}\,
        \left[\frac{m_j}{m_j+1}\,\dot{\chi}^t(r_1)\,
             +\,\chi^t(r_1)\,r_1^{1/m_j}\right]
               \right) \right|}                                \\[1.2ex]
   && =\;
    a_j^{1/m_j}\,t^{(m_j-1)/m_j}\,
    \left|
     \frac{m_j}{m_j+1}\,\ddot{\chi}^t(r_1)\,
     +\, \dot{\chi}^t(r_1)\,r_1^{1/m_j}\,
     +\, \frac{1}{m_j}\,\chi^t(r_1)\,r_1^{(1-m_j)/m_j} \right| \\[1.2ex]
   && \le\;
    a_j^{1/m_j}\,t^{(m_j-1)/m_j}\,
    \left(\frac{m_j}{m_j+1}\cdot\frac{C_0}{(b_2^t)^2}\,
          +\, \frac{C_0}{b_2^t}\,r_1^{1/m_j}\,
          +\, \frac{1}{m_j}\,r_1^{(1-m_j)/m_j} \right)         \\[1.2ex]
   && \le\;
    O(t^{(m_j-1)/m_j})\,(b_2^t)^{-2}\,
    +\, O(t^{(m_j-1)/m_j})\,(b_2^t)^{(1-m_j)/m_j}\,
    +\, O(t^{(m_j-1)/m_j})\,(b_1^t)^{(1-m_j)/m_j}              \\[1.2ex]
   && =\;
    O(\,t^{ 1-\frac{1}{m_j}-2c_2}\,)\,
    +\, O(\,t^{(1-\frac{1}{m_j})(1-c_2)}\,)\,
    +\, O(\,t^{(1-\frac{1}{m_j})(1-c_1)}\,)                    \\[1.2ex]
   && =\;
    O(\,t^{ 1-\frac{1}{m_j}-2c_2}\,)\,
    +\, O(\,t^{(1-\frac{1}{m_j})(1-c_1)}\,)\,.
 \end{eqnarray*}
 Here, we have used the fact that
 $r_1^{1/m_j}$ (resp.\ $r_1^{(1-m_j)/m_j}$)
  is an increasing (resp.\ decreasing) function over $r_1>0$ for all $j$
   and
 $(1-\frac{1}{m_j})(1-c_2)>(1-\frac{1}{m_j})(1-c_1)$.
 Under the assumption, all the exponents in the $t$-orders are positive.

 Similarly for the estimates for $|d\varepsilon^t|$
  with, in addition, the following estimates:
 For $P_j^t$,
 $$
  dr_2/dr_2\;=\;1\,,
 $$
 $$
  \left|
   \frac{d}{dr_2}
     \left(a_j^{-1}\,t^{1-m_j}\,r_2^{m_j}\right) \right|\;
   =\;    m_j\,a_j^{-1}\,t^{1-m_j}\,r_2^{m_j-1}\;
   \le\;  O(t^{(m_j-1)(c_1-1)/m_j})\,.
 $$
 For $Q_j^t$,
 \begin{eqnarray*}
  \lefteqn{\left|
   \frac{d^2}{dr_1^2}
     \left(
      a_j^{1/m_j}\,t^{(m_j-1)/m_j}\,
            \left[\frac{m_j}{m_j+1}\,\dot{\chi}^t(r_1)\,
                  +\,\chi^t(r_1)\,r_1^{1/m_j}\right]
               \right) \right|}                                \\[1.2ex]
   && =\;
    a_j^{1/m_j}\,t^{(m_j-1)/m_j}\,
    \left|\rule{0em}{1.6em}\right.
     \frac{m_j}{m_j+1}\,\dddot{\chi}^{\,t}(r_1)\,
     +\, \ddot{\chi}^t(r_1)\, r_1^{1/m_j}\,                    \\[1.2ex]
   && \hspace{10em}
    +\, \frac{2}{m_j}\,\dot{\chi}^t(r_1)\,r_1^{(1-m_j)/m_j}
    +\, \frac{1-m_j}{m_j^2}\,\chi^t(r_1)\,r_1^{(1-2m_j)/m_j}
            \left.\rule{0em}{1.6em}\right|                     \\[1.2ex]
   && \le\;
    a_j^{1/m_j}\,t^{(m_j-1)/m_j}\,
    \left(\rule{0em}{1.6em}\right.
      \frac{m_j}{m_j+1}\,\frac{C_0}{(b_2^t)^3}\,
         +\, \frac{C_0}{(b_2^t)^2}\,r_1^{1/m_j}                \\[1.2ex]
   && \hspace{10em}
         +\, \frac{2\,C_0}{m_j\,b_2^t}\,r_1^{(1-m_j)/m_j}\,
         +\, \frac{1-m_j}{m_j^2}\,r_1^{(1-2m_j)/m_j}
    \left.\rule{0em}{1.6em}\right)                             \\[1.2ex]
   && \le\;
    O(t^{(m_j-1)/m_j})\,(b_2^t)^{-3}\,
    +\,O(t^{(m_j-1)/m_j})\,(b_2^t)^{(1-2m_j)/m_j}\,            \\[1.2ex]
   && \hspace{6em}
    +\, O(t^{(m_j-1)/m_j})\,(b_2^t)^{-1}\,(b_1^t)^{(1-m_j)/m_j}
    +\, O(t^{(m_j-1)/m_j})\,(b_1^t)^{(1-2m_j)/m_j}             \\[1.2ex]
   && =\;
    O(\,t^{1-\frac{1}{m_j}-3c_2}\,)\,
    +\, O(\,t^{(1-\frac{1}{m_j})-c_2(2-\frac{1}{m_j})}\,)\,
    +\, O(\,t^{(1-\frac{1}{m_j})(1-c_1)-c_2}\,)\,
    +\, O(\,t^{(1-\frac{1}{m_j})-c_1(2-\frac{1}{m_j})}\,)     \\[1.2ex]
   && =\;
    O(\,t^{1-\frac{1}{m_j}-3c_2}\,)\,
    +\, O(\,t^{(1-\frac{1}{m_j})(1-c_1)-c_2}\,)\,
    +\, O(\,t^{(1-\frac{1}{m_j})-c_1(2-\frac{1}{m_j})}\,)\,.
  \end{eqnarray*}
 Here, we have used in addition the fact that
 $r_1^{(1-2m_j)/m_j}$
  is a decreasing function over $r_1>0$ for all $j$, and
 $(1-\frac{1}{m_j})-c_2(2-\frac{1}{m_j})
  > (1-\frac{1}{m_j})-c_1(2-\frac{1}{m_j})$.
 Under the assumption, all the exponents in the $t$-orders are positive.

\end{proof}

\begin{proposition}
{\bf [Sobolev norm estimate].}
 Continuing the situation in Proposition~3.2.4,
 with an additional assumption that $m_j\not\in\{2,6,11\}$,
 then for all $t\in (0,\delta)$, with $\delta$ small enough,
  the norms
  $\|\varepsilon^t\|_{L^{6/5}}$, $\|\varepsilon^t\|_{C^0}$,
  $\|d\varepsilon^t\|_{L^6}$, and $\|\varepsilon^t\|_{L^1}$
   are bounded above by $t$-powers with exponents a linear function
   in $c_1$ and $c_2$ with coefficients fractional functions in $m_j$,
   $j=1,\,\,\ldots\,,\,\tilde{n}_0$.

 Here the norms $\|\cdot\|_{\bullet}$ are computed using $g^t$ on $N^t$.
\end{proposition}

\noindent
Exact expressions for these exponents are given in the proof.

\begin{proof}
 The calculation is similar to that in the proof of Proposition~3.2.4,
  with the same reference map $\Theta^t$ for Taylor expansion and
  an additional ingredient from the $t$-dependent volume-forms
  on $X$ from the pull-back metric $g^t$.
 As $f^t$, $t\in (0,\delta)$, are immersions,
  we will perform the computation using data and coordinates on $Y$.

 \bigskip

 \noindent
 $(a)$ {\it On $P_j^t$.}\hspace{1em}

 \bigskip

 \noindent
 $(a.1)$ {\it Basic data for estimates on $P_j^t$.}\hspace{1em}
 Consider
  $$
   f^t(P_j^t)\; =\;
    \left\{
      (r_1e^{\sqrt{-1}\theta_1},\, u_3,\,
       r_2e^{\sqrt{-1}\theta_2},\, 0)\,
     \left|
      \begin{array}{l}
       \cdot\; r_1\;
        =\; a_j^{-1}\,t^{1-m_j}\,r_2^{m_j}\,,   \\[1.2ex]
       \hspace{2em} 0\, \le\, r_1\, \le\, b_1^t\, =\, t^{c_1}\,; \\[1.2ex]
       \cdot\;\theta_1\;=\; m_j\,\theta_2\;\;\;\;
                              (\,\mbox{\rm mod}\; 2\pi\,)
      \end{array}\right.\hspace{-1ex}\right\}.\
  $$
 The approximate metric is given by
  $$
   ds^2\; =\;
    \left(1\,+\,m_j^2\,a_j^{-2}\,t^{2(1-m_j)}\,r_2^{2(m_j-1)}
     \right)\,dr_2^2\,
    +\, r_2^2\,
        \left(1\,+\,m_j^2\,a_j^{-2}\,t^{2(1-m_j)}\,r_2^{2(m_j-1)}
         \right)\,d\theta_2^2\, +\, du_3^2\,.
  $$
  which gives the approximate volume-form:
   $$
    \vol^t\; =\;
     r_2\,
     \left(1\,+\,m_j^2\,a_j^{-2}\,t^{2(1-m_j)}\,r_2^{2(m_j-1)}
      \right)\,dr_2\wedge d\theta_2\wedge du_3\,.
   $$
 The imaginary part $\varepsilon^t$ of the ratio calibration/volume-form
  is approximated by
   $$
    \varepsilon^t\; =\; O(r_1)\,+\,O(r_2)\;
     =\; O(\,t^{1-m_j}\,r_2^{m_j}\,)\,+\, O(r_2)\,.
   $$  and
 $|d\varepsilon^t|$ is approximated by
   $$
    |d\varepsilon^t|\;=\; O(t^{1-m_j}r_2^{m_j-1})\,+\, O(1)\,.
   $$

 \bigskip

 \noindent
 $(a.2)$
 {\it Estimating $\|\varepsilon^t\|_{C^0}$ on $P_j^t$.}\hspace{1em}
 It follows from Part (a.1) (cf.\ Proposition~3.2.4) that
 $$
  \|\varepsilon^t\|_{C^0}\; =\;
  O(\,t^{c_1}\,)\,+\, O(\,t^{(1-\frac{1}{m_j})+\frac{c_1}{m_j}}\,)\,.
 $$

 \bigskip

 \noindent $(a.3)$
 {\it Estimating $\|\varepsilon^t\|_{L^{6/5}}$ on $P_j^t$.}
 \begin{eqnarray*}
  \lefteqn{\|\varepsilon^t\|_{L^{6/5}}\; =\;
   \left(
    2\pi l_j\,
    \int_{r_2=0}^{a^{1/m_j}t^{(m_j-1)/m_j}(b_1^t)^{1/m_j}}
     \left|\,
      O(\,t^{1-m_j}\,r_2^{m_j}\,)\,+\, O(r_2)
     \right|^{6/5} \right.}                          \\
   && \hspace{12em} \left.\rule{0em}{1.6em} \cdot
     r_2\,
     \left(1\,+\,m_j^2\,a_j^{-2}\,t^{2(1-m_j)}\,r_2^{2(m_j-1)}
      \right)\,dr_2\,
    \right)^{5/6}                                    \\[1.2ex]
   && =\;
   O(1)\,
    \left(\rule{0em}{1.6em}\right.
    \int_{r_2=0}^{O(\,t^{1+\frac{c_1-1}{m_j}}\,)}
     \left|\,
      O(\,t^{1-m_j}\,r_2^{m_j}\,)\,+\, O(r_2)
     \right|^{6/5}                            \\
   && \hspace{12em} \left.\rule{0em}{1.6em} \cdot
     r_2\,
     \left(1\,+\,O(\,t^{2(1-m_j)}\,r_2^{2(m_j-1)}\,)
      \right)\,dr_2\,
    \right)^{5/6}\,.
 \end{eqnarray*}
 Note that both $O(\,t^{1-m_j}\,r_2^{m_j}\,) = O(r_2)$ and
  $1\,=\,O(\,t^{2(1-m_j)}\,r_2^{2(m_j-1)}\,)$ occur at $r_2=O(t)$.
 Thus:

 If $0<c_1\le 1$,
 then for $\delta>0$ small enough, and all $t\in (0,\delta)$,
  $$
   0<O(t)\le O(\,t^{1+\frac{c_1-1}{m_j}}\,)
  $$
  \begin{eqnarray*}
   \lefteqn{\|\varepsilon^t\|_{L^{6/5}}\; =\;
    O(1)\,
    \left(\rule{0em}{1.6em}\right.
     \left(\rule{0em}{1.2em}\right.
      \int_{r_2=0}^{O(t)}\,
      +\,\int_{r_2=O(t)}^{O(\,t^{1+\frac{c_1-1}{m_j}}\,)}
     \left.\rule{0em}{1.2em}\right)
      \left|\,
       O(\,t^{1-m_j}\,r_2^{m_j}\,)\,+\, O(r_2)
      \right|^{6/5}                                 } \\
    && \hspace{20em} \left.\rule{0em}{1.6em} \cdot
      r_2\,
      \left(1\,+\,O(\,t^{2(1-m_j)}\,r_2^{2(m_j-1)}\,)
       \right)\,dr_2\,
     \right)^{5/6}            \\[1.2ex]
    && =\;
     O(1)\, \left(
      \int_{r_2=0}^{O(t)}\, O(r_2)^{6/5}
         \cdot r_2\cdot\,O(1)\,dr_2  \right. \\
    && \hspace{8em}
     +\,
     \int_{r_2=O(t)}^{O(\,t^{1+\frac{c_1-1}{m_j}}\,)}
      O(\,t^{1-m_j}\,r_2^{m_j}\,)^{6/5}\cdot r_2 \cdot
        O(\,t^{2(1-m_j)}\,r_2^{2(m_j-1)}\,)\, dr_2\,
          \left.\rule{0em}{1.6em}\right)^{5/6}    \\[1.2ex]
    && =\; \left( O(t^{16/5})\,+\, O(t^{16 c_1/5})
           \right)^{5/6}\; =\;  O(t^{8 c_1/3})\,.
  \end{eqnarray*}

 If $c_1>1$, then $O(\,t^{1+\frac{c_1-1}{m_j}}\,) < O(t)$ and
  $$
   \|\varepsilon^t\|_{L^{6/5}}\;
    =\;
    \left(\rule{0em}{1.2em}\right.
      \int_{r_2=0}^{O(\,t^{1+\frac{c_1-1}{m_j}}\,)}
       O(r_2)^{6/5}\cdot r_2\, dr_2\,
     \left.\rule{0em}{1.2em}\right)^{5/6}\;
    =\; O(\,t^{\frac{8}{3}\,(1+\frac{c_1-1}{m_j})}\,)\,.
  $$

 \bigskip

 \noindent
 $(a.4)$
 {\it Estimating $\|\varepsilon^t\|_{L^1}$ on $P_j^t$.}\hspace{1em}
 Similar to
  the estimation for $\|\varepsilon^t\|_{L^{6/5}}$ in Part (a.3),
 if $0<c_1\le 1$, then
  \begin{eqnarray*}
   \lefteqn{\|\varepsilon^t\|_{L^1}\; =\;
      \int_{r_2=0}^{O(t)}\, O(r_2)
         \cdot r_2\cdot\,O(1)\,dr_2           }\\
    && \hspace{8em}
     +\,
     \int_{r_2=O(t)}^{O(\,t^{1+\frac{c_1-1}{m_j}}\,)}
      O(\,t^{1-m_j}\,r_2^{m_j}\,)\cdot r_2\cdot
        O(\,t^{2(1-m_j)}\,r_2^{2(m_j-1)}\,)\, dr_2     \\[1.2ex]
    && =\; O(t^3)\,+\, O(t^{3c_1})\; =\;  O(t^{3c_1})\,.
  \end{eqnarray*}

 If $c_1>1$, then
  $$
   \|\varepsilon^t\|_{L^1}\;
    =\; \int_{r_2=0}^{O(\,t^{1+\frac{c_1-1}{m_j}}\,)}\, O(r_2)
         \cdot r_2\cdot\,O(1)\,dr_2\;
    =\; O(\,t^{3(1+\frac{c_1-1}{m_j})}\,)\,.
  $$

 \bigskip

 \noindent
 $(a.5)$
 {\it Estimating $\|d\varepsilon^t\|_{L^6}$ on $P_j^t$.}\hspace{1em}
 Note $ O(t^{1-m_j}r_2^{m_j-1})=O(1)$ occurs also at $r_2=O(t)$.
 Thus, similar to
  the estimation for $\|\varepsilon^t\|_{L^{6/5}}$ in Part (a.3),
 if $0<c_1\le 1$, then
  \begin{eqnarray*}
   \lefteqn{\|d\varepsilon^t\|_{L^6}\; =\;
    \left(
      \int_{r_2=0}^{O(t)}\, O(1)^6
         \cdot r_2\cdot\,O(1)\,dr_2  \right. }\\
    && \hspace{8em}
     +\,
     \int_{r_2=O(t)}^{O(\,t^{1+\frac{c_1-1}{m_j}}\,)}
      O(\,t^{1-m_j}\,r_2^{m_j-1}\,)^6\cdot r_2\cdot
        O(\,t^{2(1-m_j)}\,r_2^{2(m_j-1)}\,)\, dr_2\,
           \left.\rule{0em}{1.6em}\right)^{1/6}        \\[1.2ex]
    && =\; \left(O(t^2)\, +\,
            O(\,t^{2c_1(4-\frac{3}{m_j})+6(\frac{1}{m_j}-1)}\,)
           \right)^{1/6}\;
       =\; O(\,t^{c_1(\frac{4}{3}-\frac{1}{m_j})+(\frac{1}{m_j}-1)}\,)\,.
  \end{eqnarray*}

 If $c_1>1$, then
  $$
   \|d\varepsilon^t\|_{L^1}\;
    =\; \left(\rule{0em}{1.6em}\right.
         \int_{r_2=0}^{O(\,t^{1+\frac{c_1-1}{m_j}}\,)}\, O(1)^6
          \cdot r_2\cdot\,O(1)\,dr_2
        \left.\rule{0em}{1.6em}\right)^{1/6}\;
    =\; O(\,t^{\frac{1}{3}(1+\frac{c_1-1}{m_j})}\,)\,.
  $$

 \bigskip

 \noindent
 $(b)$ {\it On $Q_j^t$.}

 \bigskip

 \noindent
 $(b.1)$ {\it Basic data for estimates on $Q_j^t$.}\hspace{1em}
 Consider
  \begin{eqnarray*}
   \lefteqn{f^t(Q_j^t)\; =\; }\\
   &&
    \left\{
      (r_1e^{\sqrt{-1}\theta_1},\, u_3,\,
       r_2e^{\sqrt{-1}\theta_2},\, 0)\,
     \left|
      \begin{array}{l}
       \cdot\; r_2\;
        =\; a_j^{1/m_j}\,t^{(m_j-1)/m_j}\,
             \left[\frac{m_j}{m_j+1}\,\dot{\chi}^t(r_1)\,
                   +\,\chi^t(r_1)\,r_1^{1/m_j}\right]\,,  \\[1.8ex]
        \hspace{2em} b_1^t=t^{c_1}\, \le\, r_1\,
                     \le\, t^{c_2} = b_2^t\,; \\[1.2ex]
       \cdot\; m_j\,\theta_2\;=\; \theta_1\;\;\;\;
                              (\,\mbox{\rm mod}\; 2\pi\,)
      \end{array}\right.\hspace{-1ex}\right\}.
  \end{eqnarray*}
 The approximate metric is given by
  $$
   \begin{array}{l}
   ds^2\; =\;
     \left(1\,+\,
      a_j^{2/m_j}\,t^{2(m_j-1)/m_j}\,
      \left[\frac{m_j}{m_j+1}\,\ddot{\chi}^t(r_1)\,
            +\,\dot{\chi}^t(r_1)\,r_1^{1/m_j}
            +\,\frac{1}{m_j}\,\chi^t(r_1)\,r_1^{(1-m_j)/m_j}\right]^2
     \right)\,dr_1^2                                          \\[2ex]
    \hspace{6em}
    +\, \left(r_1^2\,
         +\,m_j^{-2}\,a_j^{2/m_j}\,t^{2(m_j-1)/m_j}\,
             \left[\frac{m_j}{m_j+1}\,\dot{\chi}^t(r_1)\,
                   +\,\chi^t(r_1)\,r_1^{1/m_j}\right]^2\right)\,
             d\theta_1^2\,
    +\, du_3^2\,,
   \end{array}
  $$
  which gives the approximate volume-form
  $$
   \begin{array}{l}
    \vol^t\; =\;
     \left(1\,+\,
      a_j^{2/m_j}\,t^{2(m_j-1)/m_j}\,
      \left[\frac{m_j}{m_j+1}\,\ddot{\chi}^t(r_1)\,
            +\,\dot{\chi}^t(r_1)\,r_1^{1/m_j}
            +\,\frac{1}{m_j}\,\chi^t(r_1)\,r_1^{(1-m_j)/m_j}\right]^2
     \right)^{1/2}\,                                          \\[2ex]
    \hspace{2em}
     \cdot\, \left(r_1^2\,
         +\,m_j^{-2}\,a_j^{2/m_j}\,t^{2(m_j-1)/m_j}\,
             \left[\frac{m_j}{m_j+1}\,\dot{\chi}^t(r_1)\,
                   +\,\chi^t(r_1)\,r_1^{1/m_j}\right]^2\right)^{1/2}\,
    dr_1\wedge d\theta_1\wedge du_3\,.
   \end{array}
  $$
 Thus,
  $$
   \begin{array}{ccl}
    |\vol^t|  & \le
     & \left( 1\,
        +\, O(\,t^{2(1-\frac{1}{m_j})}\,)\,
           \left[
            O(\,t^{-2c_2}\,)\,
            +\, O(\,t^{-c_2}\,r_1^{\frac{1}{m_j}}\,)\,
            +\, O(\,r_1^{\frac{1}{m_j}-1}\,) \right]^2\,\right)^{1/2} \\[3ex]
    && \hspace{2em} \cdot
       \left( r_1^2\,
        +\, O(\,t^{2(1-\frac{1}{m_j})}\,)\,
           \left[
            O(\,t^{-c_2}\,)\,
            +\, O(\,r_1^{\frac{1}{m_j}}\,) \right]^2\,\right)^{1/2}\,
       |dr_1\wedge d\theta_1\wedge du_3|       \\[3ex]
    && =
       \left( 1\,
        +\, O(\,t^{2(1-\frac{1}{m_j})}\,)\,
           \left[
            O(\,t^{-2c_2}\,)\,
            +\, O(\,r_1^{\frac{1}{m_j}-1}\,) \right]^2\,\right)^{1/2} \\[3ex]
    && \hspace{12em} \cdot
       \left( r_1^2\,
        +\, O(\,t^{2(1-\frac{1}{m_j}-c_2)}\,)\, \right)^{1/2}\,
       |dr_1\wedge d\theta_1\wedge du_3|       \\[3ex]
   \end{array}
  $$
 The imaginary part $\varepsilon^t$ of the ratio calibration/volume-form
  is approximated by
  $$
   \begin{array}{rcl}
    \varepsilon^t  &  =  &  O(r_2)\,+\, O(dr_2/dr_1)  \\[1.2ex]
     & =
     &  O\left(\,t^{(m_j-1)/m_j}\,
               \left[\frac{m_j}{m_j+1}\,\dot{\chi}^t(r_1)\,
                  +\,\chi^t(r_1)\,r_1^{1/m_j}\right]\,\right) \\[1.2ex]
    && \hspace{1em}
       +\, O\left(\,t^{(m_j-1)/m_j}\,
                \left[\frac{m_j}{m_j+1}\,\ddot{\chi}^t(r_1)\,
                  +\,\dot{\chi}^t(r_1)\,r_1^{1/m_j}
                  +\,\frac{1}{m_j}\,\chi^t(r_1)\,r_1^{(1-m_j)/m_j}
                \right]\,\right)\,;
   \end{array}
  $$
 Thus,
  $$
   \begin{array}{ccl}
    |\varepsilon^t|  & \le
     & O(\,t^{1-\frac{1}{m_j}}\,)\,
       \left[
         O(\,t^{-2c_2}\,)\,
         +\, O(\,t^{-c_2}\,r_1^{\frac{1}{m_j}}\,)\,
         +\, O(\,r_1^{\frac{1}{m_j}-1}\,)
      \right]    \\[2ex]
    & =
     & O(\,t^{1-\frac{1}{m_j}}\,)\,
       \left[
         O(\,t^{-2c_2}\,)\, +\, O(\,r_1^{\frac{1}{m_j}-1}\,)
      \right]\,.
   \end{array}
  $$
 $|d\varepsilon^t|$ is approximated by
  $$
   \begin{array}{rcl}
    |d\varepsilon^t|
     &  =   &  O(dr_2/dr_1)\,+\, O(d^2r_2/dr_1^2)         \\[2ex]
    & = & O\left(\,t^{(m_j-1)/m_j}\,
                \left[\frac{m_j}{m_j+1}\,\ddot{\chi}^t(r_1)\,
                  +\,\dot{\chi}^t(r_1)\,r_1^{1/m_j}
                  +\,\frac{1}{m_j}\,\chi^t(r_1)\,r_1^{(1-m_j)/m_j}
                \right]\,\right)                          \\[3ex]
    && \hspace{1em}
       +\, O\left(\,
        t^{(m_j-1)/m_j}\,
         \left[\rule{0em}{1em}\right.
          \frac{m_j}{m_j+1}\,\dddot{\chi}^{\,t}(r_1)\,
          +\, \ddot{\chi}^t(r_1)\, r_1^{1/m_j}\,  \right. \\[3ex]
    && \hspace{6em}
       \left.
        +\, \frac{2}{m_j}\,\dot{\chi}^t(r_1)\,r_1^{(1-m_j)/m_j}
        +\, \frac{1-m_j}{m_j^2}\,\chi^t(r_1)\,r_1^{(1-2m_j)/m_j}
                 \left.\rule{0em}{1em}\right]\, \right)   \\[3ex]
    & \le
     & O(\,t^{1-\frac{1}{m_j}}\,)\,
       \left[
        O(\,t^{-3c_2}\,)\,+\,O(\,t^{-2c_2}\,r_1^{\frac{1}{m_j}}\,)\,
        +\, O(\,t^{-c_2}\,r_1^{\frac{1}{m_j}-1}\,)\,
        +\, O(\,r_1^{\frac{1}{m_j}-2}\,) \right]          \\[3ex]
    & =
     & O(\,t^{1-\frac{1}{m_j}}\,)\,
       \left[
        O(\,t^{-3c_2}\,)\,
        +\, O(\,t^{-c_2}\,r_1^{\frac{1}{m_j}-1}\,)\,
        +\, O(\,r_1^{\frac{1}{m_j}-2}\,) \right]          \\[3ex]
    & =
     & O(\,t^{1-\frac{1}{m_j}}\,)\,
       \left[
        O(\,t^{-3c_2}\,)\,
        +\, O(\,r_1^{\frac{1}{m_j}-2}\,) \right]\,.
   \end{array}
  $$
 Here, we use
  $O(\,t^{-c_2}\,r_1^{\frac{1}{m_j}-1}\,) + O(\,r_1^{\frac{1}{m_j}-2}\,)
   = O(\,r_1^{\frac{1}{m_j}-2}\,)$
  for $0<r_1<t^{c_2}$ in the last equality.

 \bigskip

 \noindent
 $(b.2)$
 {\it Estimating $\|\varepsilon^t\|_{C^0}$ on $Q_j^t$.}\hspace{1em}
 It follows from Part (b.1) (cf.\ Proposition~3.2.4) that
  $$
   \|\varepsilon^t\|_{C^0}\;
    \le\;
     O(\,t^{ 1-\frac{1}{m_j}-2c_2}\,)\,
       +\, O(\,t^{(1-\frac{1}{m_j})(1-c_1)}\,)\,.
  $$

 %

 \bigskip

 \noindent
 $(b.3)$
 {\it Estimating $\|\varepsilon^t\|_{L^{6/5}}$ on $Q_j^t$.}\hspace{1em}
 It follows from Part (b.1) that
 \begin{eqnarray*}
  \lefteqn{\|\varepsilon^t\|_{L^{6/5}}\; \le\;
   \left(
    2\pi m_j l_j\,
    \int_{r_1=t^{c_1}}^{t^{c_2}}
     \left|\,
      O(\,t^{1-\frac{1}{m_j}}\,)\,
      \left[
        O(\,t^{-2c_2}\,)\, +\, O(\,r_1^{\frac{1}{m_j}-1}\,)
      \right]
     \right|^{6/5} \right.                          }\\[1.2ex]
  && \hspace{7em} \cdot
      \left( 1\,
        +\, O(\,t^{2(1-\frac{1}{m_j})}\,)\,
           \left[
            O(\,t^{-2c_2}\,)\,
            +\, O(\,r_1^{\frac{1}{m_j}-1}\,)
                               \right]^2\, \right)^{1/2} \\[1.2ex]
  && \hspace{16em} \left.
      \cdot
      \left( r_1^2\, +\, O(\,t^{2(1-\frac{1}{m_j}-c_2)}\,)\,
       \right)^{1/2}\,
      dr_1\,\right)^{5/6}\,.
 \end{eqnarray*}
 The equalities
  $$
   \begin{array}{rclccrcl}
    O(\,t^{-2c_2}\,) & = &  O(\,r_1^{\frac{1}{m_j}-1}\,)\,; &&
     & 1 & = & O(\,t^{2(1-\frac{1}{m_j}-2c_2)}\,)\,, \\[1.2ex]
    1 & = & O(\,t^{2(1-\frac{1}{m_j})}r_1^{2(\frac{1}{m_j}-1)}\,)\,; &&
     & r_1^2 & = & O(\,t^{2(1-\frac{1}{m_j}-c_2)}\,)
   \end{array}
  $$
  are solved by
  $$
   r_1\,=\,O(\,t^{\frac{2c_2m_j}{m_j-1}}\,)\,;\hspace{2em}
   c_2\,=\,\frac{1}{2}(1-\frac{1}{m_j})\,,\hspace{2em}
   r_1\,=\,O(t)\,;\hspace{2em}
   r_1\,=\,O(\,t^{1-\frac{1}{m_j}-c_2}\,)
  $$
  respectively.
 Denote the $t$-exponents:
  $$
   E_1\, :=\, c_1\,,\hspace{2em}
   E_2\, :=\, c_2\,,\hspace{2em}
   E_3\, :=\, \frac{2c_2m_j}{m_j-1}\,,\hspace{2em}
   E_4\, :=\, 1-\frac{1}{m_j}-c_2\,,\hspace{2em}
   E_5\, :=\, 1\,.
  $$
 Then the equations $E_i=E_j$, $1\le i<j\le 5$, divides
  $\{(c_1,c_2)\in ({\Bbb R}_+)^2\,:\, c_1>c_2\}$
  into 13 regions:

 \bigskip

 \centerline{\small\it
 \noindent
 \begin{tabular}{|l|c|c|} \hline
   $\hspace{.8em}$region      & description in $\{0<c_2<c_1\}$
    & order                   \rule{0em}{1.2em}\\[.6ex] \hline\hline
   Region$_{\,(1)}$
    & $1\le c_2\, (< c_1)\,,\;\;
       c_2\ge \frac{1}{2}(1-\frac{1}{m_j})c_1$
    & $E_4\le E_5\le E_2\le E_1\le E_3$  \rule{0em}{1.6em} \\[1.4ex] \hline
   Region$_{\,(2)}$
    & $1\le c_2\le \frac{1}{2}(1-\frac{1}{m_j})c_1$
    & $E_4\le E_5\le E_2\le E_3\le E_1$  \rule{0em}{1.6em} \\[1.4ex] \hline
   Region$_{\,(3)}$
    & $\frac{1}{2}(1-\frac{1}{m_j}) \le c_2\le 1\,,\;\;
       c_2\le \frac{1}{2}(1-\frac{1}{m_j})c_1$
    & $E_4\le E_2\le E_5\le E_3\le E_1$  \rule{0em}{1.6em} \\[1.4ex] \hline
   Region$_{\,(4)}$
    & $c_1\ge 1\,,\;\;
       \frac{(m_j-1)^2}{m_j(3m_j-1)}
        \le c_2 \le \frac{1}{2}(1-\frac{1}{m_j})$
    & $E_2\le E_4\le E_3\le E_5\le E_1$  \rule{0em}{1.6em} \\[1.4ex] \hline
   Region$_{\,(5)}$
    & $c_1\ge 1\,,\;\,
       0< c_2\le  \frac{(m_j-1)^2}{m_j(3m_j-1)}$
    & $E_2\le E_3\le E_4\le E_5\le E_1$  \rule{0em}{1.6em} \\[1.4ex] \hline
   Region$_{\,(6)}$
    & $c_1\ge 1\,,\;\;
       \frac{1}{2}(1-\frac{1}{m_j})c_1 \le c_2 \le 1$
    & $E_4\le E_2\le E_5\le E_1\le E_3$  \rule{0em}{1.6em} \\[1.4ex] \hline
   Region$_{\,(7)}$
    & $c_1\le 1\,,\;\;
       \frac{1}{2}(1-\frac{1}{m_j})\le c_2\, (<c_1)$
    & $E_4\le E_2\le E_1\le E_5\le E_3$  \rule{0em}{1.6em} \\[1.4ex] \hline
   Region$_{\,(8)}$
    & $(1-\frac{1}{m_j})-c_1
        \le c_2\le \frac{1}{2}(1-\frac{1}{m_j})\,,\;\;
       c_2\ge \frac{1}{2}(1-\frac{1}{m_j})c_1$
    & $E_2\le E_4\le E_1\le E_3\le E_5$  \rule{0em}{1.6em} \\[1.4ex] \hline
   Region$_{\,(9)}$
    & $c_1\le 1\,,\;\;
       \frac{(m_j-1)^2}{m_j(3m_j-1)}
        \le c_2 \le \frac{1}{2}(1-\frac{1}{m_j})c_1$
    & $E_2\le E_4\le E_3\le E_1\le E_5$  \rule{0em}{1.6em} \\[1.4ex] \hline
   Region$_{\,(10)}$
    & $c_1\le 1\,,\;\;
       (0<)\,c_2\le \frac{(m_j-1)^2}{m_j(3m_j-1)}\,,\;\;
       c_2\ge (1-\frac{1}{m_j})-c_1$
    & $E_2\le E_3\le E_4\le E_1\le E_5$  \rule{0em}{1.6em} \\[1.4ex] \hline
   Region$_{\,(11)}$
    & $\frac{(m_j-1)^2}{m_j(3m_j-1)}\le c_2\,(< c_1)\,,\;\;
       c_2\le (1-\frac{1}{m_j})-c_1$
    & $E_2\le E_1\le E_4\le E_3\le E_5$  \rule{0em}{1.6em} \\[1.4ex] \hline
   Region$_{\,(12)}$
    & $\frac{1}{2}(1-\frac{1}{m_j})c_1\le c_2\,(<c_1)\,,\;\;
        c_2\le \frac{(m_j-1)^2}{m_j(3m_j-1)}$
    & $E_2\le E_1\le E_3\le E_4\le E_5$  \rule{0em}{1.6em} \\[1.4ex] \hline
   Region$_{\,(13)}$
    & $(0<)\,c_2\le \frac{1}{2}(1-\frac{1}{m_j})c_1\,,\;\;
       c_2\le (1-\frac{1}{m_j})-c_1$
    & $E_2\le E_3\le E_1\le E_4\le E_5$  \rule{0em}{1.6em} \\[1.4ex] \hline
 \end{tabular}
 } 

 \bigskip

 \noindent
 Once the region $(c_1,c_2)$ lies is fixed,
 the decomposition of the integral $\int_{r_1=t^{c_1}}^{c_2}$
   is determined by the order of $E_1,\,E_2,\,E_3,\,E_4,\,E_5$;
  and for each part of the integral,
   the dominant $t$-order term in each factor of the integrand
  is determined.
 Similar computations as those in Part (a) give then
  the following bounds for $\|\varepsilon^t\|_{L^{6/5}}$ on $Q_j^t$:

 \bigskip

 \centerline{\small\it
 \begin{tabular}{|l|c|} \hline
  $(c_1,c_2)$    &   $\|\varepsilon^t\|_{L^{6/5}}\;\le\;\bullet$
                            \rule{0em}{1.2em} \\[.6ex] \hline\hline
  Region$_{\,(1)}$
   & $O(\,t^{\frac{8}{3}(1-\frac{1}{m_j})-\frac{11}{3}c_2}\,)$
                              \rule{0em}{2em} \\[.6ex] \hline
  Region$_{\,(2)}$
   & $O(\,t^{(\frac{8}{3}-\frac{11}{6}c_1)(1-\frac{1}{m_j})
             +\frac{5}{6}(c_1-c_2)}\,)\,
      +\, O(\,t^{\frac{8}{3}(1-\frac{1}{m_j})-\frac{11}{3}c_2}\,)$
                              \rule{0em}{2em} \\[.6ex] \hline
  Region$_{\,(3)}$
   & $O(\,t^{(\frac{8}{3}-\frac{11}{6}c_1)(1-\frac{1}{m_j})
             +\frac{5}{6}(c_1-c_2)}\,)\,
      +\, O(\,t^{\frac{7}{2}-\frac{8}{3m_j}-\frac{9}{2}c_2}\,)\,
      +\, O(\,t^{\frac{8}{3}(1-\frac{1}{m_j})-\frac{11}{3}c_2}\,)$
                              \rule{0em}{2em} \\[.6ex] \hline
  Region$_{\,(4)}$
   & $O(\,t^{(\frac{8}{3}-\frac{11}{6}c_1)(1-\frac{1}{m_j})
             +\frac{5}{6}(c_1-c_2)}\,)\,
      +\, O(\,t^{\frac{11}{6}(1-\frac{1}{m_j})
                 +(\frac{5}{6}-2m_j+\frac{5}{3(m_j-1)})c_2}\,)$
      \hspace{4.4em}          \rule{0em}{2em} \\ [.6ex]
   & \hspace{4em}
     $+\, O(\,t^{\frac{8}{3}(1-\frac{1}{m_j})-\frac{11}{3}c_2}\,)\,
      +\, O(\,t^{1-\frac{1}{m_j}-\frac{c_2}{3}}\,)$\hspace{2em}
     for $2\le m_j\le 5\,$,   \rule{0em}{1em} \\[.6ex]
  & $O(\,t^{(\frac{8}{3}-\frac{11}{6}c_1)(1-\frac{1}{m_j})
             +\frac{5}{6}(c_1-c_2)}\,)\,
     +\, O(\,t^{\frac{11}{6}(1-\frac{1}{m_j}-\frac{5}{6}c_2)}\,
             |\log t|^{5/6}\,)$
                              \rule{0em}{2em} \\ [.6ex]
   & \hspace{2em}
     $+\, O(\,t^{\frac{8}{3}(1-\frac{1}{m_j})-\frac{11}{3}c_2}\,)\,
      +\, O(\,t^{1-\frac{1}{m_j}-\frac{c_2}{3}}\,)$\hspace{2em}
     for $m_j=6\,$,           \rule{0em}{1em} \\[.6ex]
  & $O(\,t^{(\frac{8}{3}-\frac{11}{6}c_1)(1-\frac{1}{m_j})
             +\frac{5}{6}(c_1-c_2)}\,)\,
     +\, O(\,t^{\frac{5}{6}(2-\frac{1}{m_j}-c_2)}\,)$
      \hspace{11em}            \rule{0em}{2em} \\ [.6ex]
   & \hspace{2em}
     $+\, O(\,t^{\frac{8}{3}(1-\frac{1}{m_j})-\frac{11}{3}c_2}\,)\,
      +\, O(\,t^{1-\frac{1}{m_j}-\frac{c_2}{3}}\,)$\hspace{2em}
     for $m_j\ge 7$           \rule{0em}{1em} \\[.6ex] \hline
  Region$_{\,(5)}$
   & $O(\,t^{\frac{8}{3}(1-\frac{1}{m_j})
             -\frac{5}{6}c_2 +(\frac{11}{6m_j}-1)c_1}\,)\,
      +\, O(\,t^{\frac{5}{3}-\frac{5}{6m_j}-\frac{5}{6}c_2}\,)$
      \hspace{10em}          \rule{0em}{2em} \\ [.6ex]
   & \hspace{1em}
     $+O(\,t^{\frac{11}{6}(1-\frac{1}{m_j})
             +\frac{(11-m_j)c_2}{3(m_j-1)}}\,)\,
      +\, O(\,t^{1-\frac{1}{m_j}-\frac{1}{3}c_2}\,)$\hspace{2em}
     for $2\le m_j\le 10\,$,   \rule{0em}{1em} \\[.6ex]
  & $O(\,t^{\frac{8}{3}(1-\frac{1}{m_j})
             -\frac{5}{6}c_2 +(\frac{11}{6m_j}-1)c_1}\,)\,
      +\, O(\,t^{\frac{5}{3}-\frac{5}{6m_j}-\frac{5}{6}c_2}\,)$
      \hspace{10em}          \rule{0em}{2em} \\ [.6ex]
   & \hspace{2em}
     $+\, O(\,t^{5/3}\,|\log t|^{5/6}\,)\,
      +\, O(\,t^{1-\frac{1}{m_j}-\frac{1}{3}c_2}\,)$\hspace{2em}
     for $m_j=11\,$,   \rule{0em}{1em} \\[.6ex]
  & $O(\,t^{\frac{8}{3}(1-\frac{1}{m_j})
             -\frac{5}{6}c_2 +(\frac{11}{6m_j}-1)c_1}\,)\,
      +\, O(\,t^{\frac{5}{3}-\frac{5}{6m_j}-\frac{5}{6}c_2}\,)$
      \hspace{10em}          \rule{0em}{2em} \\ [.6ex]
   & \hspace{1.7em}
     $+\, O(\,t^{\frac{1}{6}(1-\frac{1}{m_j})(10+\frac{11}{m_j})
                 +\frac{1}{6}(1-\frac{11}{m_j})c_2}\,)
      +\, O(\,t^{1-\frac{1}{m_j}-\frac{1}{3}c_2}\,)$\hspace{2em}
     for $m_j\ge 12$   \rule{0em}{1em} \\[.6ex] \hline
  Region$_{\,(6)}$
   & $O(\,t^{\frac{7}{2}-\frac{8}{3m_j}-\frac{9}{2}c_2}\,)\,
      +\, O(\,t^{\frac{8}{3}(1-\frac{1}{m_j})-\frac{11}{3}c_2}\,)$
                              \rule{0em}{2em} \\[.6ex] \hline
  Region$_{\,(7)}$
   & $O(\,t^{\frac{8}{3}(1-\frac{1}{m_j})-\frac{11}{3}c_2}\,)$
                              \rule{0em}{2em} \\[.6ex] \hline
  Region$_{\,(8)}$
   & $O(\,t^{\frac{8}{3}(1-\frac{1}{m_j})-\frac{11}{3}c_2}\,)\,
      +\, O(\,t^{1-\frac{1}{m_j}-\frac{1}{3}c_2}\,)$
                              \rule{0em}{2em} \\[.6ex] \hline
  Region$_{\,(9)}$
   & $O(\,t^{\frac{11}{6}(1-\frac{1}{m_j})-\frac{17}{6}c_2
             +\frac{5m_jc_2}{3(m_j-1)}}\,)\,
      +\, O(\,t^{\frac{8}{3}(1-\frac{1}{m_j})-\frac{11}{3}c_2}\,)\,
      +\, O(\,t^{1-\frac{1}{m_j}-\frac{1}{3}c_2}\,)$ \hspace{5em}
                              \rule{0em}{2em} \\[.6ex]
   & \hspace{16em} for $2\le m_j\le 5$,       \\[.6ex]
  & $O(\,t^{\frac{11}{6}(1-\frac{1}{m_j})-\frac{5}{6}c_2}\,
         |\log t|^{5/6}\,)\,
      +\, O(\,t^{\frac{8}{3}(1-\frac{1}{m_j})-\frac{11}{3}c_2}\,)\,
      +\, O(\,t^{1-\frac{1}{m_j}-\frac{1}{3}c_2}\,)$
                              \rule{0em}{2em} \\[.6ex]
   & \hspace{14em} for $m_j=6$,       \\[.6ex]
  & $O(\,t^{\frac{11}{6}(1-\frac{1}{m_j})
             -(\frac{1}{m_j}-\frac{1}{6})c_1
             -\frac{5}{6}c_2}\,)\,
      +\, O(\,t^{\frac{8}{3}(1-\frac{1}{m_j})-\frac{11}{3}c_2}\,)\,
      +\, O(\,t^{1-\frac{1}{m_j}-\frac{1}{3}c_2}\,)$ \hspace{5em}
                              \rule{0em}{2em} \\[.6ex]
   & \hspace{14em} for $m_j\ge 7$       \\[.6ex] \hline
 \end{tabular}
  }


 \centerline{\small\it
 \begin{tabular}{|l|c|} \hline
  Region$_{\,(10)}$
   & $O(\,t^{\frac{8}{3}(1-\frac{1}{m_j})-(1-\frac{1}{m_j})^2
             -(\frac{2}{3}+\frac{1}{m_j})c_2}\,)\,
      +\, O(\,t^{1-\frac{1}{m_j}
                 +(\frac{1}{m_j}+\frac{2}{3})\frac{2c_2m_j}{m_j-1}}\,)\,
      +\, O(\,t^{1-\frac{1}{m_j}-\frac{1}{3}c_2}\,)$ \hspace{1em}
                              \rule{0em}{2em} \\[.6ex]
   & \hspace{16em} for $2\le m_j\le 5$,       \\[.6ex]
  & $O(\,t^{\frac{55}{36}-\frac{5}{6}c_2}\,|\log t|^{5/6}\,)\,
      +\, O(\,t^{1-\frac{1}{m_j}
                 +(\frac{1}{m_j}+\frac{2}{3})\frac{2c_2m_j}{m_j-1}}\,)\,
      +\, O(\,t^{1-\frac{1}{m_j}-\frac{1}{3}c_2}\,)$
                              \rule{0em}{2em} \\[.6ex]
   & \hspace{14em} for $m_j=6$,       \\[.6ex]
  & $O(\,t^{\frac{11}{6}(1-\frac{1}{m_j})
            +(\frac{1}{m_j}-\frac{1}{6})c_1-\frac{5}{6}c_2}\,)\,
      +\, O(\,t^{1-\frac{1}{m_j}
                 +(\frac{1}{m_j}+\frac{2}{3})\frac{2c_2m_j}{m_j-1}}\,)\,
      +\, O(\,t^{1-\frac{1}{m_j}-\frac{1}{3}c_2}\,)$ \hspace{2.4em}
                              \rule{0em}{2em} \\[.6ex]
   & \hspace{14em} for $m_j\ge 7$       \\[.6ex] \hline
  Region$_{\,(11)}$
   & $O(\,t^{1-\frac{1}{m_j}-\frac{1}{3}c_2}\,)$
                              \rule{0em}{2em} \\[.6ex] \hline
  Region$_{\,(12)}$
   & $O(\,t^{1-\frac{1}{m_j}-\frac{1}{3}c_2}\,)$
                              \rule{0em}{2em} \\[.6ex] \hline
  Region$_{\,(13)}$
   & $O(\,t^{1-\frac{1}{m_j}+\frac{4}{3}c_2+\frac{10c_2}{3(m_j-1)}}\,)\,
      +\, O(\,t^{1-\frac{1}{m_j}-\frac{1}{3}c_2}\,)$
                              \rule{0em}{2em} \\[.6ex] \hline
 \end{tabular}
 } 

 %

 \bigskip
 \newpage

 \noindent
 $(b.4)$
 {\it Estimating $\|\varepsilon^t\|_{L^1}$ on $Q_j^t$.}\hspace{1em}
 Except an adjustment of powers
   due the change from $L^{6/5}$-norm to $L^1$-norm,
  the computations in this case is completely the same
  as those in Part~(b.3):

 \bigskip

 \centerline{\small\it
 \begin{tabular}{|l|c|} \hline
  $(c_1,c_2)$    &   $\|\varepsilon^t\|_{L^1}\;\le\;\bullet$
                            \rule{0em}{1.2em} \\[.6ex] \hline\hline
  Region$_{\,(1)}$
   & $O(\,t^{3(1-\frac{1}{m_j})-4c_2}\,)$
                              \rule{0em}{2em} \\[.6ex] \hline
  Region$_{\,(2)}$
   & $O(\,t^{\frac{3}{2}-c_2}\,|\log t|\,)\,
      +\, O(\,t^{3(1-\frac{1}{m_j})-4c_2}\,)$
                              \rule{0em}{2em} \\[.6ex]
   & \hspace{14em} for $m_j=2$,       \\[.6ex]
  & $O(\,t^{3(1-\frac{1}{m_j})+(\frac{2}{m_j}-1)c_1-c_2}\,)\,
     +\, O(\,t^{3(1-\frac{1}{m_j})-4c_2}\,)$
                              \rule{0em}{2em} \\[.6ex]
   & \hspace{14em} for $m_j\ge 3$       \\[.6ex] \hline
  Region$_{\,(3)}$
   & $O(\,t^{\frac{3}{2}-c_2}\,|\log t|\,)\,
      +\, O(\,t^{4-\frac{3}{m_j}-5c_2}\,)
      +\, O(\,t^{3(1-\frac{1}{m_j})-4c_2}\,)$
                              \rule{0em}{2em} \\[.6ex]
   & \hspace{14em} for $m_j=2$,       \\[.6ex]
  & $O(\,t^{3(1-\frac{1}{m_j})+(\frac{2}{m_j}-1)c_1-c_2}\,)\,
     +\, O(\,t^{4-\frac{3}{m_j}-5c_2}\,)
     +\, O(\,t^{3(1-\frac{1}{m_j})-4c_2}\,)$
                              \rule{0em}{2em} \\[.6ex]
   & \hspace{14em} for $m_j\ge 3$       \\[.6ex] \hline
  Region$_{\,(4)}$
   & $O(\,t^{\frac{3}{2}-c_2}\,|\log t|\,)\,
      +\, O(\,t^{2(1-\frac{1}{m_j})-c_2+\frac{2c_2}{m_j-1})}\,)\,
      +\, O(\,t^{3(1-\frac{1}{m_j})-4c_2}\,)\,
      +\, O(\,t^{1-\frac{1}{m_j}}\,)$
                              \rule{0em}{2em} \\[.6ex]
   & \hspace{14em} for $m_j=2$,               \\[.6ex]
  & $O(\,t^{3(1-\frac{1}{m_j})+(\frac{2}{m_j}-1)c_1-c_2}\,)\,
      +\, O(\,t^{2(1-\frac{1}{m_j})-c_2+\frac{2c_2}{m_j-1})}\,)\,
      +\, O(\,t^{3(1-\frac{1}{m_j})-4c_2}\,)\,
      +\, O(\,t^{1-\frac{1}{m_j}}\,)$
                              \rule{0em}{2em} \\[.6ex]
   & \hspace{14em} for $m_j\ge 3$             \\[.6ex] \hline
  Region$_{\,(5)}$
   & $O(\,t^{\frac{3}{2}-c_2}|\log t|\,)\,
      +\, O(\,t^{2(1-\frac{1}{m_j})+\frac{4c_2}{m_j-1}}\,)\,
      +\, O(\,t^{1-\frac{1}{m_j}}\,)$
                              \rule{0em}{2em} \\[.6ex]
   & \hspace{14em} for $m_j=2$,               \\[.6ex]
  & $O(\,t^{3(1-\frac{1}{m_j})+(\frac{2}{m_j}-1)c_1-c_2}\,)\,
     +\, O(\,t^{2-\frac{1}{m_j}-c_2}\,)\,
     +\, O(\,t^{2(1-\frac{1}{m_j})+\frac{4c_2}{m_j-1}}\,)\,
     +\, O(\,t^{1-\frac{1}{m_j}}\,)$
                              \rule{0em}{2em} \\[.6ex]
   & \hspace{14em} for $m_j\ge 3$             \\[.6ex] \hline
  Region$_{\,(6)}$
   & $O(\,t^{4-\frac{3}{m_j}-5c_2}\,)\,
      +\, O(\,t^{3(1-\frac{1}{m_j})-4c_2}\,)$
                              \rule{0em}{2em} \\[.6ex] \hline
  Region$_{\,(7)}$
   & $O(\,t^{3(1-\frac{1}{m_j})-4c_2}\,)$
                              \rule{0em}{2em} \\[.6ex] \hline
  Region$_{\,(8)}$
   & $O(\,t^{3(1-\frac{1}{m_j})-4c_2}\,)\,
      +\, O(\,t^{1-\frac{1}{m_j}}\,)$
                              \rule{0em}{2em} \\[.6ex] \hline
  Region$_{\,(9)}$
   & $O(\,t^{2(1-\frac{1}{m_j})-c_2-\frac{2c_2}{m_j-1}}\,)\,
      +\, O(\,t^{3(1-\frac{1}{m_j})-4c_2}\,)\,
      +\, O(\,t^{1-\frac{1}{m_j}}\,)$
                              \rule{0em}{2em} \\[.6ex] \hline
  Region$_{\,(10)}$
   & $O(\,t^{2-(1+\frac{1}{m_j})(c_2+\frac{1}{m_j})}\,)\,
      +\, O(\,t^{1-\frac{1}{m_j}+\frac{2(m_j+1)}{m_j-1}\,c_2}\,)\,
      +\, O(\,t^{1-\frac{1}{m_j}}\,)$
                              \rule{0em}{2em} \\[.6ex] \hline
  Region$_{\,(11)}$
   & $O(\,t^{1-\frac{1}{m_j}}\,)$
                              \rule{0em}{2em} \\[.6ex] \hline
  Region$_{\,(12)}$
   & $O(\,t^{1-\frac{1}{m_j}}\,)$
                              \rule{0em}{2em} \\[.6ex] \hline
  Region$_{\,(13)}$
   & $O(\,t^{1-\frac{1}{m_j}+\frac{2(m_j+1)}{m_j-1}\,c_2}\,)\,
      +\, O(\,t^{1-\frac{1}{m_j}}\,)$
                              \rule{0em}{2em} \\[.6ex] \hline
 \end{tabular}
 } 

 \bigskip

 %

 \bigskip

 \noindent
 $(b.5)$
 {\it Estimating $\|d\varepsilon^t\|_{L^6}$ on $Q_j^t$.}\hspace{1em}
 It follows from Part (b.1) that
 \begin{eqnarray*}
  \lefteqn{\|\varepsilon^t\|_{L^6}\; \le\;
   \left(
    2\pi m_j l_j\,
    \int_{r_1=t^{c_1}}^{t^{c_2}}
     \left|\,
      O(\,t^{1-\frac{1}{m_j}}\,)\,
      \left[
        O(\,t^{-3c_2}\,)\, +\, O(\,r_1^{\frac{1}{m_j}-2}\,)
      \right]
     \right|^6 \right.                          }\\[1.2ex]
  && \hspace{7em} \cdot
      \left( 1\,
        +\, O(\,t^{2(1-\frac{1}{m_j})}\,)\,
           \left[
            O(\,t^{-2c_2}\,)\,
            +\, O(\,r_1^{\frac{1}{m_j}-1}\,)
                               \right]^2\, \right)^{1/2} \\[1.2ex]
  && \hspace{16em} \left.
      \cdot
      \left( r_1^2\, +\, O(\,t^{2(1-\frac{1}{m_j}-c_2)}\,)\,
       \right)^{1/2}\,
      dr_1\,\right)^{1/6}\,.
 \end{eqnarray*}
 Similar to the discussion in Part~(b.3),
 the equalities
  $$
   \begin{array}{ccccc}
    O(\,t^{-3c_2}\,)\; =\;  O(\,r_1^{\frac{1}{m_j}-2}\,)\,;
     && O(\,t^{-2c_2}\,)\; =\; O(\,r_1^{\frac{1}{m_j}-1}\,)\,, \\[1.2ex]
    1\; =\; O(\,t^{2(1-\frac{1}{m_j}-2c_2)}\,)\,,
     && 1\; =\; O(\,t^{2(1-\frac{1}{m_j})}r_1^{2(\frac{1}{m_j}-1)}\,)\,;
     && r_1^2\; =\; O(\,t^{2(1-\frac{1}{m_j}-c_2)}\,)
   \end{array}
  $$
  are solved by
  $$
   r_1\,=\,O(\,t^{\frac{3c_2m_j}{2m_j-1}}\,)\,;\hspace{1em}
   r_1\,=\,O(\,t^{\frac{2c_2m_j}{m_j-1}}\,)\,;\hspace{1em}
   c_2\,=\,\frac{1}{2}(1-\frac{1}{m_j})\,,\hspace{1em}
   r_1\,=\,O(t)\,;\hspace{1em}
   r_1\,=\,O(\,t^{1-\frac{1}{m_j}-c_2}\,)
  $$
  respectively.
 Recall/denote the $t$-exponents:
  $$
   E_1\, :=\, c_1\,,\hspace{.8em}
   E_2\, :=\, c_2\,,\hspace{.8em}
   E_3\, :=\, \frac{2c_2m_j}{m_j-1}\,,\hspace{.8em}
   E_3^{\prime}\, :=\, \frac{3c_2m_j}{2m_j-1}\,,\hspace{.8em}
   E_4\, :=\, 1-\frac{1}{m_j}-c_2\,,\hspace{.8em}
   E_5\, :=\, 1\,.
  $$
 Then the equations
  $E_i$ (or $E_3^{\prime}$) $=E_j$ (or $E_3^{\prime}$),
   $1\le i<j\le 5$,
  refine the previous $13$-region decomposition of
  $\{(c_1,c_2)\in ({\Bbb R}_+)^2\,:\, c_1>c_2\}$
  further into 26 regions:

 \bigskip

 \centerline{\footnotesize\it
 \noindent
 \begin{tabular}{|l|c|c|} \hline
   $\hspace{.8em}$region      & description in $\{0<c_2<c_1\}$
    & order                   \rule{0em}{1.2em}\\[.6ex] \hline\hline
   Region$_{\,(1)_1}$
    & $1\le c_2\, (< c_1)\,,\;\;
       c_2\ge \frac{2}{3}(1-\frac{1}{2m_j})c_1$
    & {\tiny $E_4\le E_5\le E_2\le E_1\le E_3^{\prime}\le E_3$}
                                  \rule{0em}{1.6em} \\[1.4ex] \hline
   Region$_{\,(1)_2}$
    & $1\le c_2 \le \frac{2}{3}(1-\frac{1}{2m_j})c_1\,,\;\;
       c_2\ge \frac{1}{2}(1-\frac{1}{m_j})c_1$
    & {\tiny $E_4\le E_5\le E_2\le E_3^{\prime}\le E_1\le E_3$}
                                  \rule{0em}{1.6em} \\[1.4ex] \hline
   Region$_{\,(2)}$
    & $1\le c_2\le \frac{1}{2}(1-\frac{1}{m_j})c_1$
    & {\tiny $E_4\le E_5\le E_2\le E_3^{\prime}\le E_3\le E_1$}
                                  \rule{0em}{1.6em} \\[1.4ex] \hline
   Region$_{\,(3)_1}$
    & $\frac{2}{3}(1-\frac{1}{2m_j}) \le c_2\le 1\,,\;\;
       c_2\le \frac{1}{2}(1-\frac{1}{m_j})c_1$
    & {\tiny $E_4\le E_2\le E_5\le E_3^{\prime}\le E_3\le E_1$}
                                  \rule{0em}{1.6em} \\[1.4ex] \hline
   Region$_{\,(3)_2}$
    & $\frac{1}{2}(1-\frac{1}{m_j})
        \le c_2\le \frac{2}{3}(1-\frac{1}{2m_j})\,,\;\;
       c_2\le \frac{1}{2}(1-\frac{1}{m_j})c_1$
    & {\tiny $E_4\le E_2\le E_3^{\prime}\le E_5\le E_3\le E_1$}
                                  \rule{0em}{1.6em} \\[1.4ex] \hline
   Region$_{\,(4)_1}$
    & $c_1\ge 1\,,\;\;
       \frac{(m_j-1)(2m_j-1)}{m_j(5m_j-1)}
        \le c_2 \le \frac{1}{2}(1-\frac{1}{m_j})$
    & {\tiny $E_2\le E_4\le E_3^{\prime}\le E_3\le E_5\le E_1$}
                                  \rule{0em}{1.6em} \\[1.4ex] \hline
   Region$_{\,(4)_2}$
    & $c_1\ge 1\,,\;\;
       \frac{(m_j-1)^2}{m_j(3m_j-1)}
        \le c_2 \le \frac{(m_j-1)(2m_j-1)}{m_j(5m_j-1)}$
    & {\tiny $E_2\le E_3^{\prime}\le E_4\le E_3\le E_5\le E_1$}
                                  \rule{0em}{1.6em} \\[1.4ex] \hline
   Region$_{\,(5)}$
    & $c_1\ge 1\,,\;\,
       0< c_2\le  \frac{(m_j-1)^2}{m_j(3m_j-1)}$
    & {\tiny $E_2\le E_3^{\prime}\le E_3\le E_4\le E_5\le E_1$}
                                  \rule{0em}{1.6em} \\[1.4ex] \hline
   Region$_{\,(6)_1}$
    & $c_1\ge 1\,,\;\;
       \frac{2}{3}(1-\frac{1}{2m_j})c_1 \le c_2 \le 1$
    & {\tiny $E_4\le E_2\le E_5\le E_1\le E_3^{\prime}\le E_3$}
                                  \rule{0em}{1.6em} \\[1.4ex] \hline
   Region$_{\,(6)_2}$
    & $\frac{2}{3}(1-\frac{1}{2m_j})
        \le c_2 \le \frac{2}{3}(1-\frac{1}{2m_j})c_1\,,\;\;
       \frac{1}{2}(1-\frac{1}{m_j})c_1 \le c_2 \le 1$
    & {\tiny $E_4\le E_2\le E_5\le E_3^{\prime}\le E_1\le E_3$}
                                  \rule{0em}{1.6em} \\[1.4ex] \hline
   Region$_{\,(6)_3}$
    & $c_1\ge 1\,,\;\;
       \frac{1}{2}(1-\frac{1}{m_j})c_1
         \le c_2 \le \frac{2}{3}(1-\frac{1}{2m_j})$
    & {\tiny $E_4\le E_2\le E_3^{\prime}\le E_5\le E_1\le E_3$}
                                  \rule{0em}{1.6em} \\[1.4ex] \hline
   Region$_{\,(7)_1}$
    & $c_1\le 1\,,\;\;
       \frac{2}{3}(1-\frac{1}{2m_j})\le c_2\, (<c_1)$
    & {\tiny $E_4\le E_2\le E_1\le E_5\le E_3^{\prime}\le E_3$}
                                  \rule{0em}{1.6em} \\[1.4ex] \hline
   Region$_{\,(7)_2}$
    & $\frac{2}{3}(1-\frac{1}{2m_j})c_1
        \le c_2 \le \frac{2}{3}(1-\frac{1}{2m_j})\,,\;\;
       \frac{1}{2}(1-\frac{1}{m_j})\le c_2\, (<c_1)$
    & {\tiny $E_4\le E_2\le E_1\le E_3^{\prime}\le E_5\le E_3$}
                                  \rule{0em}{1.6em} \\[1.4ex] \hline
   Region$_{\,(7)_3}$
    & $c_1\le 1\,,\;\;
       \frac{1}{2}(1-\frac{1}{m_j})
         \le c_2 \le \frac{2}{3}(1-\frac{1}{2m_j})c_1$
    & {\tiny $E_4\le E_2\le E_3^{\prime}\le E_1\le E_5\le E_3$}
                                  \rule{0em}{1.6em} \\[1.4ex] \hline
   Region$_{\,(8)_1}$
    & $(1-\frac{1}{m_j})-c_1
        \le c_2\le \frac{1}{2}(1-\frac{1}{m_j})\,,\;\;
       c_2\ge \frac{2}{3}(1-\frac{1}{2m_j})c_1$
    & {\tiny $E_2\le E_4\le E_1\le E_3^{\prime}\le E_3\le E_5$}
                                  \rule{0em}{1.6em} \\[1.4ex] \hline
   Region$_{\,(8)_2}$
    & $\frac{(m_j-1)(2m_j-1)}{m_j(5m_j-1)}
        \le c_2 \le \frac{2}{3}(1-\frac{1}{2m_j})c_1\,,\;\;
       \frac{1}{2}(1-\frac{1}{m_j})c_1
        \le c_2 \le \frac{1}{2}(1-\frac{1}{m_j}) $
    & {\tiny $E_2\le E_4\le E_3^{\prime}\le E_1\le E_3\le E_5$}
                                  \rule{0em}{1.6em} \\[1.4ex] \hline
   Region$_{\,(8)_3}$
    & $(1-\frac{1}{m_j})-c_1
        \le c_2\le \frac{(m_j-1)(2m_j-1)}{m_j(5m_j-1)}\,,\;\;
       c_2\ge \frac{1}{2}(1-\frac{1}{m_j})c_1$
    & {\tiny $E_2\le E_3^{\prime}\le E_4\le E_1\le E_3\le E_5$}
                                  \rule{0em}{1.6em} \\[1.4ex] \hline
   Region$_{\,(9)_1}$
    & $c_1\le 1\,,\;\;
       \frac{(m_j-1)(2m_j-1)}{m_j(5m_j-1)}
        \le c_2 \le \frac{1}{2}(1-\frac{1}{m_j})c_1$
    & {\tiny $E_2\le E_4\le E_3^{\prime}\le E_3\le E_1\le E_5$}
                                  \rule{0em}{1.6em} \\[1.4ex] \hline
   Region$_{\,(9)_2}$
    & $c_1\le 1\,,\;\;
       \frac{(m_j-1)^2}{m_j(3m_j-1)}
        \le c_2 \le \frac{1}{2}(1-\frac{1}{m_j})c_1\,,\;\;
       c_2\le \frac{(m_j-1)(2m_j-1)}{m_j(5m_j-1)}$
    & {\tiny $E_2\le E_3^{\prime}\le E_4\le E_3\le E_1\le E_5$}
                                  \rule{0em}{1.6em} \\[1.4ex] \hline
   Region$_{\,(10)}$
    & $c_1\le 1\,,\;\;
       (0<)\,c_2\le \frac{(m_j-1)^2}{m_j(3m_j-1)}\,,\;\;
       c_2\ge (1-\frac{1}{m_j})-c_1$
    & {\tiny $E_2\le E_3^{\prime}\le E_3\le E_4\le E_1\le E_5$}
                                  \rule{0em}{1.6em} \\[1.4ex] \hline
   Region$_{\,(11)_1}$
    & $\frac{(m_j-1)(2m_j-1)}{m_j(5m_j-1)}\le c_2\,(< c_1)\,,\;\;
       c_2\le (1-\frac{1}{m_j})-c_1$
    & {\tiny $E_2\le E_1\le E_4\le E_3^{\prime}\le E_3\le E_5$}
                                  \rule{0em}{1.6em} \\[1.4ex] \hline
   Region$_{\,(11)_2}$
    & $\frac{(m_j-1)^2}{m_j(3m_j-1)}\le c_2\,(< c_1)\,,\;\;
       \frac{2}{3}(1-\frac{1}{2m_j})c_1
        \le c_2 \le \frac{(m_j-1)(2m_j-1)}{m_j(5m_j-1)}$
    & {\tiny $E_2\le E_1\le E_3^{\prime}\le E_4\le E_3\le E_5$}
                                  \rule{0em}{1.6em} \\[1.4ex] \hline
   Region$_{\,(11)_3}$
    & $\frac{(m_j-1)^2}{m_j(3m_j-1)}
        \le c_2 < \frac{2}{3}(1-\frac{1}{2m_j})\,,\;\;
       c_2\le (1-\frac{1}{m_j})-c_1$
    & {\tiny $E_2\le E_3^{\prime}\le E_1\le E_4\le E_3\le E_5$}
                                  \rule{0em}{1.6em} \\[1.4ex] \hline
   Region$_{\,(12)_1}$
    & $\frac{2}{3}(1-\frac{1}{2m_j})c_1\le c_2\,(<c_1)\,,\;\;
        c_2\le \frac{(m_j-1)^2}{m_j(3m_j-1)}$
    & {\tiny $E_2\le E_1\le E_3^{\prime}\le E_3\le E_4\le E_5$}
                                  \rule{0em}{1.6em} \\[1.4ex] \hline
   Region$_{\,(12)_2}$
    & $\frac{1}{2}(1-\frac{1}{m_j})c_1
        \le c_2 <\frac{2}{3}(1-\frac{1}{2m_j})c_1\,,\;\;
       c_2\le \frac{(m_j-1)^2}{m_j(3m_j-1)}$
    & {\tiny $E_2\le E_3^{\prime}\le E_1\le E_3\le E_4\le E_5$}
                                  \rule{0em}{1.6em} \\[1.4ex] \hline
   Region$_{\,(13)}$
    & $(0<)\,c_2\le \frac{1}{2}(1-\frac{1}{m_j})c_1\,,\;\;
       c_2\le (1-\frac{1}{m_j})-c_1$
    & {\tiny $E_2\le E_3^{\prime}\le E_3\le E_1\le E_4\le E_5$}
                                  \rule{0em}{1.6em} \\[1.4ex] \hline
 \end{tabular}
 } 

 \bigskip
 \vspace{3em}

 \noindent
 Similar computations as those in Part~(b.3) gives then
  the following bounds for $\|d\varepsilon^t\|_{L^6}$ on $Q_j^t$:

 \bigskip

 \centerline{\small\it
 \begin{tabular}{|l|c|} \hline
  $(c_1,c_2)$    &   $\|d\varepsilon^t\|_{L^6}\;\le\;\bullet$
                            \rule{0em}{1.2em} \\[.6ex] \hline\hline
  Region$_{\,(1)_1}$
   & $O(\,t^{\frac{4}{3}(1-\frac{1}{m_j})-\frac{10}{3}c_2}\,)$
                              \rule{0em}{2em} \\[.6ex] \hline
  Region$_{\,(1)_2}$
   & $O(\,t^{\frac{4}{3}(1-\frac{1}{m_j})
            +(\frac{1}{m_j}-\frac{11}{6})c_1-\frac{1}{2}c_2}\,)\,
      +\, O(\,t^{\frac{4}{3}(1-\frac{1}{m_j})-\frac{10}{3}c_2}\,)$
                              \rule{0em}{2em} \\[.6ex] \hline
  Region$_{\,(2)}$
   & $O(\,t^{\frac{4}{3}(1-\frac{1}{m_j})
             +(\frac{7}{6m_j}-2)c_1-\frac{1}{6}c_2}\,)\,$
                              \rule{0em}{2em} \\[.6ex]
   & $+\, O(\,t^{\frac{4}{3}(1-\frac{1}{m_j})
                 -(\frac{25}{6}+\frac{5}{3(m_j-1)})c_2}\,)\,
      +\, O(\,t^{\frac{4}{3}(1-\frac{1}{m_j})-\frac{10}{3}c_2}\,)$
                              \rule{0em}{1.2em} \\[.6ex] \hline
  Region$_{\,(3)_1}$
   & $O(\,t^{\frac{4}{3}(1-\frac{1}{m_j})+(\frac{7}{6m_j}-2)c_1
             -\frac{1}{6}c_2}\,)\,
      +\, O(\,t^{-\frac{25}{6}-\frac{5}{3(m_j-1)}}\,)\,$
                              \rule{0em}{2em} \\[.6ex]
   & $+\, O(\,t^{\frac{3}{2}-\frac{4}{3m_j}-\frac{7}{2}c_2}\,)\,
      +\, O(\,t^{\frac{4}{3}(1-\frac{1}{m_j})-\frac{10}{3}c_2}\,)$
                              \rule{0em}{1.2em} \\[.6ex] \hline
  Region$_{\,(3)_2}$
   & $O(\,t^{\frac{4}{3}(1-\frac{1}{m_j})
             +(\frac{1}{m_j}-\frac{11}{6})c_1-\frac{1}{6}c_2}\,)\,
      +\, O(\,t^{\frac{4}{3}(1-\frac{1}{m_j})
                 -(\frac{25}{6}+\frac{5}{3(m_j-1)})c_2}\,)\,$
                              \rule{0em}{2em} \\[.6ex]
   & $+\, O(\,t^{-\frac{1}{2}-\frac{1}{3m_j}-\frac{1}{2}c_2}\,)\,
      +\, O(\,t^{\frac{4}{3}(1-\frac{1}{m_j})-\frac{23}{6}c_2}\,)$
                              \rule{0em}{1.2em} \\[.6ex] \hline
  Region$_{\,(4)_1}$
   & $O(\,t^{\frac{4}{3}(1-\frac{1}{m_j})+(\frac{7}{6m_j}-2)c_1
             -\frac{1}{6}c_2}\,)\,
      +\, O(\,t^{-\frac{2}{3}-\frac{1}{6m_j}-\frac{1}{6}c_2}\,)\,$
                              \rule{0em}{2em} \\[.6ex]
   & $+\, O(\,t^{\frac{7}{6}(1-\frac{1}{m_j})
                 -(\frac{23}{6}+\frac{5}{3(m_j-1)})c_2}\,)\,
      +\, O(\,t^{\frac{4}{3}(1-\frac{1}{m_j})-\frac{10}{3}c_2}\,)\,
      +\, O(\,t^{1-\frac{1}{m_j}-\frac{8}{3}c_2}\,)$
                              \rule{0em}{1.2em} \\[.6ex] \hline
  Region$_{\,(4)_2}$
   & $O(\,t^{\frac{4}{3}(1-\frac{1}{m_j})+(\frac{7}{6m_j}-2)c_1
             -\frac{1}{6}c_2}\,)\,
      +\, O(\,t^{-\frac{2}{3}-\frac{1}{6m_j}-\frac{1}{6}c_2}\,)\,$
                              \rule{0em}{2em} \\[.6ex]
   & $+\, O(\,t^{\frac{7}{6}(1-\frac{1}{m_j})
                 -(\frac{23}{6}+\frac{5}{3(m_j-1)})c_2}\,)\,
      +\, O(\,t^{(1-\frac{1}{m_j})(\frac{1}{m_j}-\frac{2}{3})
                 -(\frac{1}{m_j}-\frac{5}{3})c_2}\,)\,
      +\, O(\,t^{1-\frac{1}{m_j}-\frac{8}{3}c_2}\,)$
                              \rule{0em}{1.2em} \\[.6ex] \hline
  Region$_{\,(5)}$
   & $O(\,t^{\frac{4}{3}(1-\frac{1}{m_j})+(\frac{7}{6m_j}-2)c_1
             -\frac{1}{6}c_2}\,)\,
      +\, O(\,t^{-\frac{2}{3}-\frac{1}{6m_j}-\frac{1}{6}c_2}\,)\,$
                              \rule{0em}{2em} \\[.6ex]
   & $+\, O(\,t^{(1-\frac{1}{m_j})(\frac{7}{6m_j}-\frac{2}{3})
                 -(\frac{7}{6m_j}-\frac{11}{6})c_2}\,)\,
      +\, O(\,t^{(1-\frac{1}{m_j})-\frac{2(5m_j-3)}{3(m_j-1)}}\,)\,
      +\, O(\,t^{1-\frac{1}{m_j}-\frac{8}{3}c_2}\,)$
                              \rule{0em}{1.2em} \\[.6ex] \hline
  Region$_{\,(6)_1}$
   & $O(\,t^{\frac{3}{2}-\frac{4}{3m_j}-\frac{7}{2}c_2}\,)\,
      +\, O(\,t^{\frac{4}{3}(1-\frac{1}{m_j})-\frac{10}{3}c_2}\,)$
                              \rule{0em}{2em} \\[.6ex] \hline
  Region$_{\,(6)_2}$
   & $O(\,t^{\frac{4}{3}(1-\frac{1}{m_j})
             +(\frac{1}{m_j}-\frac{11}{6})c_1-\frac{1}{2}c_2}\,)\,
      +\, O(\,t^{\frac{3}{2}-\frac{4}{3m_j}-\frac{7}{2}c_2}\,)\,
      +\, O(\,t^{\frac{4}{3}(1-\frac{1}{m_j})-\frac{10}{3}c_2}\,)$
                              \rule{0em}{2em} \\[.6ex] \hline
  Region$_{\,(6)_3}$
   & $O(\,t^{\frac{4}{3}(1-\frac{1}{m_j})
             +(\frac{1}{m_j}-\frac{11}{6})c_1-\frac{1}{2}c_2}\,)\,
      +\, O(\,t^{-\frac{1}{2}-\frac{1}{3m_j}-\frac{1}{2}c_2}\,)\,
      +\, O(\,t^{\frac{4}{3}(1-\frac{1}{m_j})-\frac{10}{3}c_2}\,)$
                              \rule{0em}{2em} \\[.6ex] \hline
  Region$_{\,(7)_1}$
   & $O(\,t^{\frac{4}{3}(1-\frac{1}{m_j})-\frac{10}{3}c_2}\,)$
                              \rule{0em}{2em} \\[.6ex] \hline
  Region$_{\,(7)_2}$
   & $O(\,t^{\frac{4}{3}(1-\frac{1}{m_j})-\frac{10}{3}c_2}\,)$
                              \rule{0em}{2em} \\[.6ex] \hline
  Region$_{\,(7)_3}$
   & $O(\,t^{\frac{4}{3}(1-\frac{1}{m_j})
             +(\frac{1}{m_j}-\frac{11}{6})c_1-\frac{1}{2}c_2}\,)\,
      +\, O(\,t^{\frac{4}{3}(1-\frac{1}{m_j})-\frac{10}{3}c_2}\,)$
                              \rule{0em}{2em} \\[.6ex] \hline
  Region$_{\,(8)_1}$
   & $O(\,t^{\frac{4}{3}(1-\frac{1}{m_j})-\frac{7}{2}c_2}\,)\,
      +\, O(\,t^{1-\frac{1}{m_j}-\frac{8}{3}c_2}\,)$
                              \rule{0em}{2em} \\[.6ex] \hline
  Region$_{\,(8)_2}$
   & $O(\,t^{\frac{7}{6}(1-\frac{1}{m_j})
             +\frac{9-35m_j}{6(2m_j-1)}c_2}\,)\,
      +\, O(\,t^{\frac{4}{3}(1-\frac{1}{m_j})-\frac{10}{3}c_2}\,)
      +\, O(\,t^{1-\frac{1}{m_j}-\frac{8}{3}c_2}\,)$
                              \rule{0em}{2em} \\[.6ex] \hline
  Region$_{\,(8)_3}$
   & $O(\,t^{\frac{7}{6}(1-\frac{1}{m_j})
             +\frac{9-35m_j}{6(2m_j-1)}c_2}\,)\,
      +\, O(\,t^{(1-\frac{1}{m_j})(\frac{1}{m_j}-\frac{2}{3})
                 +(\frac{1}{m_j}-\frac{5}{3})c_2}\,)\,
      +\, O(\,t^{1-\frac{1}{m_j}-\frac{8}{3}c_2}\,)$
                              \rule{0em}{2em} \\[.6ex] \hline
    \end{tabular}
      }  

    \centerline{\small\it
    \begin{tabular}{|l|c|} \hline
  Region$_{\,(9)_1}$
   & $O(\,t^{\frac{7}{6}(1-\frac{1}{m_j})
             +(\frac{1}{m_j}-\frac{11}{6})c_1-\frac{1}{6}c_2}\,)\,
      +\, O(\,t^{\frac{7}{6}(1-\frac{1}{m_j})
                 -(\frac{23}{6}+\frac{5}{3(m_j-1)})c_2}\,)\,$
                              \rule{0em}{2.2em} \\[.6ex]
   & $+\, O(\,t^{\frac{4}{3}(1-\frac{1}{m_j})-\frac{10}{3}c_2}\,)\,
      +\, O(\,t^{1-\frac{1}{m_j}-\frac{8}{3}c_2}\,)$
                              \rule{0em}{1.2em} \\[.6ex] \hline
  Region$_{\,(9)_2}$
   & $O(\,t^{\frac{7}{6}(1-\frac{1}{m_j})
             +(\frac{1}{m_j}-\frac{11}{6})c_1-\frac{1}{6}c_2}\,)\,
      +\, O(\,t^{\frac{7}{6}(1-\frac{1}{m_j})
                 -(\frac{23}{6}+\frac{5}{3(m_j-1)})c_2}\,)\,$
                              \rule{0em}{2em} \\[.6ex]
   & $+\, O(\,t^{(1-\frac{1}{m_j})(\frac{1}{m_j}-\frac{2}{3})
                  +(\frac{1}{m_j}-\frac{5}{3})c_2}\,)\,
      +\, O(\,t^{1-\frac{1}{m_j}-\frac{8}{3}c_2}\,)$
                              \rule{0em}{1.2em} \\[.6ex] \hline
  Region$_{\,(10)}$
   & $O(\,t^{\frac{7}{6}(1-\frac{1}{m_j})
             +(\frac{1}{m_j}-\frac{11}{6})c_1-\frac{1}{6}c_2}\,)\,
      +\, O(\,t^{(1-\frac{1}{m_j})(\frac{1}{m_j}-\frac{2}{3})
                 +(\frac{1}{m_j}-\frac{5}{3})c_2}\,)\,$
                              \rule{0em}{2em} \\[.6ex]
   & $+\, O(\,t^{(1-\frac{1}{m_j})-\frac{2(5m_j-3)}{3(m_j-1)c_2}}\,)\,
      +\, O(\,t^{1-\frac{1}{m_j}-\frac{8}{3}c_2}\,)$
                              \rule{0em}{1.2em} \\[.6ex] \hline
  Region$_{\,(11)_1}$
   & $O(\,t^{1-\frac{1}{m_j}-\frac{8}{3}c_2}\,)$
                              \rule{0em}{2em} \\[.6ex] \hline
  Region$_{\,(11)_2}$
   & $O(\,t^{1-\frac{1}{m_j}-\frac{8}{3}c_2}\,)$
                              \rule{0em}{2em} \\[.6ex] \hline
  Region$_{\,(11)_3}$
   & $O(\,t^{(1-\frac{1}{m_j})+(\frac{1}{m_j}-\frac{5}{3})c_1}\,)\,
      +\,O(\,t^{1-\frac{1}{m_j}-\frac{8}{3}c_2}\,)$
                              \rule{0em}{2em} \\[.6ex] \hline
  Region$_{\,(12)_1}$
   & $O(\,t^{1-\frac{1}{m_j}-\frac{8}{3}c_2}\,)$
                              \rule{0em}{2em} \\[.6ex] \hline
  Region$_{\,(12)_2}$
   & $O(t^{(1-\frac{1}{m_j})+(\frac{1}{m_j}-\frac{5}{3})c_1})\,
      +\, O(\,t^{1-\frac{1}{m_j}-\frac{8}{3}c_2}\,)$
                              \rule{0em}{2em} \\[.6ex] \hline
  Region$_{\,(13)}$
   & $O(\,t^{(1-\frac{1}{m_j})+(\frac{1}{m_j}-\frac{5}{3})c_1}\,)\,
      +\, O(\,t^{(1-\frac{1}{m_j})-\frac{2(5m_j-3)}{3(m_j-1)}}\,)\,
      +\, O(\,t^{1-\frac{1}{m_j}-\frac{8}{3}c_2}\,)$
                              \rule{0em}{2em} \\[.6ex] \hline
 \end{tabular}
 } 

 %

 \bigskip

 \noindent
 $(c)$ {\it Overall estimates on $X$.} \hspace{1em}
  Adding the bound on $P_j^t$ and the bound on $Q_j^t$,
   $j=1,\,\ldots\,,\,\tilde{n}_0$ gives a bound of Sobolev norms
   on $X$.
  The proposition follows from the explicit expressions
   in Part (a) and Part (b).

\end{proof}

\begin{remark}
{\it $[$further refinement$]$.} {\rm
 Once having the explicit exponents in the $t$-powers,
  one can further refine each region in the $(c_1,c_2)$-plane
   in the proof
  so that in the end there is only one dominating $t$-power
   in each final region.
 As the details are straightforward but very tedious and
  there are issues one cannot bypass due to topological reasons
   (cf.~Sec.~4),
  we omit the discussion of such further refinement.
}\end{remark}

%
%
%
%
%
%
%
%
%
%

\bigskip

\subsection{Lagrangian neighborhoods and bounds on
            $R(g^t)$, $\delta(g^t)$.}

Recall
 Definition~3.1.4,
 the Lagrangian neighborhood $\Phi_{Z_0}:U_{Z_0}\rightarrow Y$
   of $Z_0$ in $Y$, and
 the symplectic covering map
 $f^{\diamond}_!:
   T^{\ast}X^{\diamond}\rightarrow T^{\ast}Z_0^{\diamond}$ in Sec.~3.1.
In this subsection, we address
 how tractable it is
  to construct an immersed Lagrangian neighborhood for $f^t$
  by gluing an admissible Lagrangian neighborhood
   $\Phi_{\nu_X(\tilde{\Gamma}_j)}^{j,a_j,t}:
         U_{\nu_X(\tilde{\Gamma}_j)}^{a_j,t} \rightarrow Y^{\prime,j}$
   of $L^{\prime, a_jt^{m_j-1}}$ under $\psi^{j,a_jt^{m_j-1}}$,
   constructed in Proposition~2.4.1,
   to $\Phi_{Z_0}\circ f^{\diamond}_!$
    on a neighborhood of the zero-section of $T^{\ast}X^{\diamond}$
 so that the immersed Lagrangian neighborhood of $f^t$
  can fit into Joyce's criteria (iii), (iv), and (v).

\begin{definition-lemma}
{\bf [admissible immersed Lagrangian neighborhood for $f^{t,\diamond}\,$].}
{\rm
 Let
  $f^{t,\diamond}:= f^t|_{X^{\diamond}}$.
 Define
  $f^{\diamond,+,t}_!:
     T^{\ast}X^{\diamond}\rightarrow T^{\ast}Z_0^{\diamond}$
  to be the symplectic covering map
   $f^{\diamond,+,t}_!(\,\cdot\,)= f^{\diamond}_!(\:\cdot\,+dh^t)$
   and
 recall that
  $\Phi_{Z_0}\circ f^{\diamond,+,t}|_{\,\mbox{zero-section}}
   = f^{t,\diamond}$
  from Definition~3.1.4.
 {\it Assume that $b_2^t=Ct^{1-\eta}$
             for some constants $C>0$ and $0<\eta<1$.}
 {\it Then
  there exists an $\epsilon_0>0$ such that
   for $\delta>0$ small enough and $t\in (0,\delta)$,
   the restriction $\Phi_{U_{X^{\diamond}}^t}$ of
    $f^{\diamond,+,t}_!$ to the following neighborhood
    $$
     U_{X^{\diamond}}^t\; :=\;
      \{(x,v)\in T^{\ast}X^{\diamond}\,:\,|v|_{g^t}<t\,\epsilon_0\}
    $$
   of the zero-section of $T^{\ast}X^{\diamond}$
   is a symplectic immersion into $U_{Z_0}$}.
}\end{definition-lemma}

\begin{proof}
 This follows from
  Proposition~2.4.4,
  the quasi-isometry nature of
   $\coprod_j\Upsilon^j:Y^{\prime,j}\rightarrow Y$, and
  the following estimates of
   $1\left/
     \sqrt{\left( 1\,+\,
      \left(
       \frac{d}{dr}
        \left( a^{1/m_j}\,t^{(m_j-1)/m_j}\,
            \left[\frac{m_j}{m_j+1}\,\dot{\chi}^t(r)\,
                  +\,\chi^t(r)\,r^{1/m_j}\right] \right)
      \right)^2 \right)} \right.$
   over $tR^{\prime}_0\le r\le b_2^t$, cf.\ Remark~3.1.6.

 First, note that
  \begin{eqnarray*}
   \lefteqn{\left|
    \frac{d}{dr}
      \left(
       a^{1/m_j}\,t^{(m_j-1)/m_j}\,
         \left[\frac{m_j}{m_j+1}\,\dot{\chi}^t(r)\,
              +\,\chi^t(r)\,r^{1/m_j}\right]
                \right) \right|}                                \\[1.2ex]
    && =\;
     a^{1/m_j}\,t^{(m_j-1)/m_j}\,
     \left|
      \frac{m_j}{m_j+1}\,\ddot{\chi}^t(r)\,
      +\, \dot{\chi}^t(r)\,r^{1/m_j}\,
      +\, \frac{1}{m_j}\,\chi^t(r)\,r^{(1-m_j)/m_j}  \right|    \\[1.2ex]
    && \le\;
     a^{1/m_j}\,t^{(m_j-1)/m_j}\,
     \left(\frac{m_j}{m_j+1}\cdot\frac{C_0}{(b_2^t)^2}\,
           +\, \frac{C_0}{b_2^t}\,r^{1/m_j}\,
           +\, \frac{1}{m_j}\,r^{(1-m_j)/m_j} \right)           \\[1.2ex]
    && \le\;
     O(t^{(m_j-1)/m_j})\,(b_2^t)^{-2}\,
     +\, O(t^{(m_j-1)/m_j})\,(b_2^t)^{(1-m_j)/m_j}\,
     +\, O(t^{(m_j-1)/m_j})\,(tR^{\prime}_0)^{(1-m_j)/m_j}      \\[1.2ex]
    && =\;
     O(t^{(m_j-1)/m_j})\,(b_2^t)^{-2}\,
     +\, O(t^{(m_j-1)/m_j})\,(b_2^t)^{(1-m_j)/m_j}\, +\, O(1)\,.
  \end{eqnarray*}
 Here, we have used the fact that
  $r^{1/m_j}$ (resp.\ $r^{(1-m_j)/m_j}$) is an increasing
   (resp.\ decreasing) function over $r>0$ for all $j$.
 Suppose that $b_2^t=O(t^{1-\eta})$ with $0<\eta<1$.
 Then,
  \begin{eqnarray*}
   \lefteqn{O(t^{(m_j-1)/m_j})\,(b_2^t)^{-2}\,
    +\, O(t^{(m_j-1)/m_j})\,(b_2^t)^{(1-m_j)/m_j}\, +\, O(1) } \\[1.2ex]
   && =\;
    O(t^{-2(1-\eta)})\, +\, O(t^{(1-\frac{1}{m_j})\eta})\, +\, O(1)\;\;
    \le\;\; O(t^{-2})\,.
  \end{eqnarray*}
 Consequently,
  $$
   1\left/
    \sqrt{\left( 1\,+\,
      \left(
       \frac{d}{dr}
        \left( a^{1/m_j}\,t^{(m_j-1)/m_j}\,
            \left[\frac{m_j}{m_j+1}\,\dot{\chi}^t(r)\,
                  +\,\chi^t(r)\,r^{1/m_j}\right] \right)
      \right)^2 \right)} \right.\;\;\;
   \ge\;\;\; O(t)\,.
  $$
 This implies that there exists an $\epsilon_0>0$ such that
  the following disk-bundle of constant radius
   $$
    \{(x,v)\in T^{\ast}X^{\diamond}\,:\,
     \distance_{g^t}(x,\tilde{\Gamma})>tR^{\prime}_0\,,\,
     |v|_{g^t}<\epsilon_0\,t\,\}
   $$
  is immersed into $U_{Z_0}$ by $f^{\diamond,+,t}_!$.
\end{proof}

\begin{definition}
{\bf [immersed Lagrangian neighborhood for $f^t$ via gluing].} {\rm
 Continuing the assumption of $b_2^t$ in Definition/Lemma~3.3.1.
 Recall the constructions and notations in Sec.~2.4 and Sec.~3.1.
 Let $R^{\prime}_0>0$ such that $0<tR^{\prime}_0<b_1^t$.
 Performing the construction of Sec.~2.4 to $\psi^{j,a}$ for each $j$
  gives an admissible Lagrangian neighborhood
  $\Phi_{\nu_X(\tilde{\Gamma}_j)}^{j,a_j,t}:
        U_{\nu_X(\tilde{\Gamma}_j)}^{a_j,t} \rightarrow Y^{\prime,j}$
  for $L^{\prime, a_jt^{m_j-1}}\subset Y^{\prime,j}$
  under $\psi^{j,a_jt^{m_j-1}}$.
 For $0<r<R_0$, denote
  $$
   \begin{array}{rcl}
    \nu_{T^{\ast}X}(\tilde{\Gamma})_r
     & :=       & \coprod_j\{(x_1,x_2,x_3,p_{x_1}, p_{x_2},p_{x_3})
                        \in U_{\nu_X(\tilde{\Gamma}_j)}^{a_j,t}\,:\,
                         |(x_1+\sqrt{-1}x_2)^{m_j}|<r\}  \\[1.2ex]
     & \subset  &\coprod_jT^{\ast}\nu_X(\tilde{\Gamma}_j)\,.
   \end{array}
  $$
 Then
  $\Phi_{\nu_X(\tilde{\Gamma}_j)}^{j,a_j,t}$ and
   $\Phi_{U_{X^{\diamond}}^t}$ coincide over
   $$
    \begin{array}{l}
     \nu_{T^{\ast}X}(\tilde{\Gamma})_{b_1^t}
     \cap (U_{X^{\diamond}}^t\,-\,
           \overline{\nu_{T^{\ast}X}(\tilde{\Gamma})_{tR^{\prime}_0}})
                                                         \\[1.2ex]
     \hspace{2em}
      =\;\coprod_j\{(x_1,x_2,x_3,p_{x_1}, p_{x_2},p_{x_3})
                    \in U_{\nu_X(\tilde{\Gamma}_j)}^{a_j,t}\,:\,
                tR^{\prime}_0<|(x_1+\sqrt{-1}x_2)^{m_j}|<b_1^t\}
    \end{array}
   $$
  under the built-in inclusions
   $U_{\nu_X(\tilde{\Gamma}_j)}^{a_j,t}\subset T^{\ast}X$ and
   $Y^{\prime,j}\subset T^{\ast}Z_0$.
 It follows that if one takes the following neighborhood
  of the zero-section of $T^{\ast}X$:
  $$
   U_X^t\;=\;
     \nu_{T^{\ast}X}(\tilde{\Gamma})_{b_1^t}
     \cup (U_{X^{\diamond}}^t\,-\,
           \overline{\nu_{T^{\ast}X}(\tilde{\Gamma})_{tR^{\prime}_0}})
   \hspace{2em}\mbox{in $\;T^{\ast}X$}\,,
  $$
 then
  the restrictions
   $\left.\Phi_{\nu_X(\tilde{\Gamma}_j)}^{j,a_j,t}\right|
      _{ \nu_{T^{\ast}X}(\tilde{\Gamma})_{b_1^t}}$ and
   $\left.\Phi_{U_{X^{\diamond}}^t}\right|
      _{U_{X^{\diamond}}^t\,-\,
        \overline{\nu_{T^{\ast}X}(\tilde{\Gamma})_{tR^{\prime}_0}}}\,$
   glue to a symplectic immersion
  $$
   \Phi_X^t\,:\, U_X^t\;\longrightarrow\; Y
  $$
  whose restriction to the zero-section is $f^t$.
 $\Phi_X^t$ defines thus an immersed Lagrangian neighborhood
  for the immersed Lagrangian submanifold $f^t:X\rightarrow Y$.
}\end{definition}

\begin{theorem}
{\bf [bound for size, injective radius, and curvature].}
 Recall
  Theorem~1.3 in Sec.~1.
 Making $\delta$ smaller if necessary,
 there exist $A_1,\, A_4,\, A_5,\, A_6>0$ such that
  the following hold for all $t\in (0,\delta)$:
  \begin{itemize}
   \item[{\rm (iii)}]
    The subset ${\cal B}_{A_1t}\subset T^{\ast}N^t$ of Definition~1.2
     lies in $U_{N^t}$,  and
    $\|\hat{\nabla}^k\beta^t\|_{C^0}\le A_4t^{-k}$ on ${\cal B}_{A_1t}$
     for $k=0,\,1,\,2,\,3$.

   \item[{\rm (iv)}]
    The injectivity radius $\delta(g^t)$ satisfies
     $\delta(h^t)\ge A_5t$.

   \item[{\rm (v)}]
    The Riemann curvature $R(g^t)$ satisfies
     $\|R(g^t)\|_{C^0}\le A_6t^{-2}$.
  \end{itemize}
\end{theorem}

\begin{proof}
 Except the necessary mild changes of details that are akin to
  particular situations,
 the proof of [Jo3: III.\ Theorem~6.8] applies here.
 In summary,
  note that the region in $X$ where $f^t$ and $f$ differ has
  the image contained in $\coprod_jY^{\prime,j}$ for all $t$.
 Outside this region, $f^t$ is independent of $t$ and, hence,
  all the criteria in the theorem are satisfied.
 The theorem thus is a joint consequence of the following items:
  \begin{itemize}
   \item[(1)]
    Proposition~2.4.4,
     which leads to the size lower bound $A_1t$ in Criterion (iii);

   \item[(2)]
    Proposition~2.4.6,
     which, together with Parts (3) and (5),
     leads to the bounds in Criterion (iii);

   \item[(3)]
    the estimate of
     $\left|\frac{d}{dr}
        \left(
            a^{1/m_j}\,t^{(m_j-1)/m_j}\,
             \left[\frac{m_j}{m_j+1}\,\dot{\chi}^t(r)\,
                  +\,\chi^t(r)\,r^{1/m_j}\right]
           \right)\right|$
     over $tR^{\prime}_0\le r\le b_2^t$
     and the conclusion in the proof of Definition/Lemma~3.3.1,
     which is needed to control the behavior of $\hat{\nabla}$, and
        together with Parts (3) and (5), justifies Criteria (iv) and (v);

   \item[(4)]
    the quasi-isometry nature of
     $\coprod_j\Upsilon^j:Y^{\prime,j}\rightarrow Y$,
     which is needed to relate the local study in the flat tangent space
     to that in $Y$.
  \end{itemize}

\end{proof}

\bigskip

\subsection{Sobolev immersion inequalities on $N^t$.}

In this subsection we discuss the (in)validity of the following statement
 in our situation:

\begin{statement}
{\bf [Sobolev immersion inequality].}
{\rm  (Cf.~[Jo3: III.~Theorem~6.12].)}
 Making $\delta>0$ smaller if necessary,
 there exists $A_7>0$ such that for all $t\in (0,\delta)$,
  if $v\in L^2_1(N^t)$ with $\int_{N^t}v\,dV^t=0$
  then $v\in L^6(N^t)$ and $\|v\|_{L^6}\le A_7\|dv\|_{L^2}$.
\end{statement}

\noindent
The investigation (and notations whenever possible)
 follows\footnote{Details
                  that are the same as in [Jo3: III~Sec.~6.4] and
                  not needed in the discussion are referred to ibidem
                  and, hence, omitted.}
 the setting of Joyce in [Jo3: III.\ Sec.~6.4].
See also
  [Bu1: Sec.~4.6, Sec.~5.3] of Adrian Butscher,
  [Lee: Sec.~3] of Yng-Ing Lee , and
  [Sa1: Sec.~4] of Sema Salur
 for related studies.

\bigskip

\begin{flushleft}
{\bf Preparation.}
\end{flushleft}
Some basic Sobolev inequalities are given here.

\begin{theorem}
{\bf [Michael-Simon inequality].}
{\rm ([M-S] and [Jo3: III.~Theorem~6.9].)}
 Assume that $m\ge 3$.
 Let
  $f:S\rightarrow {\Bbb R}^l$
   be an immersed $m$-submanifold of ${\Bbb R}^l$ and
  $u\in C^1_{cs}(S)$.
 Then $\|u\|_{L^{2m/(m-2)}}\le D_1(\|u\|_{L^2}+\|uH\|_{L^2})$,
  where
   $D_1>0$ depends only on $m$, and
   $H$ is the mean curvature vector of $S$ in ${\Bbb R}^l$ along $f$.
\end{theorem}

The special Lagrangian submanifold $t\cdot L^{\prime,a}$ is minimal
 (i.e.\ $H$=0) in $Y^{\prime}$. Thus:

\begin{corollary}
{\bf [Sobolev inequality].} {\rm (Cf.\ [Jo3: III.~Corollary~6.10].)}
 There exists $D_1>0$ such that
  $\|u\|_6\le D_1\|du\|_{L^2}$
  for $t>0$ and $u\in C^1_{cs}(t\cdot L^{\prime,a})$.
\end{corollary}

\begin{proposition}
{\bf [Sobolev inequality].} {\rm (Cf.\ [Jo3: III~Proposition~6.11].)}
 There exists $D_2>0$ such that for all $v\in C^1_{cs}(X^{\diamond})$,
  $$
   \|v\|_{L^6}\; \le\;
    D_2 \left(\|dv\|_{L^2}\,
         +\,\left|\int_{X^{\diamond}}v\,dV_g\right| \right)\,.
  $$
\end{proposition}

\begin{proof}
 Note that since $X$ is connected, so is $X^{\diamond}$.
 The proof of [Jo3: III~Proposition~6.11] then goes through,
  word for word, with only
  \begin{itemize}
   \item[$\cdot$]
    `embedding' replaced by `immersion';

   \item[$\cdot$]
   [Jo3: I.~Theorem~2.17] replaced by
    Lemma~2.1.2 and Example~2.1.3 of the current notes
    to take care of the counter situation of branched coverings
    under study.
  \end{itemize}

\end{proof}

\bigskip

\begin{flushleft}
{\bf (Non)existence of $A_7$ in Statement~3.4.1.}
\end{flushleft}
Recall the parameters
 $tR^{\prime}_0<b_1^t<b_2^t(=Ct^{1-\eta})<R_0\le 1$,
 which set the range of gluings in the gluing construction in Sec.~3.1
 from the aspect of target Calabi-Yau $3$-fold $Y$.
$t\in (0,\delta)$ with $0<\delta<1$ small enough.
As $f^t:X\rightarrow Y$ is an immersion
  with $X$ carrying the pull-back metric,
 we'll directly treat $X$ as a submanifold in $Y$
  in the following computations and estimates
  whenever this is more convenient.

Let $a,b\in {\Bbb R}$ with $0<a<b<1-\eta$.
Then for $\delta_0$ small enough, $2b_2^t<t^b<t^a<R_0$
 holds for all $t\in (0,\delta)$.
Let $G:(0,\infty)\rightarrow [0,1]$ be a smooth decreasing function
 with $G(s)=1$ for $s\in (0,a]$ and $G(s)=0$ for $s\in [b,\infty)$.
Write $\dot{G}$ for $dG/ds$.
Recall
 the coordinates $(x_1,x_2,x_3)$ on $\nu_X(\tilde{\Gamma}_j)$ and
 the coordinates $(u_1,u_2,u_3)$ on $\nu_{Z_0}^j(\Gamma_i)$
 in Definition~3.1.1.
For $x\in \nu_X(\tilde{\Gamma}_j)$
      for some $j\in \{1,\,\cdots,\, \tilde{n}_0\}$,
 denote $r:=|(x_1+\sqrt{-1}x_2)^{m_j}|$.
For $t\in(0,\delta)$, define a function $F^t:N^t=X\rightarrow [0,1]$
 by
 $$
  F^t(x)\;=\;  \left\{
   \begin{array}{lcl}
    0\,,
     && \mbox{for $x\in \nu_X(\tilde{\Gamma}_j)$,
         $\,j\in\{1,\,\cdots,\,\tilde{n}_0\}$,
         $\;0\,\le\,r\,\le\, t^b$}\,,                      \\[1.2ex]
    G((\log r)/(\log t))\,,
     && \mbox{for $x\in \nu_X(\tilde{\Gamma}_j)$,
                  $\,j\in\{1,\,\cdots,\,\tilde{n}_0\}$,
         $\;t^b\,\le\,r\,\le\, t^a$}\,,                    \\[1.2ex]
    1\,,
     && \mbox{for $x$ at elsewhere in $X$}\,.
   \end{array}
               \right.
 $$
Then, $F^t$ is a smooth function on $N^t$ with
 $$
  dF^t\; =\;   \left\{
   \begin{array}{lcl}
    (\log t)^{-1}\,\dot{G}((\log r)/(\log t))\,r^{-1}\,dr
     && \mbox{on
         $\,\coprod_{j=1}^{\tilde{n}_0}\{t^b\,\le\,r\,\le\, t^a\}\,
          \subset\, \coprod_{j=1}^{\tilde{n}_0}
                              \nu_X(\tilde{\Gamma}_j)$}\,, \\[1.2ex]
    0   && \mbox{elsewhere}\,.
   \end{array}
               \right.
 $$
$\{F^t\,,\,1-F^t\}$ gives thus a two-component partition of unity
 on $N^t$.
We'll use it to decompose a function $v\in C^1(N^t)$
  with $\int_{N^t}v\,dV^t=0$ into $v=F^tv+(1-F^t)v$;  and
 then treat
  $F^tv$ as a compactly-supported function on $X^{\diamond}$
   and apply Proposition~3.4.4 to it,  and
  $(1-F^t)v$ as a compactly supported function on
   $\coprod_{j=1}^{\tilde{n}_0}
             (t\cdot L^{a_j})
    \subset \coprod_{j=1}^{\tilde{n}_0} Y^{\prime,j}$ and
   apply Corollary~3.4.3 to each component of it.

For the first part,
 since $f^t$ and $f$ coincide on
   $X-\coprod_{j=1}^{\tilde{n}_0}\{0\le r\le b_2^t\}$  and
   $b_2^t<t^b$,
 $F^tv$ is naturally in $C^1_{cs}(X^{\diamond})$ and
 $g^t=g$ on the support of $F^tv$.
Proposition~3.4.4 plus a H\"{o}lder's inequality gives then
 %
 \begin{eqnarray*}
  \lefteqn{\|F^tv\|_{L^6}\;
   \le\; D_2\,
     \left(\|d(F^tv)\|_{L^2}\,+\, \left|\int_{N^t} F^tv\,dV^t\right|\,
     \right)} \\[1.2ex]
   && \le\; D_2\,
     \left( \|F^t\,dv\|_{L^2}\,
            +\, \|v\|_{L^6}\cdot
                  \left(
                   \|dF^t\|_{L^3}\,+\,\|1-F^t\|_{L^{6/5}}
                  \right)
     \right)\,.
 \end{eqnarray*}

For the second part, we prove first a lemma:

\begin{lemma}
{\bf [quasi-isometry with uniform bound on dilatation].}
 Assume that
  $b_1^t=C_1t^{1-\eta_1}$ and $b_2^t=C_2t^{1-\eta_2}$,
  where $0<\eta_1<\eta_2<1-a<1$.
 Assume further that
  $\eta_2> \max_j
    \left\{\frac{1}{2}(1+\frac{1}{m_j})\,,\,\frac{1-a}{m_j}\right\}$.
 Then, for $\delta>0$ small enough and all $t\in (0,\delta)$,
  the map
  $$
   f^t(x_1,x_2,x_3)\;  \longmapsto\;
    ((x_1+\sqrt{-1}x_2)^{m_j},\,x_3,\,
      a^{1/m_j}\,t^{(m_j-1)/m_j}\,(x_1+\sqrt{-1}x_2),0)
  $$
  gives a diffeomorphism
  that takes $f^t(\{|x_1+\sqrt{-1}x_2|^{m_j}\le t^a\})$
    to $\left.L^{\prime,a_jt^{m_j-1}}\right|
                            _{|u_1+\sqrt{-1}u_2|\le t^a}$
   quasi-isometrically with the dilatation uniformly bounded
   both from above and from below (away from zero) $t$-independently.
\end{lemma}

\begin{proof}
 Recall Remark~3.1.6 and the proof of Definition/Lemma~3.3.1.
 Under the assumption that
   $b_1^t=C_1t^{1-\eta_1}$ and $b_2^t=C_2t^{1-\eta_2}$,
   where $0<\eta_1<\eta_2<1-a<1$,
 one has the following estimates over
  $(tR^{\prime}_0<)\, b_1^t\le r \le t^a\,$:
  \begin{eqnarray*}
   \lefteqn{\left|
    \frac{d}{dr}
      \left(
       a^{1/m_j}\,t^{(m_j-1)/m_j}\,
         \left[\frac{m_j}{m_j+1}\,\dot{\chi}^t(r)\,
              +\,\chi^t(r)\,r^{1/m_j}\right]
                \right) \right|}                                \\[1.2ex]
    && =\;
     a^{1/m_j}\,t^{(m_j-1)/m_j}\,
     \left|
      \frac{m_j}{m_j+1}\,\ddot{\chi}^t(r)\,
      +\, \dot{\chi}^t(r)\,r^{1/m_j}\,
      +\, \frac{1}{m_j}\,\chi^t(r)\,r^{(1-m_j)/m_j}  \right|    \\[1.2ex]
    && \le\;
     a^{1/m_j}\,t^{(m_j-1)/m_j}\,
     \left(\frac{m_j}{m_j+1}\cdot\frac{C_0}{(b_2^t)^2}\,
           +\, \frac{C_0}{b_2^t}\,r^{1/m_j}\,
           +\, \frac{1}{m_j}\,r^{(1-m_j)/m_j} \right)           \\[1.2ex]
    && \le\;
     O(\,t^{2\eta_2-1-\frac{1}{m_j}}\,)\,
     +\, O(\,t^{\eta_2-\frac{1-a}{m_j}}\,)\,
     +\, O(\,t^{(1-\frac{1}{m_j})\eta_1}\,)\,.
  \end{eqnarray*}
 The further assumption that
   $\eta_2>\max_j
    \left\{\frac{1}{2}(1+\frac{1}{m_j})\,,\,\frac{1-a}{m_j}\right\}$
  implies that
   the right-hand side of the last inequality tends to zero
   as $t\rightarrow 0$.
 On the other hand,
  $$
   \left| \frac{d}{dr}
      \left( a^{1/m_j}\,t^{(m_j-1)/m_j}\,r^{1/m_j} \right) \right|\;
   \le\; O(\,t^{1-\frac{1-a}{m_j}}\,)
  $$
  over $b_1^t\le r \le t^a$,
  which also tends to $0$ as $t\rightarrow 0$.
 The lemma follows.

\end{proof}

Applying Corollary~3.4.3 to $(1-F^t)v$ on $t\cdot L^{\prime,a}$
 gives $\|(1-F^t)v\|_{L^6}\le D_1\,\|d((1-F^t)v)\|_{L^2}$ with
 the norms computed using $g^{\prime}$ on $t\cdot L^{\prime,a}$.
Increasing $D_1$ to $2D_1$, the same inequality holds
 with norms computed using $g^t$ for small $t$.
Thus, making $\delta>0$ smaller if necessary,
 for all $t\in(0,\delta)$ and
  each connected component of $(1-F^t)v$ on $N^t$
   (labelled by $j=1,\,\cdots,\,\tilde{n}_0$),
 $$
  \|(1-F^t)v\|_{L^6}\; \le\;  2\,D_1\,\|d((1-F^t)v)\|_{L^2}\,.
 $$
Summing over $j$ gives
 \begin{eqnarray*}
  \lefteqn{\|(1-F^t)v\|_{L^6}\;
     \le\; 2\,\sqrt{\tilde{n}_0}\,D_1\,\|d((1-F^t)v)\|_{L^2}    }\\[1.2ex]
  && \le\; 2\,\sqrt{\tilde{n}_0}\,D_1\, \left(
       \|(1-F^t)\,dv\|_{L^2}\,+\,\|v\|_{L^6}\cdot\|dF^t\|_{L^3}
                                \right)\,.
 \end{eqnarray*}
Combining the above inequality with
  the previous inequality for $\|F^tv\|_{L^6}$,
  the inequalities
   $\|F^tdv\|_{L^2},\, \|(1-F^t)dv\|_{L^2}\le \|dv\|_{L^2}$, and
  a H\"{o}lder's inequality
 then proves
 $$
  \left[1\,-\,(D_2+2\sqrt{\tilde{n}_0}D_1)\,\|dF^t\|_{L^3}\,
           -\,D_2\,\|1-F^t\|_{L^{6/5}}
  \right]\,\cdot\,\|v\|_{L^6}\;
  \le\; (D_2+2\,\sqrt{\tilde{n}_0}D_1)\,\|dv\|_{L^2}\,.
 $$
It follows from Lemma~3.4.5 that
 one can compute the $t$-order of $\|dF^t\|_{L^3}$ and
  the $t$-order of $\|1-F^t\|_{L^{6/5}}$ on $(N^t,g^t)$
 via the $t$-order of the corresponding same quantity on $L^{\prime}$
  and $t\cdot L^{\prime,a}$ in
  $\coprod_{j=1}^{\tilde{n}_0}(Y^{\prime,j},g^{\prime})$
  respectively.

On $L^{\prime}\subset Y^{\prime,j}$,
 \begin{eqnarray*}
  \|dF^t\|_{L^3}  & =
   & \left(
      \int_{u_3=0}^{l_j}\int_{\theta=0}^{2m_j\pi}\int_{r=t^b}^{t^a}\,
      \left|(\log t)^{-1}\,\dot{G}((\log r)/(\log t))\,r^{-1}\right|^3\,
                   r\,dr\,d\theta\,du_3 \right)^{1/3}   \\[1.2ex]
  & =
   & l_j^{1/3}\,(2m_j\pi)^{1/3}\,(-\log t)^{-1}\,
        \left( \int_{r=t^b}^{t^a}\,
                 |\dot{G}((\log r)/(\log t))|^3\, r^{-2}\,dr
         \right)^{1/3}\,.
 \end{eqnarray*}
It follows from the facts
  that  $|\dot{G}((\log r)/(\log t))|$ is bounded above
   by a constant independent of $t$ and $r$ and
  that  there exist $(0<)\,a<a^{\prime}<b^{\prime}<b$
   such that the restriction of $|\dot{G}((\log r)/(\log t))|$
   to $[t^{b^{\prime}},t^{a^{\prime}}]$ is bounded below
   by a positive constant independent of $t$ and $r$
 that
 $$
  O(-(\log t)^{-1}\,t^{-b^{\prime}/3})\;
   \le\; \|dF^t\|_{L^3}\; \le  O(-(\log t)^{-1}\,t^{-b/3})\,.
 $$
Since $b, b^{\prime}>0$, it follows that
 $\|dF^t\|_{L^3} \rightarrow \infty$ as $t\longrightarrow 0$.

On $t\cdot L^{\prime,a_j}\subset Y^{\prime,j}$,
 \begin{eqnarray*}
  \|1-F^t\|_{L^{6/5}}  & =
   & \left(
      \int_{u_3=0}^{l_j}\int_{\theta=0}^{2m_j\pi}\int_{r=0}^{t^a}\,
      \left|1-G((\log r)/(\log t))\right|^{6/5}  \right.  \\[1.2ex]
  &&\hspace{6em} \left.\rule{0em}{1.4em} \cdot
      r\,\left( 1\,+\,m_j^{-2}\,a_j^{2/m_j}\,r^{2(1-m_j)/m_j}\right)\,
          dr\,d\theta\,du_3 \right)^{5/6}                 \\[1.2ex]
  & \le
   & l_j^{5/6}\,(2m_j\pi)^{5/6}\,
        \left(\frac{1}{2}\,\left[
                      r^2\,+\, m_j^{-1}\,a_j^{2/m_j}\,r^{2/m_j}
                           \right]_0^{t^a} \right)^{5/6}  \\[1.2ex]
  & =
   & O\left(t^{5a/(3m_j)}\right) \hspace{2em}
      \mbox{------------}\hspace{-1ex}\longrightarrow\;0
      \hspace{2em}\mbox{as $\;t\longrightarrow 0$}\,.
 \end{eqnarray*}

It follows that one cannot extract the Sobolev inequality
 in Statement~3.4.1 that bounds $\|v\|_{L^6}$ by $\|dv\|_{L^2}$
 from the inequality derived earlier:
 $[1\,-\,(D_2+2\sqrt{\tilde{n}_0}D_1)\,\|dF^t\|_{L^3}\,
              -\,D_2\,\|1-F^t\|_{L^{6/5}}]\cdot\|v\|_{L^6}\,
                \le\, (D_2+2\,\sqrt{\tilde{n}_0}D_1)\,\|dv\|_{L^2}$.
The validity of Statement~3.4.1
 is thus left open in our situation through the above method.

%
%
%
%
%
%
%
%
%
%
%
%
%
%
%
%
%
%
%
%
%
%
%
%
%
%
%
%
%
%
%
%
%
%
%
%
%
%
%
%
%
%
%
%
%
%
%
%
%

\bigskip

\section{Summary and remark: Input from the topology of $X$ and
         the branching of $f$.}

Given a branched covering $f:X\rightarrow Z_0\simeq S^3\subset Y$
 a special Lagrangian $3$-sphere $Z_0$ in a Calabi-Yau $3$-fold $Y$,
one expects that the natural family of immersed Lagrangian deformations
 $f^t:X\rightarrow Y$ of $f$,
  $t\in (0,\delta)$ for $0<\delta<1$ small enough,
 constructed in Sec.~3.1
  (1)
   have the Sobolev norms of the variations of the mean curvature
   along $f^t$ too large as $t\rightarrow 0$, (cf.\ Sec.~3.2),
  (2)
   cannot pass directly a Sobolev inequality test that allows one
   to conclude the existence of a $t$-independent uniform positive
   lower bound for the first eigenvalue of the Laplacian $\Delta_{g^t}$
   of $(X,g^t)$, where $g^t:= (f^t)^{\ast}g$ is the pull-back metric
   on $X$, (cf.\ Sec.~3.4),
 though
  (3)
   $f^t$ still meets the criteria on the size and the geometry
     in a Lagrangian neighborhood of $f^t$ in $Y$ that demand $f^t$
    to approach the singular $(X,g^0)$ not too fast,
     as $t\rightarrow 0$, and
    to have enough room in $Y$ to deform $f^t$ further,
   (cf.\ Sec.~3.3).
Issue (1) suggests that one has to glue the approximately special
 Lagrangian local models $a\cdot L^{\prime,a}$ not to $f$,
 but some Lagrangian perturbation of $f$ that matches
 the flaring-out rate of $tL^{\prime,a}$ better.
Issue (2) suggests that one has to identify the collection of eigenvalues
 of $(X,g^t)$ that approach $0$ as $t\rightarrow 0$ and then consider
 deformations of $f$ that are only in the complementary directions
 to those that correspond to the eigenfunction associated to these
 eigenvalues.
Item (3) suggests that the gluing with a scaling construction,
 as done in [Jo3] of Joyce, is a versatile construction,
 which one would like to keep while remedying Issues (1) and (2).

While the detail of the deviation requires specific computations
 as is done in Sec.~3,
that the immersed Lagrangian deformation constructed in this note
  will deviate from being deformable to
  an immersed special Lagrangian deformation,
  no matter what adjustments one attempts,
 is anticipated even before one gets into such details.
This is because there are examples of branched coverings of $S^3$
 by $S^3$ and the construction in this note apply to them as well.
Should there be no deviations, one would have constructed
 a nontrivial family of immersed special Lagrangian $S^3$'s,
 which contradicts (the immersed generalization of) a result of
 Robert McLean ([McL: Theorem~3.6]) which implies, in particular, that
 an immersed special Lagrangian 3-sphere is rigid.
Thus, if such a deformation is possible,
 the topology of the domain $X$ of the branched covering
 $f: X\rightarrow S^3$ of a special Lagrangian $S^3$ in question
 must play a significant role in rectifying the deviations created
 in the naive gluing as we see in this note.
The detail of the local model given in Sec.~2.2 combined with
 the original study of deformations of special Lagrangian submanifolds
 in [McL: Sec.~3]
suggests that
 one such necessary condition would be:
 \begin{itemize}
  \item[$\cdot$]
   [{\it infinitesimal condition for unwrapping}$\,$]\hspace{1em}
   There exists a harmonic/closed-'n'-coclosed $1$-form $\alpha$ on $X$
   that has the decay behavior exactly of type $\rho^{m_i}d\rho$
   around the component $\Gamma_i$ of the branch locus $\Gamma\subset X$
   of the branched covering $f:X\rightarrow S^3$
   in a Calabi-Yau $3$-fold $Y$.
 \end{itemize}
This is a condition that involves both the topology of $X$ and
 the branching of the map $f:X\rightarrow S^3$.
Here,
 $m_i$ is the degree of $f$ around $\Gamma_i$,
 $\rho$ is the distance function $\distance(\,\cdot\,,\Gamma)$ on $X$
  associated to the pull-back metric
  (with curvature singularities along $\Gamma$),  and
 the harmonicity/closed-'n'-coclosedness of $\alpha$
  is with respect to the latter singular metric on $X$ as well.
If $\alpha$ behaves of type $\rho^{\lambda_i}d\rho$ along $\Gamma_i$
  with $\lambda_i<m_i$ for some $i$,
 then one has a chance of tearing off $f$ along $\Gamma_i$ when trying
  to deform $f$ as guided by $\alpha$
  in the realm of special Lagrangian immersions.
In the opposite direction,
 if $\lambda_i>m_i$ for some $i$,
 then the corresponding candidate special Lagrangian deformations of $f$
  may remain branched along $\Gamma_i$ and hence fails to be an immersion.

This leads one thus to issues on
 such harmonic/closed-'n'-coclosed $1$-forms and
 the existence/nonexistence of obstructions
 to the above first order picture.

\newpage
\baselineskip 13pt

\end{document}